\documentclass[reqno, psamsfonts, letterpaper]{amsart}
\let\oldtocsection=\tocsection

\let\oldtocsubsection=\tocsubsection

\let\oldtocsubsubsection=\tocsubsubsection

\renewcommand{\tocsection}[2]{\hspace{0em}\oldtocsection{#1}{#2}}
\renewcommand{\tocsubsection}[2]{\hspace{1em}\oldtocsubsection{#1}{#2}}
\renewcommand{\tocsubsubsection}[2]{\hspace{2em}\oldtocsubsubsection{#1}{#2}}

\usepackage{amssymb}
\usepackage[margin=1.3in]{geometry}

\usepackage{tikz-cd}
\tikzcdset{			% small default diagram size
	diagrams={sep=small}
}

\usepackage{mathtools}

\usepackage{hyperref, cleveref}
\hypersetup{
	colorlinks = true,
}

\usepackage{enumitem}
\setlist[1]{labelindent=\parindent}
\setlist[enumerate, 1]{label = \textnormal{(\arabic*)}, ref = \arabic*}
\setlist[enumerate, 2]{label = \textnormal{(\alph*)}, ref = \alph*}
\setlist[enumerate, 3]{label = \textnormal{(\roman*)}, ref = \roman*}

\numberwithin{equation}{section}

\let\emptyset\varnothing

\usepackage{thmtools}

\declaretheorem[sibling=equation, style=definition]{definition}
\declaretheorem[sibling=equation]{theorem}
\declaretheorem[sibling=equation, name=Theorem-Definition]{theorem-definition}
\declaretheorem[sibling=equation]{proposition}
\declaretheorem[sibling=equation]{lemma}
\declaretheorem[sibling=equation]{corollary}

\declaretheorem[sibling=equation, style=remark]{remark}
\declaretheorem[sibling=equation, style=remark]{example}

\DeclareMathOperator{\GL}{GL}

\DeclareMathOperator{\SO}{SO}

\DeclareMathOperator{\Hom}{Hom}

\DeclareMathOperator{\Aut}{Aut}
\DeclareMathOperator{\End}{End}
\DeclareMathOperator{\Mat}{Mat}
\DeclareMathOperator{\id}{id}
\DeclareMathOperator{\im}{im}

\DeclareMathOperator{\Pic}{Pic}

\DeclareMathOperator{\Def}{Def}

\DeclareMathOperator{\tr}{tr}

\DeclareMathOperator{\pr}{pr}
\DeclareMathOperator{\Gal}{Gal}
\DeclareMathOperator{\reg}{reg}
\DeclareMathOperator{\Alb}{Alb}

\DeclareMathOperator{\tor}{tor}
\DeclareMathOperator{\MT}{MT}

\newcommand{\ZZ}{\mathbb{Z}}
\newcommand{\QQ}{\mathbb{Q}}

\newcommand{\CC}{\mathbb{C}}

\newcommand{\PP}{\mathbb{P}}

\newcommand{\HS}{\mathsf{HS}}

\newcommand{\Kum}{\text{Kum}}

\begin{document}
	\title[]{Isotrivial Lagrangian fibrations of compact hyper-kähler manifolds}
	
	\author[Y.-J. Kim]{Yoon-Joo Kim}
	  \address{Department of Mathematics, Columbia University, New York, NY 10027}
 \email{{\tt yk3029@columbia.edu}}

	\author[R. Laza]{Radu Laza}
	 \address{Department of Mathematics, Stony Brook University, New York, NY 11794}
 \email{{\tt radu.laza@stonybrook.edu}}
	\author[O. Martin]{Olivier Martin}
	  \address{Instituto de Matemática Pura e Aplicada, Rio de Janeiro, RJ 22460-320}
 \email{{\tt olivier.martin@impa.br}}

%	\date{\today}

\maketitle

\begin{abstract}
	This article initiates the study of isotrivial Lagrangian fibrations of compact hyper-K\"ahler manifolds. We present four foundational results that extend well-known facts about isotrivial elliptic fibrations of K3 surfaces. First, we prove that smooth fibers of an isotrivial Lagrangian fibration are isogenous to a power of an elliptic curve.  Second, we exhibit a dichotomy between two types of isotrivial Lagrangian fibrations, which we call A and B. Third, we give a classification result for type A isotrivial Lagrangian fibrations. Namely, if a type A isotrivial Lagrangian fibration admits a rational section, then it is birational to one of two straightforward examples of isotrivial fibrations of hyper-K\"ahler manifolds of $\text{K3}^{[n]}$-type and $\text{Kum}_n$-type. Finally, we prove that a genericity assumption on the smooth fiber of an isotrivial Lagrangian fibration ensures that the fibration is of type A.
\end{abstract}

\section{Introduction}

Hyper-K\"ahler manifolds are one of the three building blocks of compact K\"ahler manifolds with trivial canonical bundles. Due to their rich geometry and relation to other objects of interest, they have become the subject of intense study over the past decades. However, only few examples of such manifolds are known, namely two deformation types, $\text{K3}^{[n]}$ and $\textup{Kum}_n$, in each even dimension $2n\geq 4$, and two sporadic deformation types, OG6 and OG10, in dimensions 6 and 10. A central question in the field is whether there are finitely many deformation types in each dimension.\\

One possible approach to this problem is via Lagrangian fibrations, as in \cite{saw16}. Specifically, it is expected that all deformations types of hyper-K\"ahler manifolds admit specializations to Lagrangian-fibered hyper-K\"ahler manifolds, thereby possibly reducing this boundedness problem to the study of families of abelian varieties. Bakker \cite{Bakker}, improving earlier results of van Geemen--Voisin \cite{vgee-voi16}, showed that Lagrangian fibrations are either isotrivial or have maximal variation. Since isotrivial Lagrangian fibrations are relatively tractable they are particularly interesting objects of study in the search for new deformation types. \\

In this article, we establish some foundational results on isotrivial Lagrangian fibrations of projective hyper-K\"ahler manifolds, which we call {\it hyper-K\"ahler manifolds} for simplicity\footnote{The degenerate twistor deformation of \cite{ver15, bog-deev-ver22} can be used to deform an isotrivial Lagrangian fibration into a projective isotrivial Lagrangian fibration. We may thus assume that $X$ is projective without loss of generality.}. We start by recalling the classification of isotrivial elliptic fibrations of projective K3 surfaces $\pi : S \longrightarrow \PP^1$ in terms of the monodromy group $G$ of the local system $R^1 (\pi_0)_* \underline \ZZ$, where $\pi_0:S_0 \longrightarrow \PP^1\setminus \Sigma$ is the smooth isotrivial fibration obtained by restricting $\pi$ to the complement of the singular fibers. The following theorem (see \Cref{K3appendix}) is well-known to experts but we were unable to find a suitable reference (see \cite[\S1.4.2]{FM94}, \cite{sawon14}, and \cite{moonen18}  for related statements):

\begin{theorem}\label{main:K3}
	Let $\pi : S \longrightarrow \PP^1$ be an isotrivial fibration of a projective K3 surface with smooth fiber $F$ and let $G:=\im\left(\pi_1(\mathbb{P}^1\setminus \Sigma) \longrightarrow \GL(H^1(F,\ZZ)\right))$ be its global monodromy group. Then $G$ is a subgroup of the automorphism group of the elliptic curve $F$. In particular, $G$ is either $\mu_2,\mu_3,\mu_4,\text{ or }\mu_6$. Furthermore:
	\begin{enumerate}
		\item If $G = \mu_2$, there is an abelian surface $A$ and a fibration $p: A\longrightarrow E$ to an elliptic curve $E$ with fiber $F$, such that $\pi$ is birational to the quotient of $p$ by the involution $[-1]$. In particular, $S$ is a Kummer surface.
		\item If $\pi$ admits a section, there is a curve $C$ with a $G$-action such that $\pi$ is birational to the quotient of the  projection $F \times C \longrightarrow C$ by a diagonal $G$-action. The possible curves $C$ for $G = \mu_3, \mu_4$, and $\mu_6$ are classified in tables \ref{table:classification for mu_3}, \ref{table:classification for mu_4}, and \ref{table:classification for mu_6} respectively.
	\end{enumerate}
\end{theorem}

Note that every elliptic fibration of a K3 surface is a Tate--Shafarevich twist of an elliptic fibration with a section (see, e.g., \cite[\S 11]{huy:k3}). \Cref{main:K3} thus provides a complete answer to the classification of isotrivial fibrations for K3 surfaces up to Tate--Shafarevich twist.

\begin{remark} \label{j}
	Note that if the elliptic curve $F$ has $j$-invariant different from $0$ or $1728$ then we are automatically in case (1). In case (2) the K3 surface $S/\PP^1$ admits an automorphism of order $m \ge 3$ such that the fibration $S\longrightarrow \PP^1$ is invariant (in particular $m \in \{3,4,6\}$). All examples of fibrations arising in (2) are obtained by considering degree $d$ cyclic covers $C\longrightarrow \PP^1$ branched in appropriate configurations of points (e.g. generically $12$ points for $m=6$) and performing a birational modification of the diagonal quotient
	$ (E\times C)/\mu_m \longrightarrow C/ \mu_m = \mathbb{P}^1$,
	where $E$ is the elliptic curve with $\mu_m$ automorphisms. 
	This construction is closely related to the work of Deligne--Mostow \cite{del-mos} (see also \cite{moonen18}) and, from the perspective of flat metrics on $S^2$ with conifold singularities, to the work of Thurston \cite{thu}. The Hodge theoretic underpinning of such examples is provided by van Geemen's theory of half-twists \cite{vanG}.
\end{remark}

Moving to higher dimension, the goal of this article is to provide an analogous classification of isotrivial Lagrangian fibrations of hyper-K\"ahler manifolds. Along the way we will discover several analogies with the K3 case. The first of these is the surprising fact that smooth fibers of an isotrivial Lagrangian fibration $\pi: X\longrightarrow B$ are isogenous to the power of an elliptic curve. Throughout the article we assume that the base $B$ of the Lagrangian fibration is normal (if $B$ is smooth, then it is isomorphic to $\PP^n$ by \cite{Hwang}).

\begin{theorem} \label{main:fiber}
	Let $\pi: X \longrightarrow B$ be an isotrivial Lagrangian fibration of a hyper-K\"ahler $2n$-fold with smooth fiber $F$. Then there is an elliptic curve $E$ such that $F$ is isogenous to $E^n$.
\end{theorem}

Consider an isotrivial Lagrangian fibration $\pi : X \longrightarrow B$ with general fiber $F$. Let $B_0 = B \setminus \Sigma$ be the complement of the discriminant locus $\Sigma$ and $X_0$ be the base change of $X$ to $B_0$. We write $\pi_0:X_0\longrightarrow B_0$ for the the restriction of $\pi$ to $X_0$, which is a (smooth) isotrivial fibration. As in the K3 case, we consider the monodromy group $G=\im\left(\pi_1(B_0)\longrightarrow \GL(H^1(F,\ZZ)\right)$ and the minimal Galois covering $U \longrightarrow B_0$ that trivializes the local system $R^1(\pi_0)_* \underline \ZZ$. The base change $Z\longrightarrow U$ of $\pi_0$ by $U \longrightarrow B_0$ may not be a trivial family but it plays an important role.
For this reason, we refer to $Z\longrightarrow U$ as the {\it intermediate trivialization}. We will argue that the
geometry of the isotrivial fibration $X/B$ is controlled to a large extent by the Kodaira dimension (of a smooth compactification) of the intermediate base $U$. One of our main results states that this Kodaira dimension is either $0$ or maximal, and we call the fibration $\pi$ of \emph{type A} or \emph{type B} accordingly.

\begin{theorem} \label{main:typeAB}
	Let $\pi: X\longrightarrow B $ be an isotrivial Lagrangian fibration of a hyper-K\"ahler $2n$-fold and $U$ be the base of the intermediate trivialization.
Then either
\begin{enumerate}
	\item[$\textup{(A)}$] $U$ compactifies to an abelian variety isogenous to a power of an elliptic curve, or
	\item[$\textup{(B)}$] $U$ compactifies to a smooth projective variety of general type.
\end{enumerate}
\end{theorem}

For any $n > 1$ there are two straightforward examples of isotrivial Lagrangian fibrations, which we call {\it 
 $\textup{K3}^{[n]}$-fibration and $\textup{Kum}_{n}$-fibration} (see \Cref{S:ex-trivial} for a detailed discussion).

\begin{definition}[$\text{K3}^{[n]}$ and $\text{Kum}_{n}$ fibrations] \label{def:K3n Kumn fibrations}
	An isotrivial Lagrangian fibration of a hyper-K\"ahler $2n$-fold $\pi : X\longrightarrow B$ is a
	\begin{itemize}
	\item[\emph{i})] \emph{$\text{K3}^{[n]}$-fibration} if there is an isotrivial elliptic K3 surface $f: S\longrightarrow \mathbb{P}^1$ and a commutative diagram
	\[
	\begin{tikzcd}[column sep=normal]
	X \ar[r, "\cong", "\phi"'] \ar[d,"\pi"]& S^{[n]} \ar[d,"f^{[n]}"]\\
	B\ar[r,"\cong"]& \PP^n.
	\end{tikzcd}
	\]
	
	\item[\emph{ii})] \emph{$\text{Kum}_{n}$-fibration} if there is an abelian surface $A$, an elliptic curve $E'$, a surjective homomorphism $f: A\longrightarrow E'$, and a commutative diagram  \[
	\begin{tikzcd}[column sep=normal]
	X \ar[r, "\phi"', "\cong"] \ar[d,"\pi"]& \textup{Kum}_n(A) \ar[d,"g"]\\
	B\ar[r,"\cong"]& \mathbb{P}^n,
	\end{tikzcd}
	\]
	where the morphism $g$ is induced by $f$.
	\end{itemize}
	
	We say that $\pi$ is \emph{birational to a $\text{K3}^{[n]}$-fibration} (resp. \emph{birational to a $\text{Kum}_n$-fibration}) if the isomorphism $\phi$ appearing in the above definitions is replaced by a birational map.
\end{definition}

Our next result states that a type A fibration with a rational section is birational to either a $\textup{K3}^{[n]}$-fibration or a $\textup{Kum}_{n}$-fibration. In particular, the base $B$ is isomorphic to $\PP^n$.

\begin{theorem}\label{main:classification}
	Let $\pi: X\longrightarrow B $ be a type A isotrivial Lagrangian fibration of a hyper-K\"ahler $2n$-fold. Assume that $\pi$ has a rational section. Then $\pi$ is birational to either a $\textup{K3}^{[n]}$-fibration or a $\textup{Kum}_n$-fibration.
\end{theorem}

In analogy with \Cref{j} in the K3 case, a certain genericity assumption on the fiber $F$ ensures that $\pi$ is of type A. A fiber of $\pi$ is called \emph{non-multiple} if its cycle class is of the form $\sum_j a_j [Y_j]$ with $\gcd \{ a_j \} = 1$, where $Y_j$ denotes the irreducible components of the support of the fiber.

\begin{theorem} \label{main:genimpliestypeA}
	Assume that the general singular fibers of $\pi$ are non-multiple. If the endomorphism field of the elliptic factor $E$ of $F$ is different from $\QQ(\sqrt{-1})$ or $\QQ(\sqrt{-3})$ then $\pi$ is of type A. Moreover, if $B \cong \PP^n$ then the degree of the reduced discriminant divisor $\Sigma_{\textup{red}}$ is $2(n+1)$.
\end{theorem}

We do not know examples of Lagrangian fibrations of hyper-K\"ahler manifolds with a multiple general singular fiber. For example, if $\pi$ admits a rational section then the general singular fibers are not multiple. Note also that the degree of the reduced discriminant divisor of any Lagrangian fibration is at least $n+2$ by \cite[Rmk 1.11]{park22}.

\begin{remark}
	One can consider the endomorphism algebra $\textup{End}_{\text{HS}}(H^2(X,\mathbb{Q})_{\mathrm{tr}})$ of the transcendental cohomology, which is either a totally real field or a CM field by \cite{Zarhin}. One says that $X$ has {\it weak CM} by a CM field $K$ (e.g. $\QQ(\sqrt{-d})$) if $K\subseteq \End_{\textup{HS}}(H^2_{\tr}(X,\QQ))$. \Cref{main:genimpliestypeA} (as well as \Cref{j}) essentially says that if $X$ does not have weak CM by $\QQ(\sqrt{-1})$ or $\QQ(\sqrt{-3})$, then $\pi: X\longrightarrow B$ is of type A. Note that K3-type Hodge structures with weak CM by a CM field $K$ are parameterized by Shimura subvarieties of ball quotient type (of dimension at most half the dimension of the period domain). See \cite{dol-kon07}.
\end{remark}

An interesting question is whether every deformation class of hyper-K\"ahler manifolds contains a member with an isotrivial Lagrangian fibration. We are not able to answer this question for the exceptional OG6 and OG10 types. However, our results above show that if such isotrivial Lagrangian fibrations exist they have to be somewhat exotic;  they are either of type B, or they do not admit a rational section. For instance, in view of \Cref{main:classification} and \Cref{main:genimpliestypeA}, we conclude:

\begin{corollary}\label{cor-LSV}
	Let $\pi: X\longrightarrow B$ be deformation equivalent to the LSV fibration for OG10. If $\pi$ is isotrivial then its general fiber is isogenous to $E^5$, where $E$ is an elliptic curve with $j(E)\in\{0,1728\}$. Additionally, $H^2_{\tr}(X)$ has weak CM by $\QQ(\sqrt{-1})$ or $\QQ(\sqrt{-3})$.\qed
\end{corollary}

\begin{remark}
	In fact, a more thorough analysis of the OG10 case is possible. We plan to revisit the question of the existence of isotrivial Lagrangian fibrations for the OG10 deformation class in the near future.
\end{remark}

We close by mentioning two additional results regarding isotrivial Lagrangian fibrations. First, isotriviality is a closed condition in moduli for Lagrangian fibered hyper-K\"ahler manifolds (\Cref{prop:isotriviality is a closed condition}). Second, a general deformation of a Lagrangian fibration of a hyper-K\"ahler manifold $X$ with $b_2(X) \ge 7$ is of maximal variation. (Note that all known examples of hyper-K\"ahler manifolds satisfy $b_2 \ge 7$.)

\begin{theorem}\label{thm:isos-are-rare}
	Let $\pi : X \longrightarrow B$ be a Lagrangian fibration of a hyper-K\"ahler manifold with $b_2(X) \ge 7$. There exists a deformation of the Lagrangian fibration $\pi$ which is not isotrivial.
\end{theorem}

\begin{remark}
	Type A isotrivial fibrations vary in $2$-dimensional moduli with generic endomorphism field $\End_{\HS} (H^2_{\tr} (X, \QQ)) = \QQ$. In contrast, type B isotrivial fibrations can have higher-dimensional moduli (e.g., up to dimension $9$ for K3 surfaces), in which case the generic member of the family has weak CM by $\QQ(\sqrt{-1})$ or $\QQ(\sqrt{-3})$. In the language of Shimura varieties, the former period domain is of orthogonal type whereas the latter is of ball quotient type. Therefore, moduli of type A and type B isotrivial fibrations have distinct irreducible components. Nonetheless some families of type B isotrivial fibrations specialize to a type A fibration with weak CM by $\QQ(\sqrt{-1})$ or $\QQ(\sqrt{-3})$. For example, see cases 45--47 in \Cref{table:classification for mu_6}.
\end{remark}

\medskip

Let us briefly review the content of the paper and sketch some key points.  Let $G$ be the monodromy group of the local system $R^1(\pi_0)_*\underline \ZZ$. To tackle \Cref{main:fiber}, we exploit the compatibility of two structures on the cohomology of the fiber $F$: the Hodge structure and the $G$-module structure. The $G$-equivariant Hodge structure $H^1 (F, \QQ)$ is significantly constrained due to the simplicity of the variation of Hodge structures (VHS) $R^1(\pi_0)_* \underline \QQ$, a result that goes back to \cite{voi92}, \cite{ogu09}, \cite{mat16}, and \cite{shen-yin22}. Representation theory of Mumford--Tate and finite groups shows that $F$ is necessarily isogenous to a power of a simple abelian variety (\Cref{cor:H1 is isotypic}). To show this simple abelian variety is an elliptic curve, we exploit a second rational fibration $\pi': X \dashrightarrow B'$, where $B'=\bar F/G$ is a finite quotient of some abelian variety $\bar F$ isogenous to $F$ (see \Cref{thm:diagonalization}). The variety $B'$ is $\QQ$-Fano as it is dominated by a hyper-K\"ahler manifold, so the quotient map $\bar F \longrightarrow B'$ must be ramified in codimension $1$. Its ramification locus consists of abelian subvarieties of codimension $1$, forcing $\bar F$ to have an elliptic curve factor up to isogeny.\\

For \Cref{main:typeAB}, we study the equivariant Hodge structure on the cohomology of the intermediate base $U$. The K\"unneth formula together with the fact that $h^{2,0}(X)=1$ provides a relationship between $H^1(F,\mathbb{Q})$ and $H^1(U,\mathbb{Q})$, which allows us to translate the information we have about the former into information about the latter. The resulting properties of $H^1 (U,\QQ)$ have geometric consequences for $U$ and its Albanese morphism. Finally, an equivariant version of the theorem of Ueno \cite{ueno} and Kawamata \cite{kaw81} reviewed in \Cref{sec:equivariant Ueno-Kawamata} allow us to conclude that $\kappa(U)$ is equal to $0$ or $n$. When $\kappa(U) = 0$, the normalization of $B$ in the function field $K(U)$ is an abelian variety (\Cref{prop:compactification of U}) and a compactification of $U$.\\

To obtain \Cref{main:genimpliestypeA}, we need to understand the behavior of isotrivial fibrations in codimension $1$. The two principal tools that we use for this purpose (see \Cref{S-can}) are the canonical bundle formula for $K$-trivial fibrations (e.g. \cite{Kollar-can}) and the classification of codimension $1$ fibers by Hwang--Oguiso \cite{HO09,HO11}. Both of them take a simpler form in the case of isotrivial fibrations. Druel--Bianco \cite{Druel-Bianco} provide a version of the canonical bundle formula for isotrivial fibrations and we will refine Hwang and Oguiso's result on general singular fibers of Lagrangian fibrations to show that only the fibers of type $\text{II}$, $\text{III}$, $\text{IV}$, $\text{I}_0^*$, $\text{II}^*$, $\text{III}^*$, and $\text{IV}^*$ may occur in the isotrivial case (\Cref{prop:local structure of pi}). The simplest situation is when the elliptic curve factor $E$ in $F$ doesn't have CM by $\QQ(\sqrt{-1})$ and $\QQ(\sqrt{-3})$, in which case the codimension $1$ fibers are necessarily of type $\text{I}_0^*$ and the local monodromy around the discriminant locus is $\mu_2$, allowing us to deduce the theorem.\\

\Cref{main:typeAB} and its more precise formulation \Cref{thm:typeAB} describe quite explicitly the birational behavior of the isotrivial fibration $\pi : X \longrightarrow B$. To obtain \Cref{main:classification}, assume that $\pi$ is a type A fibration with a rational section. In this case, $X$ is (birational to) a symplectic resolution of $(F \times F')/G$ for two abelian $n$-folds $F$ and $F'$ with $G$-actions. Since $(F \times F')/G$ has a symplectic resolution, $F/G$ is smooth. This forces $F$ and $F'$ together with their $G$-actions to belong to a short list of possibilities by the recent classification of smooth quotients of abelian varieties by Auffarth, Lucchini Arteche, and Quezada \cite{auf-art20, auf-art-que22, auf-art22}.\\

Additionally, we include three appendices. In Appendix \ref{K3appendix}, we discuss the proof of the classification of isotrivial fibrations for K3 surfaces (\Cref{main:K3}) in detail. While fairly straightforward, our treatment highlights some of the issues arising in higher dimensions. The reader may use this discussion as an introduction to the general case. The remaining two appendices contain technical statements needed in the main text.

\subsection*{Notation and Conventions}
An abelian variety $A$ may be considered as an abelian torsor by forgetting its origin. For clarity, we will sometimes write $A^{\tor}$ to emphasize that we consider $A$ as an abelian torsor. Given an abelian variety $A$, we write $\Aut A^{\tor}$ for the automorphism group of the torsor $A^{\tor}$ and $\Aut_0 A$ for the subgroup of automorphisms fixing the origin. The two groups are related by an isomorphism $\Aut A^{\tor} \cong A \rtimes \Aut_0 A$.\\

Given a finite group $G$, we use the term \emph{$G$-variety} for an algebraic variety with a $G$-action. A \emph{$G$-abelian variety} is an abelian variety $A$ together with a $G$-action given by a homomorphism $G \longrightarrow \Aut_0 A$. A \emph{$G$-abelian torsor} is an abelian torsor $A^{\tor}$ together with a $G$-action given by $G \longrightarrow \Aut A^{\tor}$. Note that the group action for a $G$-abelian variety fixes the origin whereas the action for a $G$-abelian torsor may not.\\

Let $A$ and $B$ be abelian varieties, $g: A\longrightarrow B$ a homomorphism, $X$ an $A$-torsor,  $Y$ a $B$-torsor, and $f: X\longrightarrow Y$ a morphism. We write
\[\begin{tikzcd}
		X \ar[d,"f"]&	A \ar[d,"g"]  \\
		Y &	B
		\end{tikzcd}\]
to express the compatibility 
$$f(a\cdot x)=g(b)\cdot f(x).$$
We will use similar notation for more involved diagrams as well.\\

Finally, given a hyper-K\"ahler manifold $X$, a \emph{rational Lagrangian fibration} of $X$ is a rational fibration $X \dashrightarrow B'$ whose general fiber is a Lagrangian subvariety of $X$. In the literature, this name is sometimes reserved for a composition $X \dashrightarrow X' \longrightarrow B'$ of a birational map from $X$ to another hyper-K\"ahler manifold $X'$ and a Lagrangian fibration of $X'$. The former definition is strictly weaker than the latter.

\subsection*{Acknowledgments}
The first author was supported by the ERC Synergy Grant HyperK (ID 854361). The second author was partially supported by NSF (DMS-2101640). The third author was partially supported by the AMS-Simons Travel Grant and would like to thank the Hausdorff Research Institute for Mathematics for its hospitality and support during the Junior Trimester Program titled \emph{Algebraic geometry: derived categories, Hodge theory, and Chow groups}  (Fall of 2023).\\ 
We thank Benjamin Bakker, Philip Engel, Robert Friedman, and Lenny Taelman for many useful comments. We thank J\'anos Koll\'ar and Chenyang Xu for pointing out an error in our original argument.

\section{Trivializations of isotrivial fibrations and Galois action on cohomology}

Throughout the paper we will consider an isotrivial Lagrangian fibration $\pi : X \longrightarrow B$ to a normal projective variety $B$, with generic fiber $F$, which we view as a (fixed) abelian torsor. We will denote by $\Delta\subset B$ the discriminant locus, which is a subvariety of pure codimension $1$, and by $B_0=B\setminus \Delta$ its complement, a smooth quasi-projective variety. The variety $X_0=\pi^{-1}(B_0)$ then admits a smooth isotrivial fibration $\pi_0=\pi_{\mid X_0}:X_0\longrightarrow B_0$. After a finite \'etale base change, such a smooth isotrivial family of abelian varieties can be trivialized. It will be useful to perform this trivialization in two steps. The first \'etale base change trivializes the global monodromy on the first cohomology of the fibers. The second, which may be trivial, provides us with a rational section.\\

To be precise, let
$$G:=\im \left( \pi_1(B_0)\to \Aut \big( H^1(F, \ZZ) \big)\right)$$ 
be the monodromy group of the local system $R^1(\pi_0)_*\underline{\ZZ}$. Since the moduli space of polarized abelian varieties with a sufficiently large level structure is a fine, $G$ is a finite group. The finite index kernel of the monodromy representation then determines a finite \'etale cover $U\longrightarrow B_0$ trivializing the local system $R^1(\pi_0)_*\underline{\ZZ}$ (see \S\ref{S:ex-trivial} for examples of this base change for well-known isotrivial Lagrangian fibrations). We denote by $Z\longrightarrow U$ the base change in the following diagram and refer to it as the \emph{intermediate trivialization}
\begin{equation} \label{diag:intermediate trivialization}
	\begin{tikzcd}
		Z \arrow[r] \arrow[d,swap, "p"] & X_0 \arrow[d,swap, "\pi_0"] \\
		U \arrow[r] & B_0.
	\end{tikzcd}
\end{equation}

If $\pi$ admits a rational section, the intermediate trivialization $Z/U$ is a trivialization of the fibration $\pi_0: X_0\longrightarrow B_0$; equivalently, $X$ is birational to a diagonal quotient $(F\times U)/G$ (Lemma \ref{lem:rational section}). In general, a further base change is needed to trivialize $X_0/B_0$. Though the action of the monodromy group of a Lagrangian fibration on the cohomology of a smooth fiber need not respect the Hodge structure, we show in \Cref{prop:G-action on cohomology} that this is indeed the case for isotrivial Lagrangian fibrations. Hence, the $G$-actions on the cohomology of the fiber $F$ and the base $U$ are compatible with Hodge structures. The assumption that $X$ is a hyper-K\"ahler manifold imposes strong restrictions on the equivariant cohomology of $F$ and $U$ (e.g. \Cref{prop:simple bimodule}, \Cref{lem:isotypic G-component in U}). In particular, we will obtain two important consequences to be used towards our main results:
\begin{itemize}
\item $F$ is isogenous to a power of a simple abelian variety (\Cref{cor:H1 is isotypic}). This falls short of \Cref{main:fiber} but plays a central role in its proof.
\item $U$ admits a map to a simple $G$-abelian torsor which is generically finite on its image. This will be applied in conjunction with \Cref{thm:equivariant Kawamata}, an equivariant version of a theorem of Kawamata, to deduce \Cref{main:typeAB}.
\end{itemize}

\subsection{Galois groups associated to isotrivial Lagrangian fibrations} \label{sec:Galois group}

As discussed above, the cover $U\longrightarrow B_0$ trivializes the polarizable VHS $R^1(\pi_0)_*\underline{\ZZ}$. There is an abelian scheme $P_0 \longrightarrow B_0$ associated to this VHS, which is also trivialized by $U\longrightarrow B_0$, and the smooth fibration $\pi_0:X_0\longrightarrow B_0$ is a torsor for $P_0/B_0$\footnote{This picture holds more generally and without isotriviality assumptions and we refer to \cite{ari-fed16} and \cite{kim25} for further discussion.}.\\

After, a further Galois base change $\tilde U\longrightarrow U\longrightarrow B_0$, which we assume minimal (but which may not be unique), we can fully trivialize the family $\pi : X_0 \longrightarrow B_0$. This is because the intermediate trivialization $p$, as a \emph{projective} $F$-torsor, determines a torsion element $\alpha \in H^1_{\text{\'et}} (U, F)$ (see \cite[Prop XIII.2.3]{ray70}). If $\alpha$ is $d$-torsion, it can be lifted to an element $\tilde \alpha \in H^1_{\text{\'et}} (U, F[d])$, where $F[d]$ is the kernel of the multiplication homomorphism $[d] : F \longrightarrow F$. The finite \'etale covering $\tilde U \longrightarrow U$ corresponding to $\tilde \alpha$ trivializes the family. Passing to a connected component or Galois closure if needed, we may assume that $\tilde U$ is connected and that $\tilde U \longrightarrow U \longrightarrow B_0$ is Galois. We get a cartesian diagram as follows and call the map $\text{pr}_2: F \times \tilde U\longrightarrow U$ the \emph{full trivialization} of $\pi$.
\begin{equation} \label{diag:full trivialization}
	\begin{tikzcd}
		F \times \tilde U \arrow[r] \arrow[d, swap, "\pr_2"] & Z \arrow[r] \arrow[d,swap, "p"] & X_0 \arrow[d,swap, "\pi"] \\
		\tilde U \arrow[r] & U \arrow[r] & B_0.
	\end{tikzcd}
\end{equation}

We denote the three Galois groups arising from the bottom row of this diagram by
\begin{equation} G = \Gal(U/B_0), \qquad \tilde G = \Gal(\tilde U/B_0), \qquad \tilde G_{\tr} = \Gal(\tilde U/U),\end{equation}
and call $G$ (resp. $\tilde G$) the {\it intermediate} (resp. {\it full}){ \it Galois group}. These groups fit in a short exact sequence
\[\begin{tikzcd}
	1 \arrow[r] & \tilde G_{\tr} \arrow[r] & \tilde G \arrow[r] & G \arrow[r] & 1
\end{tikzcd}.\]
Since a finite \'etale morphism being Galois is a property which is stable under base change, the morphism $F \times \tilde U \longrightarrow X_0$ (resp. $Z \longrightarrow X_0$) is also a Galois cover with Galois group $\tilde G$ (resp. $G$). In other words, $F \times \tilde U$ (resp. $Z$) admits a free $\tilde G$-action (resp. $G$-action) with quotient $X_0$. We emphasize that the $\tilde G$-action on $F \times \tilde U$ may not be diagonal when $\pi$ does not have a rational section. In contrast, when $\pi$ has a rational section, we have $G=\tilde G$ (thus, $\tilde G_{\tr}=1$ and $Z=F\times U$) and the $G$-action on $Z$ is diagonal:

\begin{lemma} \label{lem:rational section}
	The following are equivalent.
	\begin{enumerate}
		\item $\pi : X \longrightarrow B$ has a rational section.
		\item The intermediate trivialization $p : Z \longrightarrow U$ is $G$-equivariantly isomorphic to $\pr_2 : F \times U \longrightarrow U$, where $F \times U$ is equipped with a diagonal $G$-action
		\begin{equation} \label{eq:action on constant group scheme}
			g \cdot (x,y) = (f_g(x), \ g \cdot y) \qquad \mbox{for} \quad f_g \in \Aut_0 F .
		\end{equation}
	\end{enumerate}
\end{lemma}
\begin{proof}
	(2) $\Longrightarrow$ (1): The ($G$-equivariant) zero section $U \longrightarrow F \times U$ given by $y \mapsto (0, y)$ descends to a rational section of $\pi : X \longrightarrow B$.\\
	
	(1) $\Longrightarrow$ (2): Recall that $p : Z \longrightarrow U$ is an $F$-torsor, or a torsor under the constant abelian scheme $\pr_2 : F \times U \longrightarrow U$. In \Cref{prop:G-action on cohomology}, we will construct a homomorphism $f : G \longrightarrow \Aut_0 F$, but let us take it momentarily for granted. We can use \eqref{eq:action on constant group scheme} to endow the constant abelian scheme with a $G$-action. With respect to this action, the group scheme action morphism $\rho : (F \times U) \times_U Z \longrightarrow Z$ is equivariant. Now, any rational section $s : B \dashrightarrow X$ is necessarily defined over $B_0$ by \cite[\S 3.3]{kim25}, so we can lift it to an equivariant section $s : U \longrightarrow Z$ of the intermediate fibration $p$. We have constructed a sequence of equivariant morphisms
	\[\begin{tikzcd}[column sep=normal]
		F \times U \arrow[r, "{(\id, s)}"] & (F \times U) \times_U Z \arrow[r, "\rho"] & Z
	\end{tikzcd},\]
	and the theory of torsors implies that the composition $\rho \circ (\id, s)$ is an isomorphism. Hence $F \times U$ and $Z$ are equivariantly isomorphic.
\end{proof}

As the following example illustrates, isotrivial Lagrangian fibrations with $\tilde G_{\tr}\neq 0$ exist already in dimension $2$. In the same fashion, it is not hard to write down examples of Hilbert schemes or generalized Kummer varieties with isotrivial Lagrangian fibrations with $\tilde G_{\tr}\neq 0$.

\begin{example}
Consider an abelian surface $A$ with a fibration $f: A\longrightarrow E$, where $E$ is an elliptic curve. Let $F$ be the fiber of this fibration, $X$ the Kummer surface associated to $A$, and $\pi: X\longrightarrow E/[-1]\cong \mathbb{P}^1$ the elliptic fibration arising from $f$. In this situation,
\begin{itemize}
\item the global monodromy group is $G=\mu_2$,
\item the discriminant divisor $\Delta\subset E/[-1]$ is the image of the $2$-torsion $E[2]\subset E$,
\item $U=E\setminus E[2]$,
\item $Z=A\setminus f^{-1}(E[2]).$
\end{itemize} 
It follows that the family $p:Z\longrightarrow U$ has a section if and only if $f$ has a section i.e., if and only if $A\cong F\times E$ and $f$ is the projection onto the second factor. In fact, given isogenies $\eta: F\times E'\longrightarrow A$ and $\eta': E'\longrightarrow E$ such that $f\circ \eta=\eta'\circ \text{pr}_{2}$, the full trivialization diagram takes the form
\begin{equation} 
	\begin{tikzcd}
		F\times E' \arrow[r] \arrow[d, swap, "\pr_2"] & Z \arrow[r] \arrow[d,swap, "p"] & X_0 \arrow[d,swap, "\pi_0"] \\
		E'\setminus \eta'^{-1}(E[2]) \arrow[r] & E\setminus E[2] \arrow[r] & (E\setminus E[2])/[-1].
	\end{tikzcd}
\end{equation} 
Given such a diagram the group $\tilde G_{\tr}$ is $\ker(\eta)=\text{Gal}(F\times E'/A)$. Requiring that $\eta$ is minimal, namely that it does not factor through a similar isogeny $F\times E''\longrightarrow A$ is insufficient to ensure uniqueness.
\end{example}

\subsection{Cohomology of the fiber $F$} \label{sec:cohomology of F}
 In this subsection, we study the cohomology of the fiber $F$ as a Hodge structure and a $G$-representation. Recall that the intermediate Galois group $G$ is the image of the monodromy representation $\pi_1(B_0) \longrightarrow \GL(H^1 (F, \ZZ))$ associated to the local system $R^1(\pi_0)_* \underline \ZZ$ on $B_0$. A feature of isotrivial Lagrangian fibrations is that $G$ must be a subgroup of $\text{Aut}_0(F)$, and therefore is relatively small. On the other hand, a result of Shen-Yin-Voisin \cite{shen-yin22} states that the $G$-invariant cohomology of $F$ is small, so that $G$ is relatively large. It is this tension that we exploit in order to prove \Cref{main:fiber}. We will first present the proof of \Cref{main:fiber} under the assumption that $\pi$ has a rational section (see \Cref{thm:fiberdiag}), while the general case will be treated in \Cref{sec:fibrations of type AB}.\\

As just mentioned, the monodromy group for isotrivial fibrations is very special.
\begin{proposition} \label{prop:G-action on cohomology}
	The $G$-action on $H^* (F, \ZZ)$ respects the Hodge structure, the cup product, and the polarization induced from $X$. In particular, $G$ is isomorphic to a subgroup of $\Aut_0(F)$.
\end{proposition}
\begin{proof}
	Consider the full trivialization $\pr_2 : F \times \tilde U \longrightarrow \tilde U$. Given $y \in \tilde U$, write $F_y := F\times \{y\}\subset F\times \tilde U$ for the fiber over $y$. An element $g\in \tilde G$ acts on $F \times \tilde U$ by taking $F_y$ to $F_{g.y}$. Identifying both fibers with $F$, we obtain an automorphism of $F$. Hence, for each $g\in \tilde G$ there is a morphism
	\[ f_g : \tilde U \longrightarrow \Aut F^{\tor} .\]
	Since $\tilde U$ is connected, its image lies in a single connected component of $\Aut F^{\tor}$ and defines an element $f_g \in \Aut_0 F = \pi_0 (\Aut F^{\tor})$. One easily checks that this gives a group homomorphism
	\begin{align*} f : \tilde G& \longrightarrow \Aut_0 F\\
	g& \longmapsto\;\;\; f_g\;\;\;\;.\medskip\end{align*}
	
	The full trivialization trivializes the local system $R^k(\pi_0)_* \underline \ZZ$ for all $k$ and the composition of $f$ with the map
	$$\Aut_0 F \hookrightarrow \GL(H^1 (F, \ZZ)) \longrightarrow \GL(H^* (F, \ZZ))$$
	is the monodromy representation $\tilde G \longrightarrow \GL(H^* (F, \ZZ))$. This endows $H^* (F, \ZZ)$ with a $\tilde G$-action. It respects the Hodge structure and cup product because the action of $\Aut_0 F$ on cohomology does. By definition, the image of the homomorphism $f$ is $G$ so $\tilde G_{\tr}$ acts trivially on $H^* (F, \ZZ)$ and $G = \tilde G / \tilde G_{\tr}$ acts faithfully on cohomology. Finally, the fact that $G$ respects the polarization coming from $X$ (and indeed the very fact that such a polarization is unique up to scaling) is a consequence of \Cref{prop:G-invariant cohomology} below.
\end{proof}
	
\begin{remark}
We caution the reader that if $\pi$ does not admit a rational section there is no preferred action of $G$ on $F$. In particular the definition of the injective map $G\longrightarrow \text{Aut}_0(F)$ requires the choice of an identity for $F$.
\end{remark}

On the other hand, the monodromy group of a Lagrangian fibration on a hyper-K\"ahler manifold has to be large (\cite[Thm. 0.4]{shen-yin22}). 
\begin{proposition} \label{prop:G-invariant cohomology}	There is an ample class $\theta\in H^2 (F, \QQ)$ such that the $G$-invariant subspaces of the cohomology of $F$ are given by
	\[ H^k (F, \QQ)^G = \begin{cases}
		0 & \mbox{if}\quad k \mbox{ is odd }, \\
		\QQ \cdot \theta^{k/2} \quad & \mbox{if}\quad k \mbox{ is even }.
	\end{cases} \]
\end{proposition}
\begin{proof}
	This is a reformulation of the result of Shen--Yin--Voisin \cite[Thm 0.4]{shen-yin22}. Given a smooth fiber $F_b = \pi^{-1}(b)$, we write $i : F_b \hookrightarrow X$ for the closed immersion. The image of the restriction map $i^* : H^k (X, \QQ) \longrightarrow H^k (F_b, \QQ)$ is zero if $k$ is odd and $\QQ \cdot \theta^{k/2}$ if $k$ is even. The image of $i^*$ coincides with the monodromy invariant cohomology $H^k (F, \QQ)^{\pi_1(B_0)} = H^k(F,\QQ)^G$ by the global invariant cycle theorem.
\end{proof}

By \Cref{prop:G-action on cohomology}, the $G$-equivariant Hodge structure $H^k (F, \QQ)$ is the $k^{\textup{th}}$ exterior power of the $G$-equivariant Hodge structure $H^1 (F, \QQ)$. 
The following is a consequence of \Cref{prop:G-invariant cohomology}.

\begin{proposition}\label{prop:simple bimodule}
	$H^1 (F, \QQ)$ is simple as a $G$-equivariant Hodge structure.
\end{proposition}
\begin{proof}
	Suppose that $H^1 (F, \QQ) \cong V_1 \oplus V_2$, for nontrivial $G$-equivariant Hodge structures $V_1$ and $V_2$, which are necessarily polarizable. Choose polarizations $\theta_i \in \wedge^2 V_i$, $i = 1,2$. Given any $g\in G$, the element $g\cdot \theta_i\in \wedge^2 V_i$ is also a polarization of $V_i$. Since a sum of polarizations is a polarization and therefore nonzero, averaging $\theta_i$ over $G$ gives a nonzero $G$-invariant element of $\wedge^2 V_i$. This proves that $\dim (\wedge^2 V_i)^G \ge 1$ for $i = 1,2$, thereby violating \Cref{prop:G-invariant cohomology}.
\end{proof}

\begin{corollary} \label{cor:H1 is isotypic}
	$H^1 (F, \QQ)$ is an isotypic Hodge structure. In particular $F$ is isogenous to a power of a simple abelian variety.
\end{corollary}
\begin{proof}
	Letting $\MT$ be the Mumford--Tate group of $H^1 (F, \QQ)$, the $G$-equivariant Hodge structure on $H^1 (F, \QQ)$ is nothing but a $(G, \MT)$-bimodule structure. The claim is now a formal consequence of \Cref{lem:isotypic bimodule} below. Notice that the same argument proves that $H^1 (F, \QQ)$ is an isotypic $G$-module.
\end{proof}

\begin{lemma} \label{lem:isotypic bimodule}
	Let $G$ and $H$ be reductive algebraic groups. Then any finite dimensional isotypic $(G, H)$-bimodule is both $G$-isotypic and $H$-isotypic.
\end{lemma}
\begin{proof}
	Let $V$ be a finite dimensional $(G, H)$-bimodule. It is enough to show any $G$-isotypic decomposition $V = V_1 \oplus V_2$ of $V$ is automatically a $(G, H)$-bimodule decomposition. To prove this, we need to show $V_1$ (resp. $V_2$) is closed under the $H$-action. Given $h\in H$, the subspace $hV_1$ is a $G$-submodule since the $G$ and $H$ actions commute. Thus $hV_1\subset V_1$ because $V_1$ and $V_2$ do not share any simple $G$-module factors.
\end{proof}

In light of \Cref{cor:H1 is isotypic}, to prove \Cref{main:fiber} it suffices to show that $F$ contains a subtorus of codimension $1$. We will discuss the proof of \Cref{main:fiber} in the next section, but to motivate our argument we first present the following special case:

\begin{theorem}\label{thm:fiberdiag}
	If $\pi: X\longrightarrow B$ has a rational section then its smooth fiber $F$ is isogenous to a power of an elliptic curve.
\end{theorem}
\begin{proof}
	By \Cref{lem:rational section}, the intermediate trivialization is isomorphic to the second projection $F \times U \longrightarrow U$, where the $G$-action on $F \times U$ is diagonal. The first projection map $F \times U \longrightarrow F$ is also a $G$-equivariant morphism. Taking its $G$-quotient, we obtain a morphism $X_0 \longrightarrow F/G$, or a dominant rational map $\pi' : X \dashrightarrow B' = F/G$.

	\begin{lemma}\label{lem:Q-Fano}
		The variety $B' = F/G$ is a $\QQ$-Fano variety.
	\end{lemma}
	\begin{proof}
		We claim that
		\begin{itemize}
			\item[\emph{i})] $H^1(B', \QQ) = 0$ and $H^2(B',\mathbb{Q}) \cong \QQ$.
			\item[\emph{ii})] $B'$ is simply connected and uniruled.
			\item[\emph{iii})] $\Pic B' \cong \ZZ$.
		\end{itemize}
		The first item follows from the identification $H^i (B', \QQ) = H^i (F, \QQ)^G$ and \Cref{prop:G-invariant cohomology}. The variety $B'$ is simply connected since $X$ is simply connected and $\pi' : X \dashrightarrow B'$ has irreducible general fiber \cite[Prop 2.10.2]{kol:shaf}. It is uniruled because it is dominated by a hyper-K\"ahler manifold \cite[Thm 3]{kol-lar09} (see also \cite[Thm 1.4]{lin20}). To prove that $\Pic B' \cong \ZZ$, notice from the first item that $\Pic B' \otimes_{\ZZ} \QQ \cong \QQ$. If $B'$ had a torsion line bundle, the associated cyclic covering would give a connected finite \'etale covering of $B'$, violating the fact that $B'$ is simply connected. This shows $\Pic B' \cong \ZZ$. Since $B'$ is uniruled, the canonical divisor $K_{B'}$ must be anti-ample from \cite[Thm 2]{kol-lar09}.
	\end{proof}
	
	Returning to the proof of the theorem, note that $K_F$ is trivial and $K_{F/G}$ is anti-ample, so the quotient morphism $F \longrightarrow F/G$ must be ramified in codimension $1$. In other words, there is a prime divisor $D \subset F$ with non-trivial generic stabilizer. Using the fact that the $G$-action on $H^1 (F, \QQ)$ is faithful, it is easy to see $D$ is an abelian torsor. Hence $F$ contains a subtorus of codimension $1$ and the simple factor of $F$ from \Cref{cor:H1 is isotypic} is an elliptic curve.
\end{proof}

Turning our discussion to the $G$-module structure on the $(1,0)$-cohomology of $F$, we have:

\begin{lemma} \label{lem:H10 as a G-module}
	Denote by $V$ the $G$-module structure on $H^{1,0}(F)$. Then the $G$-modules $\wedge^k V$, $k = 0, 1, \cdots, n$, are mutually non-isomorphic and simple. In particular, $V$ is a non-symplectic simple $G$-module (i.e., $(\wedge^2 V)^G=0$).
\end{lemma}
\begin{proof}
	It is a standard fact from representation theory that given a representation $W$ of a finite group $G$, the conjugate and dual representations of $W$ are isomorphic. Hence $H^{0,1}(F)$ is isomorphic to $V^{\vee}$. We thus have $G$-module isomorphisms
	\[ H^{k,k} (F) \cong \wedge^k H^{1,0}(F) \otimes \wedge^k H^{0,1} (F) \cong \wedge^k V \otimes \wedge^k V^{\vee} \cong \End_{\QQ} (\wedge^k V) .\]
	Taking $G$-invariants, we obtain $H^{k,k}(F)^G \cong \End_G (\wedge^k V)$, which is $1$-dimensional by \Cref{prop:G-invariant cohomology}. Schur's lemma implies that $\wedge^k V$ is a simple $G$-module. Similarly, the fact that $H^{p,q}(F)^G=0$ for $p \neq q$ implies that $\wedge^p V$ and $\wedge^q V$ are not isomorphic. In particular, $\wedge^2 V$ is a simple $G$-module that is not isomorphic to the trivial $G$-module, so $(\wedge^2 V)^G = 0$.
\end{proof}

\begin{corollary}
	If $n \ge 2$, the intermediate Galois group $G$ is not abelian.
\end{corollary}
\begin{proof}
	If $G$ is abelian, then every simple $G$-module over $\CC$ is $1$-dimensional. Since $H^{1,0}(F)$ is an $n$-dimensional simple $G$-module by \Cref{lem:H10 as a G-module}, we get $n = 1$.
\end{proof}

\subsection{Cohomology of the intermediate base $U$} In this subsection, we study the cohomology of the intermediate base $U$ as a $G$-representation and a mixed Hodge structure. We exploit numerical constraints on the cohomology of the hyper-K\"ahler manifold $X$ and the open subset $X_0$ (\Cref{lem:cohomology of X0}) to translate information about the cohomology of $F$ obtained in the last section into information about the cohomology of $U$ (\Cref{prop:simple Albanese morphism} and  \Cref{lem:isotypic G-component in U}).\\

Let us first observe that the cohomology of $Z$ is determined by that of $F$ and $U$:
\begin{proposition} \label{prop:cohomology of Z}
	There is a $G$-equivariant mixed Hodge structure isomorphism
	\[ H^* (Z, \QQ) \cong H^* (F, \QQ) \otimes H^* (U, \QQ) ,\]
	which respects the cup product. In particular,
	\[ W_kH^k (Z, \QQ) \cong \bigoplus_{i=0}^k H^i (F, \QQ) \otimes W_{k-i}H^{k-i} (U, \QQ) ,\qquad 0\leq k\leq 2n.\]
	
\end{proposition}
\begin{proof}
	Recall that there is a Galois covering $F \times \tilde U \longrightarrow Z$ and that the Galois group $\tilde G_{\tr}$ acts trivially on the cohomology of $F$. We thus have a sequence of isomorphisms of $G$-mixed Hodge structures
	\begin{align*}
		H^* (Z, \QQ) \cong  H^* (F \times \tilde U, \QQ)^{\tilde G_{\tr}} &\cong  \Big( H^* (F, \QQ) \otimes H^* (\tilde U, \QQ) \Big)^{\tilde G_{\tr}} \\
		&\cong H^* (F, \QQ) \otimes H^* (\tilde U, \QQ)^{\tilde G_{\tr}} \cong H^* (F, \QQ) \otimes H^* (U, \QQ) .
	\end{align*}
	All these isomorphisms are obviously isomorphisms of mixed Hodge structures.
\end{proof}

The following is the main result of this subsection, illustrating again the rigid cohomological behavior of the intermediate base $U$. It will play an important role in both \Cref{main:typeAB} and \Cref{main:classification}.
		
\begin{proposition} \label{prop:simple Albanese morphism}
	There exists a generically finite $G$-equivariant morphism $U \longrightarrow A$ to a simple $G$-abelian torsor $A$ whose dimension is a positive multiple of $n$. Moreover, this morphism is unique up to composition with a $G$-equivariant isogeny on the left.
\end{proposition}
	
In particular, $U$ has maximal Albanese dimension. The $G$-abelian torsor $A$ may have dimension strictly greater than $n$. We refer the reader to the cases $\mu=\mu_3,\mu_4,$ or $\mu_6$ of \S \ref{SS:K3fib} for examples with $\dim A > n$. However, when $\dim A = n$ the cohomology of $A$ resembles that of $F$.\\

\begin{corollary} \label{cor:G-invariant cohomology for base}
	If $\dim A = n$, the $G$-invariant cohomology of $A$ is
	\[ H^k (A, \QQ)^G \cong \begin{cases}
		0 & \mbox{if}\quad k \mbox{ is odd} \\
		\QQ \quad & \mbox{if}\quad k \mbox{ is even}\quad .
	\end{cases} \]
\end{corollary}
\begin{proof}
	If $\dim A = n$, the construction of $A$ (see the upcoming proof of \Cref{prop:simple Albanese morphism}) gives a $G$-module isomorphism $H^{1,0}(A) \cong V^{\vee}$. Hence $H^{p,q}(A)^G \cong \big( \wedge^p V^{\vee} \otimes \wedge^q V \big)^G \cong \Hom_G (\wedge^p V, \wedge^q V)$, and the result follows from \Cref{lem:H10 as a G-module}.
\end{proof}
	
To prove \Cref{prop:simple Albanese morphism} we need to understand the mixed Hodge structure on $H^1(U,\mathbb{Q})$, in particular its weight one associated graded piece. Since $U$ is a smooth variety, the mixed Hodge structure on $H^1 (U, \QQ)$ has weights $\geq 1$. For notational simplicity, we will write $H^{1,0}(U)$ to denote the $(1,0)$-part of the weight $1$ associated graded $W_1 H^1 (U, \CC)=\text{Gr}^W_1 H^1 (U, \CC)$, i.e.
\[ H^{1,0}(U) := \big( W_1 H^1 (U, \CC) \big)^{1,0} \ \ \subset \ \ H^1 (U, \CC) .\]
Similarly, we use the notation $H^{k,0}(X_0) := \big( W_k H^k (X_0, \CC) \big)^{k,0}$ and $H^{k,0}(Z) := \big( W_k H^k (Z, \CC) \big)^{k,0}$.

\begin{lemma} \label{lem:cohomology of X0}
	Denote by $\sigma_X \in H^{2,0}(X)$ the holomorphic symplectic form on $X$ and by $\sigma_{X_0}\in H^{2,0}(X_0)$ its restriction to $X_0$.
	\begin{enumerate}
		\item We have an isomorphism of mixed Hodge structures
		$$H^* (X_0, \QQ) \cong \Big( H^* (F, \QQ) \otimes H^* (U, \QQ) \Big)^G .$$
		\item $H^{k,0}(X_0) = \begin{cases}
			0 \quad & \mbox{if } k\text{ is even,} \\
			\CC\cdot \sigma_{X_0}^{k/2} \quad & \mbox{if } k\text{ is odd.}
		\end{cases}$
	\end{enumerate}
\end{lemma}
\begin{proof}
	The first item follows from the equality $X_0 = Z / G$ and \Cref{prop:cohomology of Z}. The second item is the standard fact that the restriction homomorphism to a Zariski open subset $H^* (X, \QQ) \longrightarrow H^* (X_0, \QQ)$ gives an isomorphism on the $(k,0)$-piece of the weight $k$ associated graded.
\end{proof}

Let $\sigma$ be the holomorphic symplectic form on $Z$ pulled back from the holomorphic symplectic form on $X$. It is a $G$-invariant symplectic form on $Z$ which is unique up to scaling.

\begin{lemma} \label{lem:symplectic form generates G-invariant}
	$\Big( H^{1,0}(F) \otimes H^{1,0}(U) \Big)^G = \CC \cdot \sigma$ and $\Big(H^{n,0}(F) \otimes H^{n,0}(U) \Big)^G = \CC \cdot \sigma^n$.
\end{lemma}
\begin{proof}
	From \Cref{lem:cohomology of X0} we obtain
	\[ H^{2,0}(F)^G \oplus \Big( H^{1,0}(F) \otimes H^{1,0}(U) \Big)^G \oplus H^{2,0}(U)^G = \CC\cdot  \sigma .\]
	The first summand $H^{2,0}(F)^G$ vanishes by \Cref{prop:G-invariant cohomology}. The third summand $H^{2,0}(U)^G$ also vanishes, because $U/G = B_0$, where $B_0$ is a Zariski open subset of $B$, and $H^{2,0}(B) = 0$ by \cite[Thm 0.4]{shen-yin22}. The first claim follows. The second claim is proven analogously.
\end{proof}

The $G$-module $V:=H^{1,0}(F)$ is simple by \Cref{lem:H10 as a G-module}. We have
\[ H^{1,0}(F) = V, \qquad H^{n,0}(F) = \wedge^n V .\]
Note that $\wedge^n V$ is a non-trivial one-dimensional simple $G$-module.

\begin{lemma} \label{lem:isotypic G-component in U}
	The $V^{\vee}$-isotypic component of $H^{1,0}(U)$ has multiplicity one and the $\wedge^n V^{\vee}$-isotypic component of $H^{n,0}(U)$ has multiplicity one. In particular, $h^{1,0}(U) \ge n$. 
\end{lemma}
\begin{proof}
	Let $m$ be the multiplicity of the $V^{\vee}$-isotypic component of $H^{1,0}(U)$. \Cref{lem:symplectic form generates G-invariant} and Schur's lemma force $m = 1$. The same argument shows that the multiplicity of the $\wedge^n V^{\vee}$-isotypic component of $H^{n,0}(U)$ is one.
\end{proof}

\begin{proof} [Proof of \Cref{prop:simple Albanese morphism}]
	We have seen in \Cref{lem:isotypic G-component in U} that the $V^{\vee}$-isotypic component of $H^{1,0}(U)$ has multiplicity one. Consider the smallest $G$-equivariant Hodge structure $W \subset W_1 H^1 (U, \QQ)$ containing $V^{\vee} \subset H^{1,0}(U)$. It is a simple $G$-equivariant Hodge substructure and determines a $G$-equivariant homomorphism $\Alb_U \longrightarrow A$, where $A$ is a simple $G$-abelian variety. This morphism is unique up to composition with a $G$-equivariant isogeny on the left. Its kernel $K = \ker (\Alb_U \longrightarrow A)$ acts on $\Alb_U^{\tor}$ freely, so we may take a quotient $\Alb_U^{\tor} \longrightarrow A^{\tor} = \Alb_U^{\tor} / K$. The composition $U \longrightarrow \Alb_U^{\tor} \longrightarrow A^{\tor}$ is the desired morphism.\\
	
The symplectic form $\sigma\in H^{1,0}(F) \otimes H^{1,0}(U)\subset H^{2,0}(Z)$ has a nontrivial $n$-th power $\sigma^n \in H^{2n,0}(Z) = H^{n,0} (F) \otimes H^{n,0} (U)$. This implies that the cup product homomorphism restricts to a nonzero homomorphism
	\[ \Big( H^{1,0}(F) \otimes H^{1,0}(U) \Big)^{\otimes n}\longrightarrow \wedge^n H^{1,0}(F) \otimes \wedge^n H^{1,0}(U)  \longrightarrow H^{n,0}(F) \otimes H^{n,0}(U). \]

The conclusion is that the cup product homomorphism $\wedge^n H^{1,0}(U) \to H^{n,0}(U)$ is nonzero and therefore that the pullback $H^{n,0}(\Alb_U^{\tor}) \longrightarrow H^{n,0}(U)$ obtained from the Albanese morphism $U \longrightarrow \Alb_U^{\tor}$ is nonzero. In particular, $U$ has maximal Albanese dimension. The same argument applies to the subspace $V^{\vee} \subset H^{1,0}(U)$ because the symplectic form $\sigma$ is in fact contained in $H^{1,0}(F) \otimes V^{\vee} \subset H^{1,0}(F) \otimes H^{1,0}(U)$. We deduce that the composition $U \longrightarrow \Alb_U^{\tor} \longrightarrow A^{\tor}$ is generically finite on its image.\\
	
	It remains to prove that $\dim A$ is divisible by $n$. Recall that we have constructed $A$ (up to $G$-equivariant isogeny) from a simple $G$-equivariant Hodge structure $W\subset W_1 H^1(U, \QQ)$. We claim that there is an isomorphism of $G$-modules $W^{1,0} \cong V^{\vee} \oplus mV$ for some $m \in \ZZ_{\ge 0}$. Since $W$ is a simple $G$-equivariant Hodge structure, it is an isotypic $G$-module by \Cref{lem:isotypic bimodule}. Because $H^1(F, \CC)=H^{1,0}(F)\oplus H^{0,1}(F)=V\oplus V^{\vee}$, the $G$-module $V \oplus V^{\vee}$ is defined over $\QQ$. This implies that $W_{\CC}$ is a multiple of $V \oplus V^{\vee}$. Finally, $W^{1,0}$ must be of the form $V^{\vee} \oplus mV$ since it contains $V^{\vee}$ with multiplicity one by \Cref{lem:isotypic G-component in U}.
\end{proof}

\subsection{Trivializations of the standard examples} \label{S:ex-trivial}
In this section we discuss examples of isotrivial Lagrangian fibrations. In particular, we revisit $\textup{K3}^{[n]}$ and $\textup{Kum}_{n}$-fibrations introduced in \Cref{def:K3n Kumn fibrations} and discuss their trivializations and associated Galois groups. At the moment, all isotrivial Lagrangian fibrations which the authors know about are $\textup{K3}^{[n]}$ or $\textup{Kum}_{n}$-fibrations. As discussed in the introduction, we are neither able to produce nor to exclude an isotrivial fibration in the deformation type of LSV fibrations of manifolds of OG10 type  \cite{LSV} (see however \Cref{cor-LSV}). \\

Furthermore, if $(n-1)$ is not square free, hyper-K\"ahler manifolds of $\textup{K3}^{[n]}$ type admit additional deformation types of Lagrangian fibrations (see Markman \cite{mar14}; similarly, for $\Kum_n$, see Wieneck \cite{wie18}). Note that, in contrast to $K3^{[n]}$-fibrations, these Lagrangian fibrations do not admit a rational section (even topologically). The first occurrence of such an additional fibration on a hyper-K\"ahler manifold of $\text{K3}^{[n]}$ type is the case $n=5$ (i.e. $n-1=k^2$ with $k=2$). 
Interestingly, this Lagrangian fibration $\pi: X\longrightarrow B$ on $K3^{[5]}$, admits an involution $\iota$ compatible with the Lagrangian fibration, such that the quotient $X/\iota\longrightarrow B$ is deformation equivalent to the LSV fibration on the OG10 manifold (see \cite{dCRS}). More precisely, $X/\iota$ admits a symplectic resolution which is a hyper-K\"ahler manifold of OG10 type, and the Lagrangian fibration $(X/\iota)\longrightarrow B$ deforms to the LSV fibration associated to the intermediate Jacobian fibration for cubic fourfolds\footnote{$X/\iota\longrightarrow B$ can be interpreted as the LSV fibration associated to a period point belonging to the $\mathcal C_2$ Hassett divisor in the moduli of cubic fourfolds (see \cite{Laza10}). The second author learned about this connection from K. O'Grady a long time ago.}. It is clear that $X/B$ is an isotrivial fibration if and only if $(X/\iota)/ B$ is isotrivial. In fact, using the connection between LSV fibrations and cubic fourfolds, it is possible to prove that if an LSV fibration admits an isotrivial member, then it has to be of the form 
$(X/\iota)/ B$ as above. In other words, the deformation class of Markman fibrations on $K3^{[5]}$ with $k=2$ admits an isotrivial member if and only if the deformation class of LSV fibrations on OG10 admits an isotrivial member.  

\subsubsection{Intermediate trivialization of $\textup{K3}^{[n]}$-fibrations}\label{SS:K3fib}
Let $f : S \longrightarrow \PP^1$ be an isotrivial elliptic fibration of a K3 surface and $\pi : X = S^{[n]} \longrightarrow \PP^n$ be its associated $\text{K3}^{[n]}$-fibration. The intermediate Galois group of the isotrivial fibration $f$ is $\mu = \mu_2, \mu_3, \mu_4$, or $\mu_6$ (see \Cref{K3appendix}). Consider its intermediate trivialization
\[\begin{tikzcd}
	Z \arrow[r] \arrow[d] & S \arrow[d, "f"] \\
	C \arrow[r] & \PP^1 .
\end{tikzcd}\]
We can construct from it another commutative diagram
\[\begin{tikzcd}
	Z^n \arrow[r, dashrightarrow] \arrow[d, "p"] & S^{[n]} \arrow[d, "\pi"] \\
	C^n \arrow[r] & (\PP^1)^{(n)} ,
\end{tikzcd}\]
whose restriction to the complement of the discriminant divisor $B_0 \subset B$ gives the intermediate trivialization diagram of $\pi : S^{[n]} \longrightarrow B = (\PP^1)^{(n)}$. The intermediate Galois group is $G = \mu^{\times n} \rtimes \mathfrak S_n$.\\

Notice from the results in \Cref{K3appendix} (or \Cref{thm:typeAB}) that if $C$ is an elliptic curve, then $Z$ is an abelian surface. In this case, the intermediate trivialization $p : Z^n \longrightarrow C^n$ is a morphism between abelian torsors (of dimensions $2n$ and $n$, respectively). If $C$ has genus at least $2$, then $Z$ has Kodaira dimension $1$ and the Kodaira dimensions of the varieties arising in the intermediate trivialization is $\kappa (Z^n) = \kappa(C^n) = n$. In conclusion, the Kodaira dimensions of $Z^n$ and $C^n$ is either $0$ or $n$.\\

We can also describe the discriminant divisor $\Delta\subset B$. It consists of the union of hyperplanes
$$H_{t}=\{x_1+\cdots+x_{n-1}+t\in \text{Sym}^n\PP^1: x_1+\cdots+x_n\in \text{Sym}^{n-1}\mathbb{P}^1\}\subset \text{Sym}^n\PP^1\cong (\PP^1)^{[n]}\cong B,$$
one for each $t$ in the discriminant locus of $f$, together with the big diagonal
$$\delta=\{x_1+\cdots+x_n\in \text{Sym}^n\PP^1: x_i=x_j \text{ for some }i,j\}\subset \text{Sym}^n\PP^1\cong (\PP^1)^{[n]}.$$
It is interesting to consider the behavior of the discriminant divisor as a generic elliptic K3 surface $S/\PP^1$ specializes to an isotrivial elliptic K3. When $S/\mathbb{P}^1$ is generic, the reduced discriminant divisor is the union of $24$ hyperplanes $H_t$ (coming from the $24$ special fibers $S_t$ of $S/\PP^1$), and the big diagonal $\delta$ as above, which is a hypersurface of degree $2(n-1)$ (see \cite[\S4.4]{Sawon-sing}). Note that the local monodromy around the hyperplanes $H_t$ is infinite and the fiber over the generic point of $H_t$ is of Kodaira type $\text{I}_1$ in the Hwang-Oguiso terminology \cite{HO09}.
On the other hand, the local monodromy around the big diagonal is of order $2$, and the generic fiber is of type $\text{I}_0^*$. In this situation, the canonical bundle formula for the fibration $\pi: S^{[n]}\longrightarrow \mathbb{P}^n$(see \Cref{S-can}) reads 
\begin{eqnarray*}
-\deg K_B&=&\frac{1}{12}\deg \sum H_t+\frac{1}{2}\deg \delta, \qquad \text{or}\\
n+1&=& \frac{1}{12}\cdot 24+\frac{2(n-1)}{2}.
\end{eqnarray*}
When $S/\PP^1$ specializes to an isotrivial elliptic fibration, the big diagonal $\delta$ remains a component of the discriminant divisor and its contribution to the canonical bundle formula is unchanged, but the hyperplanes $H_t$ collide and should be counted with a multiplicity corresponding to the Euler characteristic of the singular fiber $S_t$ (compare with \Cref{K3appendix}). The local monodromy around $H_t$ becomes finite, reflecting the monodromy around $t$ for the elliptic fibration $S/\PP^1$.  Similarly, the Hwang--Oguiso fiber type at the generic point of $H_t$ corresponds to the Kodaira fiber type at $t$. This discussion will be generalized in \Cref{S-can}.

\subsubsection{Intermediate trivialization of $\textup{Kum}_{n}$-fibrations}
We will begin by describing the $\textup{Kum}_{n}$-fibration introduced in \Cref{def:K3n Kumn fibrations}. Consider a short exact sequence
\[\begin{tikzcd}
	0 \arrow[r] & E \arrow[r] & A \arrow[r, "f"] & E' \arrow[r] & 0
\end{tikzcd},\]
where $A$ is an abelian surface and $E, E'$ are elliptic curves. Let $(E')^{(n+1)} \longrightarrow E'$ be the summation map from the $(n+1)^{\text{st}}$ symmetric power of $E'$. Define a fiber product $P$ by
\[\begin{tikzcd}
	P \arrow[r] \arrow[d] & A \arrow[d, "f"] \\
	(E')^{(n+1)} \arrow[r] & E' .
\end{tikzcd}\]
The morphism $P \longrightarrow A$ is a $\PP^n$-bundle and the summation map $A^{[n+1]} \longrightarrow A$ admits a factorization $A^{[n+1]} \longrightarrow P \longrightarrow A$ by the universal property of fiber products. Taking the fiber over $0 \in A$, we obtain the isotrivial Lagrangian fibration $\pi : \Kum_n(A) \longrightarrow \PP^n$. Its smooth fiber is isomorphic to $\ker \big( E^{n+1} \longrightarrow E \big) \cong E^n$.\\

Next, we will construct the intermediate trivialization of $\pi$ and describe the Galois group $G$. Let $(E')^{n+1} \longrightarrow (E')^{(n+1)}$ be a quotient map and define a fiber product $Q$ by
\[\begin{tikzcd}
	Q \arrow[r] \arrow[d] & P \arrow[d] \\
	(E')^{n+1} \arrow[r] & (E')^{(n+1)} .
\end{tikzcd}\]
Then $Q$ is an abelian variety of dimension $n+2$ and $Q \longrightarrow P$ is the quotient map for the $\mathfrak S_{n+1}$-action on $Q$. We have a commutative diagram
\[\begin{tikzcd}
	A^{n+1} \arrow[r, dashrightarrow] \arrow[d, "p"] & A^{[n+1]} \arrow[d] \\
	Q \arrow[r] & P \arrow[r] & A ,
\end{tikzcd}\]
where the morphism $p : A^{n+1} \longrightarrow Q$ is again constructed by the universal property. Taking the fiber of the diagram over $0 \in A$, we obtain a commutative diagram
\[\begin{tikzcd}
	\ker (A^{n+1} \to A) \arrow[r, dashrightarrow] \arrow[d, "p"] & X \arrow[d, "\pi"] \\
	\ker \big( (E')^{n+1} \to E' \big) \arrow[r] & \PP^n,
\end{tikzcd}\]
which, after restriction to $B_0 \subset \PP^n$, the complement of the discriminant locus of $\pi$, is an intermediate trivialization diagram. The intermediate Galois group is $G = \mathfrak S_{n+1}$, the base of the intermediate trivialization $\ker \big( (E')^{n+1} \longrightarrow E' \big)$ is isomorphic to $(E')^n$, and the total space of the intermediate trivialization $\ker (A^{n+1} \longrightarrow A)$ is isomorphic to $A^n$.

\section{Fibrations of type A and B} \label{sec:fibrations of type AB}
This section is devoted to the proof of \Cref{main:fiber} and \ref{main:typeAB}. Recall that we have already presented a proof of \Cref{main:fiber} subject to the additional assumption that the isotrivial fibration $\pi$ has a rational section (\Cref{thm:fiberdiag}). A key input was the existence of a second isotrivial rational Lagrangian fibration $X \dashrightarrow F/G$ obtained as a $G$-quotient of the projection map to the first factor $Z\cong F \times U \longrightarrow F$. We will show in \Cref{prop:second isotrivial fibration} that an analogous second isotrivial rational Lagrangian fibration always exists, which will conclude the proof of \Cref{main:fiber} in its full generality. The projection map that was used in the presence of a rational section is replaced by a fibration onto an abelian variety $\bar F$ isogenous to $F$.

\begin{theorem} \label{thm:diagonalization}
	There exists a $G$-equivariant isotrivial fibration $p' : Z \longrightarrow \bar F$ such that:
	\begin{enumerate}
		\item $\bar F$ is a $G$-abelian torsor isogenous to $F$.
		\item The product morphism $(p', p) : Z \longrightarrow \bar F \times U$ is $G$-equivariant, finite \'etale, and Galois.
	\end{enumerate}
	Moreover, the morphism $p'$ is unique up to $G$-equivariant isomorphism of $\bar F$.
\end{theorem}

Some care is needed in using the notion of a fibration since $p'$ is not proper and we refer the reader to \Cref{sec:background} for definitions. If $\pi$ admits a rational section, then $p'$ is merely the projection map $F \times U \longrightarrow F$. One can think of \Cref{thm:diagonalization} as approximating $Z$ by a product $\bar F \times U$ equipped with a diagonal $G$-action.\\

\Cref{thm:diagonalization} also allows us to prove a slightly stronger version of \Cref{main:typeAB}, which describes the Kodaira dimensions\footnote{By the Kodaira dimension of a smooth quasi-projective variety we mean the Kodaira dimension of a smooth projective compactification.} of both $U$ and $Z$.

\begin{theorem} \label{thm:typeAB}
	Let $p : Z \longrightarrow U$ be the intermediate trivialization of the isotrivial fibration $\pi : X \longrightarrow B$. Then one of the following holds.
	\begin{enumerate}
		\item[\textnormal{(A)}] $U$ compactifies to an abelian torsor isogenous to the $n$-th power of an elliptic curve, and either
		\begin{enumerate}
			\item[\textnormal{(A.1)}] $Z$ is birational to an abelian torsor of dimension $2n$, or
			\item[\textnormal{(A.2)}] $\kappa(Z) = n$ and $p : Z \longrightarrow U$ is the Iitaka fibration of $Z$.
		\end{enumerate}
		\item[\textnormal{(B)}] $U$ compactifies to a variety of general type. Moreover, $\kappa(Z) = n$ and $p : Z \longrightarrow U$ is the Iitaka fibration of $Z$.
	\end{enumerate}
\end{theorem}

\begin{example} In view of \Cref{main:K3} and \Cref{SS:K3fib}, it is easy to produce examples of type A.1 and type B respectively. Namely, if $S\longrightarrow \PP^1$ is an isotrivial elliptic fibration as in \Cref{main:K3}(1), then the associated isotrivial fibration $S^{[n]} \longrightarrow \PP^n$ (as in \S\ref{SS:K3fib}) will be of type A.1. On the other hand, if $S\longrightarrow \PP^1$ is as in \Cref{main:K3}(2) and the base $C$ of its associated intermediate trivialization is not elliptic (in some sense, the generic case; compare \Cref{j}), then the associated isotrivial fibration $S^{[n]}\longrightarrow \PP^n$ is of type B. 
\end{example}
\begin{remark}
If $\pi$ admits a rational section, then $Z$ is isomorphic to $F \times U$ by \Cref{lem:rational section} so the Kodaira dimensions of $U$ and $Z$ are the same. Hence, type A.2 may only arise when $\pi$ does not have a rational section. At the moment we do not have examples of isotrivial Lagrangian fibrations of type A.2. This type seems to be related to the presence of multiple fibers in codimension $1$ (conjecturally such Lagrangian fibrations do not exist, but there are no local obstructions to the existence of multiple fibers, e.g. \cite[Ex. 6.3]{HO11}). 
\end{remark}

Our approach is to apply an equivariant version of a theorem of Kawamata  (see \Cref{sec:equivariant Ueno-Kawamata}) to the morphism $U \longrightarrow A$ from \Cref{prop:simple Albanese morphism}. This will give information on the Iitaka fibration of $U$ and in particular on the Kodaira dimension of $U$. The same strategy can be applied to $Z$ by using the $G$-equivariant composition $Z\longrightarrow \bar F\times U\longrightarrow \bar F\times A$.

\subsection{Diagonal approximation}
We will begin with the proof of \Cref{thm:diagonalization}. Recall the full trivialization diagram
\begin{equation}
\begin{tikzcd}
	F \times \tilde U \arrow[r] \arrow[d, swap, "\pr_2"] & Z \arrow[r] \arrow[d,swap, "p"] & X_0 \arrow[d,swap, "\pi"] \\
	\tilde U \arrow[r] & U \arrow[r] & B_0,
\end{tikzcd} \tag{\ref{diag:full trivialization} restated}
\end{equation}
where the horizontal arrows are finite \'etale and Galois. We denoted the three Galois groups by
\[ G = \Gal(U/B_0), \qquad \tilde G = \Gal(\tilde U/B_0),\qquad\text{and} \qquad \tilde G_{\tr} = \Gal(\tilde U/U) ,\]
which satisfy $G \cong \tilde G / \tilde G_{\tr}$. To prove \Cref{thm:diagonalization} we first establish an analogous statement for the full trivialization. It can be interpreted as approximating the $\tilde G$-action on $F \times \tilde U$ by a diagonal action. We recall that the $\tilde G$-action on $F \times \tilde U$ may not be diagonal.

\begin{proposition} \label{prop:diagonalization}
	There exists a $\tilde G$-equivariant isotrivial fibration $p_1 : F \times \tilde U \longrightarrow F^{\flat}$ such that:
	\begin{enumerate}
		\item $F^{\flat}$ is a $\tilde G$-abelian torsor isogenous to $F$.
		\item The product morphism $(p_1, \pr_2) : F \times \tilde U \longrightarrow F^{\flat} \times \tilde U$ is $\tilde G$-equivariant, finite \'etale, and Galois. Here $F^{\flat} \times \tilde U$ is equipped with a diagonal $\tilde G$-action.
	\end{enumerate}
	Moreover, the morphism $p_1$ is unique up to composition with a $\tilde G$-equivariant isomorphism on the left.
\end{proposition}

\Cref{thm:diagonalization} follows from \Cref{prop:diagonalization}:

\begin{proof} [Proof of \Cref{thm:diagonalization}]
	Consider the $\tilde G$-equivariant fibration $F \times \tilde U \longrightarrow F^{\flat}$ from \Cref{prop:diagonalization}. Taking its $\tilde G_{\tr}$-quotient, we obtain the desired $G$-equivariant fibration $Z \longrightarrow \bar F := F^{\flat} / \tilde G_{\tr}$. The group $\tilde G_{\tr}$ acts by translations on $F^{\flat}$ because it acts trivially on $H^* (F^{\flat}, \QQ) = H^* (F, \QQ)$. Hence $\bar F = F^{\flat} / \tilde G_{\tr}$ is again an abelian torsor. Any such $G$-equivariant isotrivial fibration $Z \longrightarrow \bar F$ lifts to a $\tilde G$-equivariant fibration $F \times \tilde U \longrightarrow F^{\flat}$, so uniqueness follows.
\end{proof}

The rest of this section will be devoted to the proof of \Cref{prop:diagonalization}.\\

A $\tilde G$-action on $F \times \tilde U$ is a collection of automorphisms
\[ g : F \times \tilde U \longrightarrow F \times \tilde U, \qquad g \in \tilde G, \]
satisfying the condition
\begin{equation}\label{actionaxiom}g\cdot (h\cdot (x,y)) = (gh)\cdot (x,y).\end{equation}
Moreover, since in our setting the morphism $\pr_2 : F \times \tilde U \longrightarrow \tilde U$ is $\tilde G$-equivariant, the action is of the form (\Cref{prop:structure theorem for abelian morphism})
\begin{equation} \label{eq:action}
	g\cdot (x,y) = (f_g(x) + e_g (g\cdot y), \ g\cdot y) , \qquad f_g \in \Aut_0 F, \quad e_g : \tilde U \longrightarrow F.
\end{equation}
The group $\operatorname{Mor}(\tilde U, \Aut F^{\tor})$ admits a $\tilde G$-action induced from that on $\tilde U$. Condition \eqref{actionaxiom} boils down to a cocyle condition on the map
\begin{align*} \tilde G& \longrightarrow   \operatorname{Mor}(\tilde U, F) \rtimes \Aut_0 F\cong \operatorname{Mor}(\tilde U, \Aut F^{\tor}) \\
g& \longmapsto \;\;\;\;\;\;\;\;\;\;(e_g, f_g) \;\;\;,\end{align*}
so we obtain an element in the nonabelian cohomology
$$H^1 (\tilde G, \ \operatorname{Mor}(\tilde U, \Aut F^{\tor})) .$$
Conversely, one checks that this nonabelian cohomology classifies $\tilde G$-actions on $F \times \tilde U$ such that $\pr_2 : F \times \tilde U \longrightarrow \tilde U$ is equivariant.\\

The $\tilde G$-action on $F\times \tilde U$ is diagonal if and only if the morphisms $e_g : \tilde U \longrightarrow F$ in \eqref{eq:action} are constant for all $g \in \tilde G$. We will show that this can also be seen from the $\tilde G$-action on the Albanese variety or Albanese torsor of $F\times \tilde U$. As we did for abelian varieties, we will need to distinguish between the notion of Albanese variety and Albanese torsor (see \Cref{sec:background} for details). The reason for these technicalities is that when defining the Albanese map without choosing a basepoint one obtains a morphism to an abelian torsor, not an abelian variety. The Albanese diagram in our $\tilde G$-equivariant setting is
\[\begin{tikzcd}
	F \times \tilde U \arrow[r] \arrow[d, "\pr_2"] & \Alb_{F \times \tilde U}^{\tor} \arrow[d] & \Alb_{F \times \tilde U} \arrow[d] \\
	\tilde U \arrow[r] & \Alb_{\tilde U}^{\tor} & \Alb_{\tilde U} ,
\end{tikzcd}\]
where the diagram on the left is a commutative diagram of $\tilde G$-varieties and the one on the right is a homomorphism of $\tilde G$-abelian varieties.

\begin{lemma} \label{cri:diagonal action}
	The following are equivalent:
	\begin{enumerate}
		\item The $\tilde G$-action on $F \times \tilde U$ is diagonal.
		\item The induced $\tilde G$-action on the Albanese variety $\Alb_{F \times \tilde U} = F \times \Alb_{\tilde U}$ is diagonal.
		\item The induced $\tilde G$-action on the Albanese torsor $\Alb_{F \times \tilde U}^{\tor} = F^{\tor} \times \Alb_{\tilde U}^{\tor}$ is diagonal.
	\end{enumerate}
\end{lemma}
\begin{proof}
	Note that the Albanese variety of $\tilde U$ is determined by the weight $1$ pure Hodge structure $W_1 H^1 (\tilde U, \ZZ)=\text{Gr}_{W}^1 H^1(\tilde U, \ZZ)$ together with its polarization. We claim that the $\tilde G$-action on $F \times \tilde U$ is diagonal if and only if the K\"unneth decomposition
	\[ W_1 H^1 (F \times \tilde U, \QQ) \cong H^1 (F, \QQ) \oplus W_1 H^1 (\tilde U, \QQ) \]
	is $\tilde G$-equivariant. This will prove the desired equivalence.\\
	
	The implication $(\Longrightarrow)$ is clear, so we concentrate on the implication $(\Longleftarrow)$ and assume that the $\tilde G$-action is diagonal on the weight one associated graded. We will use the notation from \eqref{eq:action} for the action of $g \in \tilde G$ on $F \times \tilde U$. With respect to the K\"unneth decomposition of the first cohomology
	$$W_1 H^1 (F \times \tilde U, \QQ) = H^1 (F, \QQ) \oplus W_1 H^1 (\tilde U, \QQ),$$ the automorphism $g^*$ on $W_1 H^1 (F \times \tilde U, \QQ)$ takes the form
	\[ g^* = \begin{psmallmatrix}
		f_g^* & 0 \\
		g^* \circ e_g^* & g^*
	\end{psmallmatrix} .\]
It follows that $g^* \circ e_g^* : H^1 (F, \QQ) \to W_1 H^1 (\tilde U, \QQ)$ vanishes, or equivalently that $e_g^* : H^1 (F, \QQ) \to W_1 H^1 (\tilde U, \QQ)$ vanishes. By the universal property of the Albanese morphism $\tilde U \to \Alb_{\tilde U}^{\tor}$, we have a factorization:
	\[\begin{tikzcd} [column sep=tiny]
		\tilde U \arrow[rr, "e_g"] \arrow[rd] & & F .\\
		& \Alb_{\tilde U}^{\tor} \arrow[ru, "e_g'"']
	\end{tikzcd}\]
	The homomorphisms $(e_g')^* : H^1 (F, \QQ) \to H^1 (\Alb_{\tilde U}^{\tor}, \QQ) = W_1 H^1 (\tilde U, \QQ)$ vanishes since $(e_g')^* = e_g^* = 0$. But $e_g'$ is a morphism between abelian torsors, so $e_g'$ must be a constant map. As a result, $e_g$ is constant as well.
\end{proof}

We can now start the proof of \Cref{prop:diagonalization}. For clarity, we first prove the following intermediate claim without the uniqueness assertion (if we do not require $p_1 : F \times \tilde U \longrightarrow F^{\flat}$ to be a fibration then we lose uniqueness since we can compose on the left with a $\tilde G$-equivariant isogeny).

\begin{lemma} \label{lem:diagonalizagion}
	There exists a $\tilde G$-equivariant isotrivial morphism $p_1 : F \times \tilde U \longrightarrow F^{\flat}$ satisfying (1) and (2) in \Cref{prop:diagonalization}.
\end{lemma}
\begin{proof}
	To avoid confusion, we distinguish between the notions of abelian variety and abelian torsor. We shall divide the proof in three steps.
	\begin{itemize}
	\item Step 1: Construction of $F^{\flat}$,
	\item Step 2: Construction of the morphism $F^{\tor} \times \tilde U \longrightarrow F^{\flat, \tor}$,
	\item Step 3: Proof that $F^{\tor} \times \tilde U \longrightarrow F^{\flat, \tor} \times \tilde U$ satisfies the desired properties.
	\end{itemize}
	
	\noindent (Step 1) The projection $\pr_2 : F^{\tor} \times \tilde U \longrightarrow \tilde U$ induces a homomorphism $f : \Alb_{F \times \tilde U} \longrightarrow \Alb_{\tilde U}$ between $\tilde G$-Albanese varieties and a short exact sequence of $\tilde G$-equivariant Hodge structures
	\[\begin{tikzcd}
		0 \arrow[r] & H^1 (\Alb_{\tilde U}, \ZZ) \arrow[r, "f^*"] & H^1 (\Alb_{F \times \tilde U}, \ZZ) \arrow[r, "\varphi"] & H^1 (F, \ZZ) \arrow[r] & 0
	\end{tikzcd}.\]
The reader should recall the subtle fact that $H^* (F, \ZZ)$ admits a canonical $\tilde G$-module structure even though $F$ does not. Since a $\tilde G$-equivariant polarizable Hodge structure over $\QQ$ is completely reducible this short exact sequence splits over $\QQ$, i.e., there exists an injective $\tilde G$-equivariant Hodge structure homomorphism $H^1 (F, \QQ) \longrightarrow H^1 (\Alb_{F \times \tilde U}, \QQ)$ splitting the sequence. Define a $\tilde G$-equivariant Hodge substructure $W \subset H^1 (\Alb_{F \times \tilde U}, \ZZ)$ by
	\[ W := \im \Big[ H^1 (F, \QQ) \longrightarrow H^1 (\Alb_{F \times \tilde U}, \QQ) \Big] \cap H^1 (\Alb_{F \times \tilde U}, \ZZ) \quad \subset \ \ H^1 (\Alb_{F \times \tilde U}, \ZZ) .\]
	
	Using the contravariant equivalence between the category of polarizable $\tilde G$-Hodge structures of weight one and that of $\tilde G$-abelian varieties, we can translate this geometrically to a surjective homomorphism of $\tilde G$-abelian varieties
	\[ \Alb_{F \times \tilde U} \longrightarrow F^{\flat},\]
where $H^1 (F^{\flat}, \ZZ) = W $. The composition $W \hookrightarrow H^1 (\Alb_{F \times \tilde U}, \ZZ) \twoheadrightarrow H^1 (F, \ZZ)$ has finite kernel so the following composition is an isogeny of abelian varieties inducing a $\tilde G$-equivariant map on cohomology
\[
\begin{tikzcd}[row sep=-0.1cm]
		F \ar[r] & \Alb_{F \times\tilde U} \ar[r] & F^{\flat}\\
		x \ar[rr,mapsto] && x^\flat.
\end{tikzcd}
\]
	
	\noindent (Step 2) The morphism $\Alb_{F \times \tilde U} \to F^{\flat} \times \Alb_{\tilde U}$ is a $\tilde G$-equivariant isogeny of $\tilde G$-abelian varieties, where the $\tilde G$-action on the right hand side is diagonal. Let $K$ be its kernel. It is a finite $\tilde G$-abelian group acting freely on the Albanese torsor $\Alb_{F \times \tilde U}^{\tor} = F^{\tor} \times \Alb_{\tilde U}^{\tor}$ and the quotient $(F^{\tor} \times \Alb_{\tilde U}^{\tor})/K$ is an $F^{\flat} \times \Alb_{\tilde U}$-torsor, so it is isomorphic to $F^{\flat, \tor} \times \Alb_{\tilde U}^{\tor}$. The quotient morphism is necessarily of the form (\Cref{prop:structure theorem for abelian morphism})
	\begin{align*} q' : F^{\tor} \times \Alb_{\tilde U}^{\tor} &\longrightarrow F^{\flat, \tor} \times \Alb_{\tilde U}^{\tor}\\
	 (x, y) \;\;\;\;\;\;\;&\longmapsto (x^{\flat} + e'(y), y) \end{align*}
	for a morphism $e' : \Alb_{\tilde U}^{\tor} \longrightarrow F^{\flat, \tor}$. Defining a new morphism $q$ by
	\begin{align*} q : F^{\tor} \times \tilde U &\longrightarrow \; F^{\flat, \tor} \times \tilde U\\
	 (x,y) \;\;\; &\longmapsto (x^{\flat} + e(y), y), \end{align*}
	where $e$ denotes the composition $\tilde U \to \Alb_{\tilde U}^{\tor} \xrightarrow{e'} F^{\flat, \tor}$, we obtain a \emph{cartesian} diagram
	\begin{equation}\label{diag:Galois}\begin{tikzcd}
		F^{\tor} \times \tilde U \arrow[r] \arrow[d, "q"] & F^{\tor} \times \Alb_{\tilde U}^{\tor} \arrow[d, "q'"] \\
		F^{\flat, \tor} \times \tilde U \arrow[r] & F^{\flat, \tor} \times \Alb_{\tilde U}^{\tor} .
	\end{tikzcd}\end{equation}
	The desired morphism $p_1 : F^{\tor} \times \tilde U \to F^{\flat, \tor}$ is the composition of $q$ with the projection to the first factor i.e., $p_1(x,y)= x^{\flat} + e(y)$.\\
	
	\noindent (Step 3) It remains to prove that the morphism $q$ is a Galois covering and that the $\tilde G$-action on $F^{\flat, \tor} \times \tilde U$ is diagonal. First, the $\tilde G$-action on $F^{\flat, \tor} \times \tilde U$ is diagonal because the induced action on its Albanese variety $F^{\flat} \times \Alb_{\tilde U}$ is diagonal by construction (\Cref{cri:diagonal action}). The map $q$ is a Galois covering since $q'$ is and diagram \eqref{diag:Galois} is cartesian.
\end{proof}

The Stein factorization of a morphism $F \times \tilde U \longrightarrow F^{\flat}$ as in \Cref{lem:diagonalizagion} will be the desired (unique) fibration of \Cref{prop:diagonalization}. 

\begin{proof} [Proof of \Cref{prop:diagonalization}]
	We again divide the proof into two steps: the existence and uniqueness of the fibration.\\
	
	\noindent (Existence) Let $p_1 : F \times \tilde U \longrightarrow F^{\flat}$ be a morphism as in \Cref{lem:diagonalizagion}, which takes the form $p_1 (x, y) = x^{\flat} + e(y)$, for a morphism $e : \tilde U \longrightarrow F^{\flat}$. Consider a smooth projective compactification $Y\supset \tilde U$. Then $e$ uniquely extends to $e : Y \longrightarrow F^{\flat}$, so that we can extend $p_1$ canonically to
	\begin{align*}p_1 : F \times Y &\longrightarrow \;\;\;\;F^{\flat}\\ (x,y) \;&\longmapsto x^{\flat} + e(y) .\end{align*}
	It is a smooth and projective fiber bundle whose fiber is isomorphic to a (possibly disconnected) finite \'etale covering of $Y$. Let $F^{\natural}$ be the variety appearing in the Stein factorization of $p_1$:
\begin{equation}\label{diag:Fsharp}\begin{tikzcd}
		F \times Y \arrow[r] \arrow[rd, "p_1"'] & F^{\natural} \arrow[d] \\
		& F^{\flat}
	\end{tikzcd}.\end{equation}
	The morphism $F \times Y \longrightarrow F^{\natural}$ is an isotrivial fibration whose fiber $T$ is a finite \'etale connected covering of $Y$. Its restriction $F \times \tilde U \longrightarrow F^{\natural}$ is also isotrivial with a smooth and connected fiber $M \subset T$. The morphism $F^{\natural} \to F^{\flat}$ is finite \'etale because the number of connected components of the fibers of $p_1$ is constant. This proves that $F^{\natural}$ is an abelian torsor. Finally, from the definition one easily checks that $F \times Y \longrightarrow F^{\natural} \times Y$ is a Galois covering.\\
	
	It remains to prove that $F^{\natural}$ admits a $\tilde G$-action making diagram \eqref{diag:Fsharp} $\tilde G$-equivariant. By definition, the function field $K(F^{\natural})$ of the Stein factorization is the algebraic closure of $K(F^{\flat})$ in $K(F \times \tilde U)$. The embedding $K(F^{\flat}) \subset K(F \times \tilde U)$ is $\tilde G$-equivariant, so its algebraic closure is closed under the $\tilde G$-action. This endows $F^{\natural}$ with a birational $\tilde G$-action. However, since $F^{\natural}$ is an abelian torsor, a birational automorphism of $F^{\natural}$ is an automorphism. It follows that $F^{\natural}$ is a $\tilde G$-abelian torsor and that diagram \eqref{diag:Fsharp} is $\tilde G$-equivariant.\\
	
	\noindent (Uniqueness) Suppose $F_i^\flat$, $i=1,2$ are two $\tilde G$-abelian varieties isogenous to $F$ and that $p_i: F\times \tilde U\longrightarrow F_i^\flat$, $i=1,2$ are two fibrations satisfying the conditions of the proposition. Fix a smooth projective compactification $Y\supset \tilde U$. The morphism $p_i : F \times \tilde U \longrightarrow F^{\flat}_i$ is necessarily of the form $p_1(x,y) = x^{\flat} + e(y)$ by \Cref{prop:structure theorem for abelian morphism}, so it canonically extends to a fibration $F \times Y \longrightarrow F^{\flat}_i$.  The abelian varieties $F_1^\flat$ and $F_2^\flat$ are isogenous to $F$ and the $\tilde G$-representations $H^1(F_i^\flat,\mathbb{Q}),$ $i=1,2,$ are isomorphic to each other. We deduce that $F_1^\flat$ and $F_2^\flat$ are $\tilde G$-equivariantly isogenous. Hence, there exists a third $\tilde G$-abelian variety $F_3^{\flat}$ with $\tilde G$-equivariant isogenies dominated by both $F_1^{\flat}$ and $F_2^{\flat}$, giving a commutative diagram
	\[\begin{tikzcd}
		F \times Y \arrow[r,"p_1"] \arrow[d,"p_2"] & F_1^{\flat} \arrow[d] \\
		F_2^{\flat} \arrow[r] & F_3^{\flat} .
	\end{tikzcd}\]
The (equivariant) uniqueness of the Stein factorization then implies that $p_1=p_2$.
\end{proof}

\subsection{Kodaira dimension of the intermediate trivialization}We now return to the study of the birational type of the base $U$ of the intermediate trivialization. Namely, we establish that the Kodaira dimension $\kappa(U)$ is equal to $0$ or $n$, as stated in \Cref{thm:typeAB} of the introduction. The proof of the fact that if $\kappa(U) = 0$ then $U$ compactifies to an abelian torsor will appear in the next section (\Cref{prop:compactification of U}).
	
\begin{proof} [Proof of \Cref{thm:typeAB}]
	We first prove that $U$ is either of general type or birational to a $G$-abelian torsor. Consider the equivariant morphism $U \longrightarrow A$ from \Cref{prop:simple Albanese morphism} and fix an equivariant smooth projective compactification $U \subset \tilde Y$. Then the morphism $U \longrightarrow A$ extends to a morphism $\tilde Y \longrightarrow A$ because $\tilde Y$ is smooth. Its Stein factorization gives an equivariant fibration $\tilde Y \longrightarrow Y$ together with an equivariant finite morphism $f : Y \longrightarrow A$. The fibration $\tilde Y \longrightarrow Y$ is a birational morphism since $\tilde Y \longrightarrow A$ is generically finite on its image by \Cref{prop:simple Albanese morphism}. It thus suffices to prove that the normal projective $G$-variety $Y$ is either of general type or an abelian torsor.\\

	By an equivariant version of a theorem of Ueno and Kawamata \cite[Thm 13]{kaw81}, which we include as \Cref{thm:equivariant Kawamata} for completeness, there exists a $G$-abelian subvariety $A_1 \subset A$ such that
	\[ \kappa(Y) = n - \dim A_1 .\]
	Since $A$ is a simple $G$-abelian variety, $A_1 \subset A$ is either a point or $A$ itself. If $A_1$ is a point then we obtain $\kappa(Y) = n$. If $A_1 = A$, since $\dim A \ge n$ we have $\kappa(Y) = n - \dim A \le 0$. This forces $\kappa(Y) = 0$ and $\dim A = n$. A theorem of Kawamata--Viehweg \cite{kaw-vie80} \cite[Thm 4]{kaw81} states that if a morphism $f : Y \longrightarrow A$ to an abelian torsor is finite surjective and $\kappa(Y) = 0$, then $f$ is \'etale and $Y$ is an abelian torsor. One shows that it is isogenous to the $n$-th power of an elliptic curve by the same reasoning as in \Cref{thm:fiberdiag}: use \Cref{cor:G-invariant cohomology for base} to deduce that $Y$ is isogenous to a power of a simple abelian torsor and the fact that $Y/G$ is $\QQ$-Fano to conclude that it is a power of an elliptic curve.\\

	Next, we apply to same method to study the Kodaira dimension of $Z$. In \Cref{thm:diagonalization}, we constructed a finite morphism $Z \longrightarrow \bar F \times U$ and in  \Cref{prop:simple Albanese morphism} a morphism $U \longrightarrow A$, which is generically finite on its image. We thus obtain a morphism $Z \longrightarrow \bar F \times A$ that is generically finite on its image. As above, considering a smooth equivariant compactification of $Z$ and a Stein factorization gives an equivariant finite morphism $f : Y \longrightarrow \bar F \times A$ from a normal projective $G$-variety $Y$ that is birational to $Z$. Applying \Cref{thm:equivariant Kawamata} to $f : Y \longrightarrow \bar F \times A$, we find three possibilities: $\kappa(Z) = 2n$, $n$ or $0$ (because $\dim A$ is a multiple of $n$ by \Cref{prop:simple Albanese morphism}). However, $\kappa(Z) \le n$ because the morphism $p : Z \longrightarrow U$ with general fiber $F$ gives an upper bound on the Kodaira dimension (e.g., \cite[Cor 2.3]{mori87})
	\[ \kappa(Z) \le \kappa(F) + \dim U = n .\smallskip\]
	
	If $\kappa(Z) = 0$ then Kawamata's characterization of abelian torsors \cite[Cor 2]{kaw81} shows that $Z$ is birational to an abelian torsor. If $\kappa(Z) = n$, notice that the dimension of $U$ coincides with the Kodaira dimension of $Z$, and that $Z \longrightarrow U$ is a fibration whose very general fiber has Kodaira dimension $0$. These two properties characterize the Iitaka fibration of $Z$ (see \cite[Def-Thm 1.11]{mori87}). In fact, the diagram
	\[\begin{tikzcd}
		Z \arrow[r] \arrow[d, "p"] & \bar F \times A \arrow[d, "\pr_2"] \\
		U \arrow[r] & A
	\end{tikzcd}\]
	is precisely the Ueno--Kawamata diagram from \Cref{thm:equivariant Kawamata}.
\end{proof}

\subsection{The second isotrivial fibration} \label{sec:second isotrivial fibration}
Another interesting byproduct of \Cref{thm:diagonalization} is the existence of a second fibration of the hyper-K\"ahler manifold $X$. Since $p' : Z \longrightarrow \bar F$ is $G$-equivariant, taking its $G$-quotient yields a new fibration $X_0 \longrightarrow B':= \bar F/G$, or a rational fibration $\pi' : X \dashrightarrow B'$.

\begin{proposition} \label{prop:second isotrivial fibration}
	The Lagrangian fibration $\pi: X\longrightarrow B$ uniquely determines a second rational isotrivial Lagrangian fibration $\pi' : X \dashrightarrow B' $ to a projective $\mathbb{Q}$-Fano variety $B'$.
\end{proposition}
\begin{proof}
	As mentioned, $\pi'$ is constructed by taking the quotient of $p' : Z \longrightarrow \bar F$ in \Cref{thm:diagonalization}. Let us first show that its general fiber is Lagrangian. Let $\sigma$ be a symplectic form on $Z$ pulled back from that on $X$. A general fiber of $\pi'$ is a general fiber of $p' : Z \longrightarrow \bar F$, so it suffices to prove the latter is Lagrangian with respect to $\sigma$. Decompose $p'$ into $Z \longrightarrow \bar F \times U \longrightarrow \bar F$, where the first map $(p', p)$ is finite \'etale by \Cref{thm:diagonalization} and the second map is the projection onto the first factor. The pullback homomorphism $H^* (\bar F \times U, \QQ) \longrightarrow H^* (Z, \QQ)$ is an isomorphism of $G$-equivariant Hodge structures by \Cref{prop:cohomology of Z}. By \Cref{lem:symplectic form generates G-invariant}, the general fiber of $\pr_1 : \bar F \times U \longrightarrow \bar F$ is Lagrangian with respect to $\text{Tr}(\sigma)\in H^{2,0}(\bar F\times U)$. Hence the general fiber of $p'$ is Lagrangian with respect to $\sigma$. Finally, the base $B'$ is $\QQ$-Fano by \Cref{lem:Q-Fano}.
\end{proof}

\begin{proof}[Proof of \Cref{main:fiber}]
	\Cref{cor:H1 is isotypic} shows that $F$ is isogenous to a power of a simple abelian variety. It therefore suffices to show that $F$ has an isogeny factor of dimension $1$. Let $\pi': X \dashrightarrow B'$ be the second fibration from \Cref{prop:second isotrivial fibration}. We proved that the variety $B' = \bar F/G$ is $\QQ$-Fano. If a finite quotient $B'$ of an abelian torsor $\bar F$ is $\QQ$-Fano, then $\bar F$ is isogenous to a product of elliptic curves by \cite[Lem 4.2]{shi21}. This concludes the proof of the theorem.\\
	
	For the sake of completeness, let us also present a simple argument which does not rely on \cite[Lem 4.2]{shi21}. Let
	$ q : \bar F \longrightarrow \bar F / G = B' $ be the quotient map. Since $K_{\bar F}$ is trivial and $K_{B'}$ is anti-ample, $q$ is necessarily ramified at some codimension $1$ points in $\bar F$. Equivalently, there exists a prime divisor $D \subset \bar F$ whose generic point has a non-trivial stabilizer $H \subset G$. The restriction map $H^1(\bar F, \QQ) \longrightarrow H^1(D, \QQ)$ cannot be injective since $G$ acts faithfully on $H^1(\bar F, \QQ)$ and $H$ acts trivially on $H^1 (D, \QQ)$. The following lemma ensures that $\bar F$ has an isogeny factor of dimension one.
\end{proof}

\begin{lemma}
	Let $\bar F$ be an abelian torsor and $D \subset \bar F$ a prime divisor such that $H^1 (\bar F,\mathbb{Q}) \longrightarrow H^1 (D,\mathbb{Q})$ is not injective. Then $D$ is an abelian subtorsor of $\bar F$. In particular, $\bar F$ has a 1-dimensional isogeny factor.
\end{lemma}
\begin{proof}
	Factorize the inclusion $D \subset \bar F$ through the Albanese morphism $D \longrightarrow \Alb^{\tor}_D \longrightarrow \bar F$. Since $H^1 (\bar F, \QQ) \longrightarrow H^1 (D, \QQ)$ is not injective, the morphism $\Alb^{\tor}_D \longrightarrow \bar F$ is not surjective. Its image $D$ is therefore an abelian subtorsor containing $D$. This implies that $D$ is an abelian subtorsor of codimension $1$.\end{proof}

\begin{remark}
	Let $M$ be a general fiber of the rational (isotrivial) Lagrangian fibration $\pi' : X \dashrightarrow B'$ above. Then one can prove that $M$ is either birational to an abelian variety or compactifies to a variety of general type. This can be deduced, for example, from \Cref{prop:classification without conjecture} below.
\end{remark}

\section{General singular fibers and the canonical bundle formula}\label{S-can}
So far, our analysis has been focused on the smooth fibration $\pi_0:X_0\longrightarrow B_0$ obtained by restricting $\pi$ to the complement of the singular fibers, and in particular on the cohomological properties of the general fiber $F$ and of the base $U$ of the intermediate trivialization. We now extend our study to a fibration $\pi_1: X_1\longrightarrow B_1$ obtained by restricting $\pi$ to the complement of a codimension two locus pulled back from the base $B$. Since $B$ is normal, we may assume $B_1$ is smooth. There are two general tools that apply to this situation. The first is Hwang and Oguiso's description of general singular fibers using the notion of characteristic cycle \cite{HO09,HO11}. The second is the canonical bundle formula of Kawamata \cite{Kaw-subadjunction}, Ambro \cite{Ambro-JDG, Ambro-b}, and Fujino \cite{Fujino-can}, which is a higher dimensional generalization of Kodaira's canonical bundle formula for elliptic fibrations. It takes the form
\begin{equation}\label{eq-can-formula}
	K_X \sim_{\QQ} \pi^*(K_B+M_B+D_\pi) ,
\end{equation}
and expresses the ($\mathbb{Q}$-linear equivalence class of the) canonical bundle of the total space of a relatively minimal $K$-trivial fibration $\pi:X \longrightarrow B$ in terms of three $\QQ$-divisors on the base $B$ (see \cite[Thm. 8.5.1]{Kollar-can}): the canonical class of the base $K_B$, the moduli divisor $M_B$ (which only depends on the restriction of the fibration to the smooth locus), and a discriminant divisor $D_\pi$ which accounts for the singular fibers in codimension one (compare \eqref{eq-dpi} below).\\

In the isotrivial case, both the local structure of the general singular fibers and the canonical bundle formula take a particularly simple form. Namely, the local monodromy groups are automatically finite, and in fact cyclic of order $m \in \{2,3,4,6\}$ (\Cref{cor:local monodromy index}, \cite[Lemma 3.6]{BS}), and $m > 2$ is only possible if the elliptic curve $E$ occurring in \Cref{main:fiber} has CM by $\QQ(\sqrt{-1})$ or $\QQ(\sqrt{-3})$. Moreover, under the assumption that there are no multiple fibers in codimension $1$,  we show that the only Hwang--Oguiso fiber types that occur for isotrivial fibrations are $\text{II}$, $\text{III}$, $\text{IV}$, $\text{I}_0^*$, $\text{II}^*$, $\text{III}^*$, and $\text{IV}^*$. Also, as in the case of elliptic fibrations, the type of singular fiber that can occur is controlled by the local monodromy. In particular, if the local mondromy is of order $2$, the only possible fiber type is $\text{I}_0^*$. Returning to the canonical bundle formula, in the isotrivial case, the moduli divisor contribution is trivial (compare  \cite[Thm 3.3]{Druel-Bianco}) and thus the only contribution to be determined is the discriminant divisor $D_\pi$, which is supported on the reduced discriminant $\Sigma$, and whose coefficients for the various components $D_i$ of $\Sigma$ depend on the Hwang--Oguiso singular fiber type at the generic point of $D_i$. These coefficients are familiar from the elliptic fibration setting (compare with \Cref{K3appendix}), but to our knowledge the computations in this section (e.g. \Cref{prop:local structure of pi}), refining Hwang--Oguiso's work, are new in the hyper-K\"ahler setting. We also refer the reader to \Cref{SS:K3fib} for a concrete application of these ideas in higher dimensions.\\

Note that most of the results in this section rely on the assumption that \emph{every general singular fiber of $\pi : X \longrightarrow B$ is a non-multiple fiber}, which explains this hypothesis in \Cref{main:genimpliestypeA}. This assumption is satisfied when $\pi$ has a rational section.

\subsection{Compactification of the intermediate base $U$}
Let us start with the following general result about the existence of a preferred $G$-equivariant compactification $Y \supset U$, which we call the \emph{normal $G$-compactification}.

\begin{proposition}
	Let $G$ be a finite group and $U$ a smooth quasi-projective variety with a free $G$-action. Then
	\begin{enumerate}
		\item For any normal projective compactification $B$ of $U/G$, there exists a unique $G$-equivariant normal projective compactification $Y$ of $U$ such that $Y/G \cong B$.
		\item If $B$ is $\QQ$-factorial, then $Y$ is $\QQ$-Gorenstein.
	\end{enumerate}
\end{proposition}
\begin{proof}
	This is standard (e.g. \cite[Thm. 3.8]{Greb-QE}, \cite[Prop. 5.20]{kol-mori}). The compactification $Y$ is obtained as the normalization of $B$ in the function field $K(U)$.
\end{proof}

Recall that the intermediate base $U$ is a Galois covering $U \longrightarrow B_0 \subset B$. The base $B$ of the Lagrangian fibration is a $\QQ$-factorial $\QQ$-Fano klt variety by \cite[Thm 2]{mat99}. Let $Y \supset U$ be the normal $G$-compactification whose quotient is $B$ and consider the quotient map
\[ q : Y \longrightarrow B .\]
We write $\Sigma \subset B$ for the (reduced) discriminant locus of the isotrivial fibration $\pi$, and $\{ D_i \}$ for its irreducible components. By \cite[Prop 3.1]{HO09}, $\Sigma$ is a closed subset of pure codimension $1$ and $D_i$ is a prime divisors on $B$. The Hurwitz formula for the quotient morphism $q : Y \longrightarrow B$ reads 
\begin{equation} \label{eq:Hurwitz}
	K_Y \sim_{\QQ} q^* \big( K_B + D_q \big),
\end{equation}
where $D_q$ is an an effective $\QQ$-divisor on $B$ supported on the branch locus $\Sigma$.

\begin{proposition} \label{prop:compactification of U}
	Let $Y$ be the normal $G$-compactification of the intermediate base $U$.
	\begin{enumerate}
		\item If $\pi$ is of type A, then $Y$ is an abelian torsor. In particular, $U$ compactifies to an abelian torsor.
		\item If $\pi$ is of type B, then $K_Y$ is ample.
	\end{enumerate}
\end{proposition}
\begin{proof}
	First, by \cite[Thm 2]{mat99}, we have $\Pic B \cong \ZZ$. Hence a $\mathbb{Q}$-Cartier divisor on $B$ is either ample, anti-ample, or trivial. From the Hurwitz formula \eqref{eq:Hurwitz} and the fact that $q$ is finite we deduce the second item, namely that the variety $Y$ is of general type if and only if $K_Y$ is ample.\\

Similarly, if $\pi$ is of type A then $K_Y \sim_{\mathbb{Q}} 0$. We have already shown that for type A fibrations, $U$ (or $Y$) admits a birational map
\[ f : Y \dashrightarrow A \]
to an abelian torsor $A$. Since $Y$ is normal, the complement of its smooth locus $Y_{\reg} \subset Y$ has codimension at least $2$. The restriction of $f$ to $Y_{\reg}$ is a morphism $f : Y_{\reg} \longrightarrow A$. As the canonical divisor of $Y_{\reg}$ is trivial, this morphism is unramified and hence is an open immersion. It follows that the map $Y/G\dashrightarrow A/G$ induced by $f$ on the quotient does not contract any divisors.\\
	
	We claim that $f$ is an isomorphism in codimension $1$, namely that the complement of $f(Y_{\reg}) \subset A$ has codimension at least $2$. Assume on the contrary that $A \setminus f(Y_{\reg})$ contains a prime divisor $E$. Then the inverse birational map $f^{-1} : A \dashrightarrow Y$ contracts $E$, and the quotient birational map $A/G \dashrightarrow Y/G$ contracts the effective divisor $(G \cdot E) / G$. On the other hand, we have seen that the map $Y/G \dashrightarrow A/G$ induced by $f$ on the quotient does not contract any divisor. The conclusion is that $\rho(A/G) \ge \rho(Y/G) + 1 = \rho(B) + 1 = 2$, which implies that $H^2 (A, \QQ)^G$ has dimension $\ge 2$. This violates \Cref{cor:G-invariant cohomology for base}, thereby proving that $f$ is an isomorphism in codimension $1$. The conclusion follows from the following lemma. \end{proof}
	
\begin{lemma} \label{lem:no flop on abelian variety}
	Let $Y$ be a normal $\QQ$-Gorenstein projective variety with $K_Y \sim_{\QQ} 0$ and $A$ an abelian torsor. A birational map $f : Y \dashrightarrow A$ extends to an isomorphism.
\end{lemma}
\begin{proof}
Since $K_Y \sim_{\QQ} 0$ and $Y$ is birational to $A$, it has canonical singularities. By \cite{HMcK}, the fibers of a resolution of indeterminacy of $f$ are rationally connected. Since an abelian variety does not contain rational curves $f$ extends to a morphism.
\end{proof}

\begin{remark} In the type B case, we only have partial control on the singularities of the compactification $Y\supset U$. We plan to return to this question in the future in connection to the study of type $B$ isotrivial fibrations.	
\end{remark}

\subsection{Tubular neighborhood of general singular fibers}
In this subsection, we will study the geometry of the general singular fibers of the isotrivial fibration $\pi : X \longrightarrow B$ along with the local monodromy. Our main tools will be the work of \cite{HO09,HO11} and the fact that the monodromy of an isotrivial fibration respects the Hodge structure (\Cref{prop:G-action on cohomology}).\\

We let $D_i \subset \Sigma$ be an irreducible component of the discriminant locus of $\pi$, a prime divisor on $B$. Its preimage $q^{-1}(D_i) \subset Y$ is a (possibly reducible) divisor on $Y$. 

\begin{definition}
	Let $D_i\subset \Sigma$ be an irreducible component of the discriminant locus and $\eta_i \in S$ the generic point of any irreducible component of $q^{-1}(D_i) \subset Y$. The \emph{local monodromy group} around $D_i$ is the stabilizer subgroup
	\[ G_i = \operatorname{Stab}(\eta_i) \ \ \subset \ \ G ,\]
which is well-defined up to conjugation in $G$. The \emph{local monodromy index} around $D_i$ is its order
	\[ m_i = |G_i| .\]
\end{definition}

The following is a well-known equivalent definition of the local monodromy group and index. Fix a general point $b \in D_i$ and consider a small analytic open neighborhood $b \in V \subset B$. Since $b$ is a smooth point of $B$, we may assume that there are biholomorphisms
\[ V \cong \Delta^n ,\qquad V^* := V \setminus D_i \cong \Delta^* \times \Delta^{n-1} ,\]
where $\Delta$ is a $1$-dimensional open disc, $\Delta^* = \Delta \setminus \{ 0 \}$ is a punctured $1$-dimensional open disc, and $b \in V$ corresponds to $(0,\cdots,0) \in \Delta^n$ under the biholomorphism $V \cong \Delta^n$. The fundamental group of $\Delta^* \times \Delta^{n-1}$ is isomorphic to $\ZZ$, so we have homomorphisms of groups
\begin{equation} \label{eq:localmon}
	\ZZ \cong \pi_1 (V^*) \longrightarrow \pi_1 (B_0) \longrightarrow \GL (H^1 (F, \ZZ)),
\end{equation}
where the last is the global monodromy representation with image $G$. The local monodromy group $G_i$, a finite cyclic subgroup of $G$, is the image of this composition \eqref{eq:localmon}. We thus have an isomorphism $G_i \cong \mu_{m_i}$, where $m_i = |G_i|$ is the local monodromy index.\\

Now suppose that the fiber $X_b$ is not multiple. By the classification of \cite[Thm 1.4]{HO09} and \cite[Thm 2.4]{HO11}, it has at least one reduced irreducible component. Therefore, shrinking $V$ if necessary, we may assume that $\pi : X_V \longrightarrow V$ admits a section passing through the reduced component of $X_b$. Using the same argument as in \Cref{lem:rational section}, we can construct a cartesian diagram
\begin{equation}
\begin{tikzcd} \label{diag:pilocal}
	F \times \tilde V^* \arrow[d, "\pr_2"] \arrow[r] & X_{V^*} \arrow[d, "\pi"] \\
	\tilde V^* \arrow[r] & V^*,
\end{tikzcd}
\end{equation}
where $\tilde V^* \longrightarrow V^*$ is the Galois covering corresponding to the local monodromy group $G_i$, and $G_i$ acts on $F \times \tilde V^*$ diagonally and by fixing a point in the abelian torsor $F$. We will regard $F$ as a $G_i$-abelian variety by fixing an origin.\\

A surprising fact about isotrivial Lagrangian fibrations is that the local monodromy group $G_i$ is the automorphism group of an elliptic curve, and in particular has order $2$, $3$, $4$, or $6$. This was tacitly discussed in \cite{HO11}, though in a slightly different context. See also \cite[Lem 3.6]{BS} for an alternative approach.

\begin{proposition} \label{prop: local monodromy}
	Let $G_i$ be the local monodromy group around an irreducible component $D_i$ of the discriminant locus $\Sigma$. There exists a $1$-dimensional $G_i$-abelian subvariety $E_i \subset F$ such that $G_i$ acts faithfully on $E_i$ and trivially on $F/E_i$.
\end{proposition}
\begin{proof}
	We will construct a short exact sequence of $G_i$-abelian varieties
	\begin{equation} \label{eq:ses of local elliptic curve}
		0 \longrightarrow E_i \longrightarrow F \longrightarrow F / E_i \longrightarrow 0 ,
	\end{equation}
	where $\dim E_i = 1$ and the quotient $F / E_i$ has a trivial $G_i$-action. Since $H^1 (F, \ZZ)^{G_i} \subset H^1 (F, \ZZ)$ is a (saturated) Hodge substructure (see \Cref{prop:G-action on cohomology}), it is enough to prove that the Hodge substructure $H^1 (F, \QQ)^{G_i} \subset H^1 (F, \QQ)$ has dimension $2n-2$. Since $F$ is an abelian variety, this is equivalent to showing that $H^0 (F, T_F)^{G_i}$ has dimension $n-1$.\\
	
	Fixing a general point $b \in D_i$, an analytic neighborhood $b \in V \subset B$, and a section of $X_V \longrightarrow V$, we construct diagram \eqref{diag:pilocal} as above. By \cite[\S 2]{HO09}, there exists a $1$-dimensional disc $b \in \Delta \subset V$ transversal to the discriminant divisor $D_i$, such that $X_{\Delta} = \pi^{-1}(\Delta)$ is smooth and we have $n-1$ \emph{pointwise} linearly independent holomorphic vector fields
	\[ \xi_j \in H^0 (X_{\Delta}, \, T_{X_{\Delta}}) ,\qquad j = 1, \cdots, n-1 .\]
The vector fields $\xi_j$ are vertical in the sense that given a smooth fiber $F$ we have $(\xi_j)_{|F} \in H^0 (F, T_F)$. We can pull back these vector fields to vector fields $\tilde \xi_j$ on $F \times \tilde \Delta^*$ by the top horizontal arrow of \eqref{diag:pilocal}. Since the zero section $\{ 0 \} \times \tilde \Delta^*$ is $G_i$-invariant and its quotient recovers the section of $X_{\Delta^*} \longrightarrow \Delta^*$, the vector fields $\tilde{\xi}_j$ are $G_i$-invariant. In particular, restricting $\tilde\xi_j$ to a fiber $F$ of the projection $F \times \tilde \Delta^* \longrightarrow \tilde \Delta^*$, we see that $\tilde\xi_j|_F \in H^0(F,T_F)\cong T_0 F$ is $G_i$-invariant.
\end{proof}

\begin{corollary} \label{cor:local monodromy index}
	Let $m_i$ be the local monodromy index around an irreducible component $D_i$ of the discriminant locus. Then
	\begin{enumerate}
		\item $m_i = 2$, $3$, $4$, or $6$.
		\item If the endomorphism field of $E$ is not $\QQ(\sqrt{-1})$ or $\QQ(\sqrt{-3})$, then $m_i = 2$.
	\end{enumerate}
\end{corollary}
\begin{proof}
	The first item follows from \Cref{prop: local monodromy}, since $G_i$ acts faithfully on the elliptic curve $E_i$. The elliptic curve $E_i$ is an isogeny factor of $F$ and therefore is isogenous to $E$. Hence, if the endomorphism field of $E$ is not $\QQ(\sqrt{-1})$ or $\QQ(\sqrt{-3})$ then $\Aut_0 (E_i) = \{ \pm1 \}$.
\end{proof}

Using the local monodromy group $G_i$ and the associated elliptic curve $E_i \subset F$, we can understand the geometry of a tubular neighborhood around the general singular fiber. Extend the Galois covering $\tilde V^* \longrightarrow V^*$ to a ramified covering
\begin{align*}\label{eq:ramgalois}
	q_i : \tilde V = \tilde \Delta \times \Delta^{n-1}& \longrightarrow V = \Delta \times \Delta^{n-1}\\
	(t, \ s_1, \cdots, s_{n-1}) &\longmapsto (t^{m_i}, \ s_1, \cdots, s_{n-1}) .
\end{align*}
Here $\tilde \Delta$ also denotes the $1$-dimensional disc but we use a different notation to distinguish it from $\Delta$. Complete diagram \eqref{diag:pilocal} to the following
\[\begin{tikzcd}[column sep=normal]
	F \times \tilde V \arrow[d, "\pr_2"] \arrow[r] & \bar X_V \arrow[r, dashed, "\operatorname{bir}"] \arrow[d, "\bar \pi"] & X_V \arrow[dl, "\pi"] \\
	\tilde V \arrow[r] & V ,
\end{tikzcd}\]
where
\[ \bar X_V = (F \times \tilde V)/G_i \]
denotes the quotient of $F \times \tilde V$ by the diagonal $G_i$-action. Note that $G_i = \langle g \rangle$ was cyclic and it acts on $\tilde V = \tilde \Delta \times \Delta^{n-1}$ by $g \cdot (t,s_1,\ldots, s_{ n-1}) = (\zeta t, s_1,\ldots, s_{n-1})$ for a primitive $m_i^{\text{th}}$ root of unity $\zeta$.\\

Consider the short exact sequence \eqref{eq:ses of local elliptic curve} realizing $F \longrightarrow F/E_i$ as an $E_i$-torsor. By the Poincar\'e reducibility theorem, there is an isogeny $A_i \longrightarrow F/E_i$ of abelian varieties of dimension $n-1$ and a cartesian diagram
\begin{equation} \label{diag:local trivialization base change}
	\begin{tikzcd}
		A_i \times E_i \arrow[r] \arrow[d] & A_i \arrow[d] \\
		F \arrow[r] & F/E_i
	\end{tikzcd}
\end{equation}
trivializing the $E_i$-torsor $F \longrightarrow F/E_i$. The abelian variety $A_i$ is isogenous to $E^{n-1}$ (or $E_i^{n-1}$).

\begin{proposition} [Local structure of $\bar \pi$ around the general singular fibers] \label{prop:local structure of bar pi}
	Let $D_i$ be an irreducible component of the discriminant locus $\Sigma$ and $b \in D_i$ a general point. Then
	\begin{enumerate}
		\item There exists an analytic open neighborhood $b \in V \subset B$ biholomorphic to $\Delta^n$ such that $\bar \pi : \bar X_V \longrightarrow V$ factors as
		\[\begin{tikzcd}[column sep=normal]
			\bar \pi : \bar X_V \arrow[r] & F/E_i \times V \arrow[r, "\pr_2"] & V
		\end{tikzcd}.\]
		
		\item There exists a cartesian diagram
		\begin{equation}\begin{tikzcd}[column sep=normal]\label{diag:desired}
			A_i \times \bar S_i \times \Delta^{n-1} \arrow[r, "\id \times \bar f_i \times \id"] \arrow[d] & A_i \times \Delta \times \Delta^{n-1} \arrow[d] \\
			\bar X_V \arrow[r] & F/E_i \times V\;\;,
		\end{tikzcd}\end{equation}
		where $A_i \longrightarrow F/E_i$ is the \'etale covering from \eqref{diag:local trivialization base change} and $\bar f_i : \bar S_i \longrightarrow \Delta$ is an isotrivial elliptic fibration over a one-dimensional disc.
	\end{enumerate}
	In other words, $\bar \pi : \bar X_V \longrightarrow V$ is biholomorphic to a finite \'etale quotient of the product of $\bar f_i : \bar S_i \longrightarrow \Delta$ and $\pr_2 : A_i \times \Delta^{n-1} \longrightarrow \Delta^{n-1}$.
\end{proposition}
\begin{proof}
	Let $V$ be an open analytic neighborhood of $b$ as in the discussion above. Consider the product of diagram \eqref{diag:local trivialization base change} with $\tilde V = \tilde \Delta \times \Delta^{n-1}$ and equip it with a diagonal $G_i$-action. Here $G_i$ acts by deck transformations on $\tilde \Delta$ and trivially on $\Delta^{n-1}$. The $G_i$-quotient of this diagram is \eqref{diag:desired} and the isotrivial elliptic fibration $\bar f_i : \bar S_i \longrightarrow \Delta$ is the $G_i$-quotient of the trivial family $\pr_2 : E_i \times \tilde \Delta \longrightarrow \tilde \Delta$, i.e., $\bar S_i = (E_i \times \tilde \Delta)/G_i$. Note that $\bar S_i$ is a singular surface.\\
	
	It remains to prove that the resulting diagram is cartesian. Since the $G_i$-action on $A_i$ and $\Delta^{n-1}$ are trivial, it is enough to prove that the diagram
	\[\begin{tikzcd}
		(E_i \times \tilde \Delta)/G_i \arrow[r] \arrow[d, hook] & \Delta \arrow[d, hook] \\
		(F \times \tilde \Delta)/G_i \arrow[r] & F/E_i \times \Delta
	\end{tikzcd}\]
	is cartesian. This can be done by an explicit computation. The fiber product of $(F \times \tilde \Delta)/G_i$ and $\Delta$ over $F/E_i \times \Delta$ is
	\[ \Big\{ \big( (f, \tilde t), \ t \big) \ \in (F \times \tilde \Delta)/G_i \times \Delta \ : \ f = 0 \ \in F/E_i, \ \ t = \tilde t^{m_i} \Big\} .\]
	One readily shows that it is isomorphic to $(E_i \times \tilde \Delta)/G_i$, which completes the proof.
\end{proof}

The elliptic fibration $\bar f_i : \bar S_i \longrightarrow \Delta$ was constructed as the $G_i$-quotient of the trivial family $\pr_2 : E_i \times \tilde \Delta \longrightarrow \tilde \Delta$. By classical surface theory, $\bar f_i$ admits a unique relative minimal model $f_i : S_i \longrightarrow \Delta$, where $S_i$ is a smooth surface. The central fiber of $f_i$ belongs to the Kodaira classification of singular fibers and we obtain the following diagram:
\[\begin{tikzcd}[column sep=normal]
	E_i \times \tilde \Delta \arrow[r] \arrow[d, "\pr_2"'] & \bar S_i \arrow[r, dashed, "\operatorname{bir}"] \arrow[d, "\bar f_i"'] & S_i \arrow[ld, "f_i"] \\
	\tilde \Delta \arrow[r, "t \longmapsto t^{m_i}"] & \Delta .
\end{tikzcd}\]
Notice that this description of a minimal isotrivial elliptic fibration is identical to the one we provide in \Cref{K3appendix}. That is, we understand the relation between $\bar S_i$ and $S_i$. We will see that $\bar X_V$ and $X_V$ are related in very much the same manner.
\[\begin{tikzcd}[column sep=tiny]
	\bar S_i \arrow[rd, "\bar f_i"'] \arrow[rr, dashed, "\textup{bir}"] & & S_i \arrow[ld, "f_i"] \\
	& \Delta
\end{tikzcd} \qquad \begin{tikzcd}[column sep=tiny]
	\bar X_V \arrow[rd, "\bar \pi"'] \arrow[rr, dashed, "\textup{bir}"] & & X_V \arrow[ld, "\pi"] \\
	& V
\end{tikzcd}.\]
The following can be understood as a refinement of Hwang--Oguiso's description of general singular fibers in \cite{HO09} for the case of isotrivial fibrations.

\begin{proposition} [Local structure of $\pi$ around the general singular fibers] \label{prop:local structure of pi}
	We use the notation from \Cref{prop:local structure of bar pi}.
	\begin{enumerate}
		\item The original isotrivial fibration $\pi : X_V \longrightarrow V$ factors as
		\[\begin{tikzcd}[column sep=normal]
			\pi : X_V \arrow[r] & F/E_i \times V \arrow[r, "\pr_2"] & V
		\end{tikzcd}.\]
		
		\item There exists a cartesian diagram
		\[\begin{tikzcd}[column sep=normal]
			A_i \times S_i \times \Delta^{n-1} \arrow[r, "\id \times f_i \times \id"] \arrow[d] & A_i \times \Delta \times \Delta^{n-1} \arrow[d] \\
			X_V \arrow[r] & F/E_i \times V
		\end{tikzcd}\]
		where $A_i \longrightarrow F/E_i$ is an \'etale covering in \eqref{diag:local trivialization base change} and $f_i : S_i \longrightarrow \Delta$ is the relative minimal model of $\bar f_i : \bar S_i \longrightarrow \Delta$.
	\end{enumerate}
	In other words, the isotrivial fibration $\pi : X_V \longrightarrow V$ is biholomorphic to a finite \'etale quotient of the product of $f_i : S_i \longrightarrow \Delta$ and $\pr_2 : A_i \times \Delta^{n-1} \longrightarrow \Delta^{n-1}$.
\end{proposition}
\begin{proof}
	Consider the (rational) factorization of $\pi : X_V \longrightarrow V$ as
	\[ \pi : X_V \dashrightarrow \bar X_V \longrightarrow F/E_i \times V \xlongrightarrow{\pr_2} V ,\]
	where the first map is a birational map and the latter morphisms are as in \Cref{prop:local structure of bar pi}. Since $X_V$ is smooth and $F/E_i$ is an abelian torsor, the rational map $X_V \dashrightarrow F/E_i$ is in fact a morphism. Hence, we have a factorization
	\[ \pi : X_V \longrightarrow F/E_i \times V \xlongrightarrow{\pr_2} V, \]
	and moreover $X_V$ and $\bar X_V$ are birational over $F/E_i \times V$. By \Cref{prop:local structure of bar pi}, the \'etale covering of $\bar X_V$ obtained by the base change $A_i \longrightarrow F/E_i$ is isomorphic to $A_i \times \bar S_i \times \Delta^{n-1}$. Let us denote by $\tilde X_V$ a similar \'etale covering of $X_V$. We have a cartesian diagram
	\begin{equation}\label{diag:X_V}\begin{tikzcd}[column sep=tiny]
		A_i \times \bar S_i \times \Delta^{n-1} \arrow[rr, dashed, "\operatorname{bir}"] \arrow[dd] \arrow[dr] & & \tilde X_V \arrow[dd] \arrow[ld] \\
		& A_i \times \Delta \times \Delta^{n-1} \arrow[dd] \\
		\bar X_V \arrow[rr, dashed, "\operatorname{bir}"] \arrow[rd] & & X_V \arrow[ld] \\
		& F/E_i \times V .
	\end{tikzcd}\end{equation}
	Note that both $\tilde X_V$ and $A_i \times S_i \times \Delta^{n-1}$ are smooth, and are relatively minimal elliptic fibrations over $A_i \times \Delta \times \Delta^{n-1}$. They are birational by diagram \eqref{diag:X_V}, so they are isomorphic by the uniqueness of relative minimal models of curves. This completes the proof.
\end{proof}

\begin{corollary}
	Let $b \in D_i$ be a general point of an irreducible component of the discriminant locus, $X_b := \pi^{-1} (b)$ the fiber over $b$, and $S_0$ the central fiber of the minimal elliptic fibration $S_i \longrightarrow \Delta$ from \Cref{prop:local structure of pi}. Then there is a cartesian diagram
	\[\begin{tikzcd}
		S_0 \times A_i \arrow[r] \arrow[d] & A_i \arrow[d] \\
		X_b \arrow[r] & F/E_i .
	\end{tikzcd}\]
	In other words, a general singular fiber $X_b$ is a finite \'etale quotient of $S_0 \times A_i$. \qed
\end{corollary}

\begin{definition}\label{def:kodtype}
	The \emph{Kodaira type} of a general singular $X_b$ of $\pi$ is the Kodaira type of the central fiber $S_0$ of the minimal isotrivial elliptic fibration $S_i \longrightarrow \Delta$ constructed in \Cref{prop:local structure of pi}.
\end{definition}

We emphasize again that \Cref{def:kodtype} is limited to \emph{isotrivial} Lagrangian fibrations \emph{whose general singular fibers are non-multiple}. \Cref{def:kodtype} coincides with the original definition of \cite{HO09}, which uses the notion of a characteristic cycle. However, with this definition at hand, it is clear from the Kodaira classification \cref{table:Kodaira} in \Cref{K3appendix} that general singular fibers can only have type $\text{II}$, $\text{III}$, $\text{IV}$, $\text{I}_0^*$, $\text{II}^*$, $\text{III}^*$, or $\text{IV}^*$.

\begin{proposition}
	A general singular fiber of an isotrivial fibration $\pi : X \longrightarrow B$ can be of Kodaira type $\textup{II}$, $\textup{III}$, $\textup{IV}$, $\textup{I}_0^*$, $\textup{II}^*$, $\textup{III}^*$, or $\textup{IV}^*$.\qed
\end{proposition}

\subsection{Log canonical thresholds and the canonical bundle formula}
In this subsection we will discuss the canonical bundle formula and explain its relation with \Cref{thm:typeAB} (\Cref{prop:equivalent characterization of type A}). In short, $\pi$ is of type $A$ if and only if for every irreducible component $D_i \subset \Sigma$, the fiber of $\pi$ over the generic point of $D_i$ is of $*$-type, namely of type $\text{I}_0^*, \text{II}^*, \text{III}^*,$ or $\text{IV}^*$. When the endomorphism algebra of $E$ is different from $\QQ(\sqrt{-1})$ or $\QQ(\sqrt{-3})$ the local monodromy is $\mu_2$ (\Cref{cor:local monodromy index}) forcing the general singular fibers to be of type $\text{I}_0^*$. \Cref{main:genimpliestypeA} will follow easily.\\

The following $\QQ$-divisor on $B$ governs the behavior of the canonical divisor of $Y$ via the Hurwitz formula $K_Y \sim_\QQ q^* \big( K_B + D_q \big)$  \eqref{eq:Hurwitz}:
\begin{equation} \label{eq:Delta_q}
	D_q = \sum_i \big( 1 - \tfrac{1}{m_i} \big) D_i .
\end{equation}
Let us introduce another $\QQ$-divisor on $B$ supported on the discriminant locus. Following \cite[\S 2.11]{Druel-Bianco}, we define the \emph{discriminant $\QQ$-divisor} $D_{\pi}$ of the morphism $\pi : X \longrightarrow B$ by
\begin{equation}\label{eq-dpi} D_{\pi} = \sum_i \big( 1 - c_i \big) D_i ,\qquad c_i = \operatorname{lct}_{D_i} (X,  X_{D_i}) \ \in \QQ .\end{equation}
Here $c_i$ is the log canonical threshold of the pair $(X_{\eta_i}, X_{D_i,\eta_i})$, where $X_{D_i}:= \pi^{-1} (D_i)$ and $\eta_i$ is the generic point of $D_i$.\\

The importance of these $\QQ$-divisors stems from the canonical bundle formula, a higher-dimensional generalization of Kodaira's canonical bundle formula for elliptic fibrations, which is due to Ambro \cite{Ambro-JDG, Ambro-b}, Kawamata \cite{Kaw-subadjunction} and Fujino \cite{Fujino-can} (see also Koll\'ar \cite[Thm. 8.3.7]{Kollar-can}). In the case of isotrivial Lagrangian fibrations this formula takes the following simple form:

\begin{theorem}[{\cite[Thm 3.3]{Druel-Bianco}}] \label{thm:canonical bundle formula}
	There exists a $\QQ$-linear equivalence of $\QQ$-divisors
	\[ K_X \sim_{\QQ} \pi^* \big( K_B + D_{\pi} \big) .\]
	In other words, $K_B + D_{\pi} \sim_{\QQ} 0$ on $B$.
\end{theorem}

\Cref{prop:local structure of pi} is describing the local structures around a general singular fiber of $\pi$ and allows us to compute the log canonical thresholds $c_i$, and thus $D_{\pi}$.

\begin{proposition} \label{prop:equivalent characterization of type A}
	Assume that all general singular fibers of $\pi$ are non-multiple. Then we have
	\[ D_q \ge D_{\pi} .\]
	Moreover, the following are equivalent.
	\begin{enumerate}
		\item $D_q = D_{\pi}$.
		\item $\pi$ is of type A.1 or A.2.
		\item The general singular fibers of $\pi$ are of Kodaira type $\textup{I}_0^*$, $\textup{II}^*$, $\textup{III}^*$, or $\textup{IV}^*$.
	\end{enumerate}
\end{proposition}
\begin{proof}
	To show the inequality, we need to prove $\frac{1}{m_i} \le c_i$ for all irreducible components $D_i \subset \Sigma$.
	The log canonical threshold of a pair is invariant under \'etale covering, so the \'etale covering $A_i \times S_i \times \Delta^{n-1} \longrightarrow X_V$ in \Cref{prop:local structure of pi} gives us the equality
	\[ c_i = \operatorname{lct}(S_i, S_0) ,\]
	where $f : S_i \longrightarrow \Delta$ is the minimal isotrivial elliptic fibration constructed in the proposition and $S_0 \subset S_i$ its (singular) central fiber. The log canonical thresholds of the singular fibers of a minimal isotrivial elliptic fibration are listed in \Cref{table:Kodaira}. One checks from the table that the inequality $\frac{1}{m_i} \le c_i$ is satisfied for every Kodaira fiber type.\\
	
	In \Cref{table:Kodaira}, the equality $c_i = \frac{1}{m_i}$ holds exactly when the singular fiber $S_0$ in the elliptic fibration is of Kodaira $*$-type. From the \'etale covering $A_i \times S_i \times \Delta^{n-1} \longrightarrow X_V$ in \Cref{prop:local structure of pi}, we obtain the equivalence (1) $\Longleftrightarrow$ (3). Recall that we had $K_B + D_{\pi} \sim_{\QQ} 0$ from the canonical bundle formula in \Cref{thm:canonical bundle formula}. The Hurwitz formula \eqref{eq:Hurwitz} then implies that $K_Y \sim_{\QQ} 0$ if and only if $D_q \sim_{\QQ} D_{\pi}$. Since $D_q \geq D_{\pi}$ this is in fact equivalent to the equality $D_q = D_{\pi}$. This proves (1) $\Longleftrightarrow$ (2).
\end{proof}

\begin{remark}
One can add an additional equivalent condition in \Cref{prop:equivalent characterization of type A}, namely that $\pi$ is of type A if and only if $\bar S_i$ has canonical singularities, or equivalently if and only if $\bar X_V$ has canonical singularities (see \Cref{prop:local structure of bar pi}). In this case $\bar X_V$ has symplectic singularities and $X_V$ is a symplectic resolution of $\bar X_V$.
\end{remark}

\Cref{main:genimpliestypeA} is now an easy consequence of the description of the local monodromy group from \Cref{cor:local monodromy index}.

\begin{proof}[Proof of \Cref{main:genimpliestypeA}]
	If the endomorphism field of $E$ is not $\QQ(\sqrt{-1})$ or $\QQ(\sqrt{-3})$,  \Cref{cor:local monodromy index} states that the local monodromy index $m_i$ is $2$ for all $i$. It follows from \Cref{table:Kodaira} that the general singular fibers of $\pi$ are of type $\textup{I}_0^*$. \Cref{prop:equivalent characterization of type A} implies that $\pi$ is of type $A$.\\
	
	Using the same argument, we obtain that for every irreducible component $D_i\subset \Sigma$ the log canonical threshold $c_i = \operatorname{lct}_{D_i} (X,  X_{D_i})$ is equal to $1/2$. If the base $B$ is isomorphic to $\PP^n$, then \Cref{thm:canonical bundle formula} implies that 
	\[ \deg \Sigma = 2 \deg D_q = -2\deg(K_{\mathbb{P}^n}) = 2(n+1) .\qedhere\]
\end{proof}

\section{Classification of type A fibrations}

In this section we prove \Cref{main:classification}, which describes type A.1 fibrations admiting a rational section. Recall from \Cref{thm:typeAB} that the intermediate trivializiation $p: Z\longrightarrow U$ is equivariantly birational to a morphism between
\begin{itemize}
	\item two $G$-abelian torsors (type A.1), or
	\item a $G$-variety of Kodaira dimension $n$ and a $G$-abelian torsor (type A.2), or
	\item two $G$-varieties of Kodaira dimension $n$ (type B).
\end{itemize}
In particular, if $\pi$ is a type A.1 fibration then $X$ is birational to the quotient of an abelian torsor of dimension $2n$ by the finite group $G$. As an intermediate step we introduce \Cref{prop:classification without conjecture}, which does not rely on the existence of a rational section.\\

\begin{proposition} \label{prop:classification without conjecture}
	Let $\pi : X \longrightarrow B$ be a type A.1 isotrivial fibration of a hyper-K\"ahler $2n$-fold, so that $U$ and $Z$ are equivariantly birational to $G$-abelian torsors $\bar F'$ and $T$. Then
	\begin{enumerate}
		\item $T/G$ is a primitive symplectic variety admitting a symplectic resolution by a hyper-K\"ahler manifold $X'$.
		
		\item There exists a commutative diagram
		\begin{equation}\label{diag}\begin{tikzcd}
			X \arrow[d, "\pi"] \arrow[r, dashed, "\operatorname{bir}"] & X' \arrow[d] \arrow[r, "\pi'"] & \bar F / G = B' ,\\
			B \arrow[r, "\cong"] & \bar F'/G
		\end{tikzcd}
		\end{equation}
		where $\pi'$ is an isotrivial fibration of a hyper-K\"ahler manifold $X'$.
	\end{enumerate}
\end{proposition}

We will use the notion of a primitive symplectic variety following \cite[Def 3.1]{bak-lehn22} (see also \cite{sch20}). A \emph{(projective) primitive symplectic variety} is a normal projective symplectic variety $(Y, \sigma)$ such that $H^0 (Y, \Omega_Y^{[1]}) = 0$ and $H^0 (Y, \Omega_Y^{[2]}) = \CC \cdot \sigma$, where $\Omega_{\bar X}^{[k]}$ is the sheaf of reflexive $k$-forms on $X$.
	
\begin{lemma} \label{lem:quotient of an abelian variety is symplectic}
	Let $G$ be a finite group and $T$ be a $G$-abelian variety of dimension $2n$. Suppose that there is a simple $G$-module $V$ so that
	\[ H^{1,0}(T) \cong V \oplus V^{\vee}. \]
	Then
	\begin{enumerate}
		\item The quotient variety $T/G$ is a symplectic variety.
		\item If we further assume that $V$ is not a symplectic $G$-module, then $T/G$ is a primitive symplectic variety.
	\end{enumerate}
\end{lemma}
\begin{proof}
	The projective variety $\bar X = T/G$ has quotient singularities so it is well behaved from the point of view of Hodge theory. In particular, we have $H^{1,0}(\bar X) = H^0 (\bar X, \Omega_{\bar X}^{[1]})$ and $H^{2,0}(\bar X) = H^0 (\bar X, \Omega_{\bar X}^{[2]})$, as well as the equalities
	\[ H^{1,0} (\bar X) = H^{1,0} (T)^G = \big( V \oplus V^{\vee} \big)^G = 0 \]
	and
	\[ H^{2,0} (\bar X) = H^{2,0}(T)^G = \Big( \wedge^2 V \oplus \wedge^2 V^{\vee} \oplus \big( V \otimes V^{\vee} \big) \Big)^G .\]
	Note that the last term is $(V \otimes V^{\vee})^G = \End_G (V) = \CC \cdot \id$ by Schur's lemma. Hence there is a $G$-invariant holomorphic symplectic form $\sigma$ on $T$. It follows that $\bar X$ is a symplectic variety by \cite[Prop 2.4]{bea00}. The above identity also shows that if $V$ is not a symplectic module then $T/G$ is a primitive symplectic variety since $(\wedge^2 V)^G = 0\implies h^{2,0}(\bar X) = 1$.
\end{proof}

\begin{lemma} \label{lem:birational to HK}
	Any primitive symplectic variety that is birational to a hyper-K\"ahler manifold admits a symplectic resolution by a hyper-K\"ahler manifold.
\end{lemma}
\begin{proof}
	Any primitive symplectic variety $\bar X$ admits a partial resolution $X \longrightarrow \bar X$ by a $\QQ$-factorial, terminal and primitive symplectic variety $X$ by \cite[Prop 11]{sch20}. If we assume $X$ is birational to a hyper-K\"ahler manifold $X'$, then $X$ itself is necessarily a hyper-K\"ahler manifold by \cite[Prop 6.5]{greb-lehn-rol13}.
\end{proof}

\begin{proof} [Proof of \Cref{prop:classification without conjecture}]
	Let $\pi : X \longrightarrow B$ be a type A isotrivial fibration. Then $X$ is birational to a finite quotient of an abelian variety $T$ by the finite group $G$. Combining \Cref{lem:H10 as a G-module}, \ref{lem:isotypic G-component in U}, \ref{lem:quotient of an abelian variety is symplectic}, and \ref{lem:birational to HK}, we obtain that $T/G$ is a primitive symplectic variety admitting a symplectic resolution $X' \longrightarrow T/G$ by a hyper-K\"ahler manifold $X'$. Note that we have two equivariant morphisms $T \longrightarrow \bar F$ and $T \longrightarrow \bar F'$; the former extends the intermediate trivialization $p : Z \longrightarrow U$ and the latter comes from the fibration $p' : Z \longrightarrow \bar F$ from \Cref{thm:diagonalization}. Taking their $G$-quotients, we obtain the commutative diagram \eqref{diag}, but it remains to prove that the birational map $B \dashrightarrow \bar F'/G$ is an isomorphism. This follows from \cite[Cor 2]{mat14}.
\end{proof}

The last ingredient we need to prove \Cref{main:classification} is the following lemma, which is a consequence of a series of recent results of Auffarth, Lucchini Arteche, and Quezada providing a classification of smooth quotients of abelian varieties.

\begin{lemma}[\cite{auf-art20, auf-art-que22, auf-art22}] \label{lem:classification of G-abelian varieties}
	Let $F$ be a $G$-abelian torsor. Assume that (1) $G$ acts on $F$ faithfully and without any translations, (2) $F/G$ is smooth, and (3) $H^* (F/G,\QQ) \cong H^* (\PP^n, \QQ)$. Then there exists a choice of an origin for $F$ making it into a $G$-abelian variety satisfying:
	\begin{enumerate}
		\item $F \cong E^n$ for an elliptic curve $E$.
		\item $G$ acts faithfully on $E^n$ by fixing the origin and this action belongs to the following list:
		\begin{enumerate}
			\item $\mathfrak S_{n+1}$ acting on $\ker (E^{n+1} \longrightarrow E)$ by permuting the factors.
			\item $\mu^{\times n} \rtimes \mathfrak S_n$ acting on $E^n$ in the obvious way, where $\mu = \mu_2, \mu_3, \mu_4$ or $\mu_6$ acts on $E$ faithfully by fixing the origin.
			\item \textnormal{(Only when $n = 2$)} The Pauli group acting on $E^2$ for $E = \CC / \ZZ [\sqrt{-1}]$.
		\end{enumerate}
		\item $F/G \cong \PP^n$.
	\end{enumerate}
\end{lemma}

The \emph{Pauli group} is the following non-abelian group of order $16$:
\[ \left\{ \begin{psmallmatrix} \sqrt{-1}^a & 0 \\ 0 & \sqrt{-1}^b \end{psmallmatrix} : a \equiv b \mod 2 \right\} \rtimes \left\langle\begin{psmallmatrix} 0 & 1 \\ 1 & 0 \end{psmallmatrix}\right\rangle \quad \subset \ \GL_2 (\ZZ[\sqrt{-1}]) .\]
We refer the reader to \cite[\S 3.2]{auf-art-que22} for the Pauli group action on $(E')^2$ appearing in the lemma.

\begin{proof}
	Consider a normal subgroup of $N \subset G$ and its quotient
	\[ N = \{ g \in G : F^g \neq \emptyset \}, \qquad H = G / N .\]
	By \cite[Thm 1.1]{auf-art22}, there exists at least one point of $F$ fixed by all of $N$, i.e., $F^N \neq \emptyset$. We may thus choose an origin for $F$ so that $N$ acts on $F$ by fixing the origin, i.e., so that $F$ is an $N$-abelian variety. Moreover $F/N$ is a smooth projective fiber bundle over an abelian variety $Z$ whose fiber is positive dimensional and of the form $\PP^{n_1} \times \cdots \times \PP^{n_i}$. Moreover, $H$ acts freely on both $F/N$ and $Z$ so that the fiber bundle is $H$-equivariant. Taking the $H$-quotient, we obtain a fiber bundle $F/G \to Z/H$ whose fiber is of the form $\PP^{n_1} \times \cdots \times \PP^{n_i}$.\\
	
	The cohomology $H^* (F/G, \QQ)$ can be computed in terms of the cohomology of the base $H^* (Z/H, \QQ)$ by the Leray--Hirsch theorem:
	\[ H^* (F/G, \QQ) \cong H^* (Z/H, \QQ) \otimes H^* (\PP^{n_1} \times \cdots \times \PP^{n_i}, \QQ) .\]
	By assumption $H^* (F/G, \QQ) \cong H^* (\PP^n, \QQ)$, so this forces the base $Z/H$ to be a point and $F/G \cong \PP^n$. Since $Z$ is an abelian variety and $H$ is a finite group acting freely on $Z$, this means that $Z$ is a point and $H$ is trivial. That is, $F$ is a $G$-abelian variety. The lemma now follows from the main results of \cite{auf-art20, auf-art-que22} which provide a classification of $G$-abelian varieties $F$ with smooth quotient $F/G$.
\end{proof}

We are now ready to conclude the proof of \Cref{main:classification}.

\begin{proof}[Proof of \Cref{main:classification}]
	We take up where we left off in \Cref{prop:classification without conjecture} and use the same notation. Since $\pi$ has a rational section, we can set $\bar F = F$ in \Cref{thm:diagonalization}. Let us write $F' = \bar F'$ accordingly. The $G$-action on $F$ has a fixed point by \Cref{lem:rational section}, so we may consider $F$ as a $G$-abelian variety. Consider the two isotrivial fibrations
	\[\begin{tikzcd}
		X' \arrow[d, "\pi"] \arrow[r, "\pi'"] & B' = F/G \\
		B = F'/G,&
	\end{tikzcd}\]
	where $X'$ is a hyper-K\"ahler resolution of $T/G$ for a $G$-abelian torsor $T$. Again by \Cref{lem:rational section}, we have an isomorphism of $G$-abelian torsors $T \cong F \times F'$, where the $G$-action on the right is diagonal.\\
	
	Let us first prove that $B = F'/G$ is a smooth variety. Let $y \in F'$ be an arbitrary point and $G_y \subset G$ the associated stabilizer subgroup. Since the $G$-action on $F$ fixes the origin of $F$, the point $(0, y) \in F \times F'$ has stabilizer subgroup $G_y$. The group $G_y$ acts around $(0, y)$ linearly according to the $G$-module structure on $V \oplus V^{\vee}$. Since the quotient $(F \times F')/G$ is symplectically resolvable at the image of $(0, y)$, \cite[Thm 1.1]{ver00} shows that the stabilizer $G_y$ acts around $y \in F'$ by a complex reflection group. In other words, $F'/G$ is smooth at the image of $y \in F'$.\\
	
	Now $B = F'/G$ is smooth and has cohomology $H^*(B, \QQ) = H^*(\PP^n, \QQ)$ by \cite[Thm 0.4]{shen-yin22}. \Cref{lem:classification of G-abelian varieties} shows that $F' \cong (E')^n$ admits a $G$-action corresponding to one of three possibilities. In particular, it fixes the origin, i.e., $F'$ is a $G$-abelian variety. Applying the same stabilizer group method to a point $x \in F$, we conclude that $B' = F/G$ is smooth and that $F \cong E^n$ admits a $G$-action similarly.\\
	
	The case $G = \mathfrak S_{n+1}$ corresponds to $\Kum_n$-fibrations and the case $G = \mu^{\times n} \rtimes \mathfrak S_n$ corresponds to $\text{K3}^{[n]}$-fibrations, as we have explicitly described in \S \ref{S:ex-trivial}. When $n = 2$, there is one additional possibility for which $G$ is the Pauli group. We will show that the corresponding quotient does not admit a symplectic resolution, thereby contradicting \Cref{prop:classification without conjecture}. Since $G$ acts on $T \cong E^2 \times (E')^2$ by fixing the origin, around $0$ the variety $T/G$ is analytically isomorphic to $(V \oplus V^{\vee})/G$. It is shown in \cite{bel09} that such a singularity admits a symplectic resolution if and only if (1) $G$ is of the form $\mu_m^{\times n} \rtimes \mathfrak S_n$ or (2) $G$ is the group $G_4$, a group of order $24$ appearing in the Shephard--Todd classification of complex reflection groups. It follows that $T/G$ cannot admit a symplectic resolution around $0$ if $G$ is the Pauli group.
\end{proof}

\section{Additional remarks}\label{S-remarks}
In this closing section we collect some remarks about the locus of hyper-K\"ahler manifolds admitting an isotrivial Lagrangian fibration. Our main result is \Cref{thm:isos-are-rare}, (restated here as \Cref{thm:restated}), which states that an isotrivial Lagrangian fibration of a hyper-K\"ahler variety $X$ can be deformed into a Lagrangian fibration with maximal variation provided that the second Betti number of $X$ is at least $7$. Let us first observe that isotriviality of a Lagrangian fibration is a closed condition in moduli.

\begin{proposition} \label{prop:isotriviality is a closed condition}
	Let $\mathcal X \xlongrightarrow{\pi} \mathcal B \longrightarrow M$ be a family of Lagrangian fibered hyper-K\"ahler $2n$-folds over a complex manifold $M$. Then the following is a Zariski closed subset of $M$:
	\[ \textup{Iso}(M) = \{ t \in M : \pi : X_t \longrightarrow \mathcal{B}_t \mbox{ is isotrivial} \}.\]
\end{proposition}
\begin{proof}
	For simplicity we assume that $\pi$ is projective; the general case when $\pi$ is proper can be dealt with by using the notion of relative automorphism scheme and torsor. Let $\mathcal{B}_0 \subset \mathcal B$ be a Zariski open subset over which $\pi$ is smooth and proper and pick a large positive integer $N$. We can then find a finite \'etale cover $\tilde{\mathcal{B}}_0$ of $\mathcal{B}_0$ such that the base change $\tilde \pi: \mathcal{X}_{\tilde{ \mathcal{B}}_0}\longrightarrow \tilde B_0$ has a section and such that there is a compatible choice of level $N$ structure on the fibers of $\tilde \pi$. We thus obtain a moduli map $\mu : \tilde{\mathcal{B}}_0\longrightarrow \mathcal{A}$, where $\mathcal{A}$ is a fine moduli space of abelian $n$-folds with level $N$ structure and given polarization type.\\
	
	Write $f: \widetilde{\mathcal{B}}_0\longrightarrow M$ for the composition of the \'etale cover with the map $\mathcal{B}_0\longrightarrow M$. Since fibers of the map $(\mu,f): \tilde{\mathcal{B}}_0\longrightarrow \mathcal{A}\times M$ have dimension at most $n$, upper semicontinuity of fiber dimension implies that the union of $n$-dimensional fibers is a Zariski closed subset $W_0 \subset \tilde{\mathcal{B}}_0$. Moreover, a point $t\in \tilde{\mathcal{B}}_0$ is in $W_0$ if and only if $\mu$ is constant on $f^{-1}(f(t))\subset \tilde{\mathcal{B}}_0$, so that $W_0$ is the $f$-preimage of $\textup{Iso}(M) \subset M$. Since $f$ is flat and $W_0$ is closed in $\tilde{\mathcal B}_0$, $\textup{Iso}(M)$ is closed in $M$.
\end{proof}

\begin{theorem}[= \Cref{thm:isos-are-rare}]\label{thm:restated}
	Let $\pi : X \longrightarrow B$ be a Lagrangian fibration of a hyper-K\"ahler manifold with $b_2(X) \ge 7$. There exists a deformation of the Lagrangian fibration $\pi$ which is not isotrivial.
\end{theorem}

\begin{proof}
	There exists a local universal deformation family $\Def(\pi)$ of the Lagrangian fibration $\pi$ which is an analytic open ball of dimension $b_2(X) - 3$ \cite{mat16}. Assume that every Lagrangian fibration over $\Def(\pi)$ is isotrivial. Taking an appropriate Noether--Lefschetz locus of $\Def(\pi)$ parametrizing Lagrangian fibrations $\pi' : X' \longrightarrow B'$ with $X'$ projective, we can obtain a $3$-parameter family of isotrivial Lagrangian fibrations whose generic Mumford--Tate group is isomorphic to $\SO(T, q)$, where $T$ is the generic transcendental lattice and $\dim_{\QQ} T = 5$. By \cite{Zarhin}, this implies that the endomorphism field $\End_{\HS}(T)$ is $\QQ$, which violates the following lemma.
\end{proof}

\begin{lemma} \label{lem:totally real transcendental HS}
	Let $\pi: X\longrightarrow B$ be an isotrivial Lagrangian fibration of a hyper-K\"ahler manifold and let $T \subset H^2 (X, \QQ)$ be the transcendental Hodge structure. If the endomorphism field of $T$ is $\QQ$, then $\dim_{\QQ} T = 3$ or $4$.
\end{lemma}
\begin{proof}
	Recall from the proof of \Cref{prop:simple Albanese morphism} that there is a generically finite $G$-equivariant morphism $U\longrightarrow A$ to a $G$-abelian torsor $A$ such that the holomorphic symplectic form $\sigma\in H^{2,0}(Z)$ is contained in the image of $H^1 (F, \QQ) \otimes H^1 (A, \QQ) \longrightarrow H^2(Z,\QQ)$. This yields an inclusion $T \subset \big( H^1 (F, \QQ) \otimes H^1 (A, \QQ) \big)^G$. Since the smooth fiber $F$ of $\pi$ is isogenous to $E^n$ and $A$ is isogenous to the power of a simple abelian torsor, there is an inclusion
	\[ T \subset H^1 (E, \QQ) \otimes W, \]
where $W$ is a simple polarizable Hodge structure of weight $1$. The following lemma appeared in a special case in \cite[Prop 12]{vgee-voi16} and more generally in \cite[Thm 7.2, Rmk 7.3]{moonen18}. It implies that $H^1 (E, \QQ)$ must be isomorphic to the \emph{simple Kuga--Satake partner} $\operatorname{KS}(T)$ of $T$ defined in \cite[\S 7.1]{moonen18}:
	
	\begin{lemma}[{\cite[Prop 12]{vgee-voi16} and \cite[Thm 7.2]{moonen18}}] \label{lem:moonen}
		Let $W_1,W_2,$ and $T$ be simple polarizable Hodge structures of weights $1$, $1$, and $2$. Suppose that 
		\begin{enumerate}
			\item $T$ is of K3 type,
			\item there exists an inclusion of Hodge structures $T \subset W_1 \otimes W_2$, and
			\item $\End_{\HS}(T) = \QQ$.
		\end{enumerate}
		Then $W_1 \cong W_2 \cong \operatorname{KS}(T)$.
	\end{lemma}
	
	The simple Kuga--Satake partner $\operatorname{KS}(T)$ is a simple polarizable Hodge structure of weight $1$. Its behavior depends on the parity of $\dim_{\QQ} T$ because the (half) spin representations of type B and D simple algebraic groups behave differently. If $\dim_{\QQ} T = 2r+1$ is odd then $\dim_{\QQ} \operatorname{KS}(T) = 2^r$ or $2^{r+1}$. By the lemma, we have an isomorphism $H^1 (E, \QQ) \cong \operatorname{KS}(T)$ so $\dim \operatorname{KS}(T) = 2$. This implies that $r = 1$ and $\dim T = 3$. On the other hand, if $\dim T = 2r$ is even then $\dim \operatorname{KS}(T) = 2^{r-1}$ or $2^r$, and we deduce that $r = 1$ or $2$. Note that $\dim T$ cannot be $2$, because any $2$-dimensional weight $2$ Hodge structure of K3 type is necessarily CM. Hence $\dim T = 4$.
\end{proof}

\begin{lemma}
Let $E$ be an elliptic isogeny factor of $F$ and suppose that $K = \End_{\QQ}(E)$ is a CM field. Then the transcendental Hodge structure $H^2_{\tr}(X)$ has weak CM by $K$.\end{lemma}
\begin{proof}
	Since $K\cong \End_{\HS}(H^1 (E, \QQ))$ and $F$ is isogenous to $E^n$, we have an isomorphism $\End_{\HS}(H^1(F,\QQ)) \cong \Mat_n(K)$. Moreover, \Cref{prop:G-action on cohomology} gives inclusions
	\[ G \hookrightarrow \Aut_0 (F) \cong \Aut_{\HS} (H^1 (F,\ZZ)) \ \subset \ \End_{\HS} (H^1 (F, \QQ)) \cong \Mat_n(K) .\]
	Since $K$ is the center of $\Mat_n(K)$, the $G$-module and $K$-vector space structures on $H^1 (F, \QQ)$ are compatible, namely $H^1 (F, \QQ)$ is a $G$-representation over the field $K$.\\
	
	Recall the inclusion $H^2_{\tr}(X) \subset \big( H^1 (F, \QQ) \otimes H^1 (A, \QQ) \big)^G$ from \Cref{lem:totally real transcendental HS}. We divide the proof into two cases according to whether $\dim A=n$ or $\dim A>n$. If $\dim A=n$, by \Cref{cor:G-invariant cohomology for base} and the fact that $A/G$ is $\QQ$-Fano, $A$ is isogenous to $(E')^n$ for an elliptic curve $E'$. Hence we get
	\[ H^2_{\tr}(X) \subset H^1(E, \QQ) \otimes H^1(E',\QQ),\]
and in fact
\[ H^2_{\tr}(X) \cong \left(H^1(E, \QQ) \otimes H^1(E',\QQ)\right)_{\text{tr}} .\]
Since the right hand side is a $K$-vector space, so is the left hand side.\\
	
The second case is when $\dim A > n$. In this case, we claim that
	\begin{equation}\label{eq:simpleH2} H^2_{\tr}(X) \cong \big( H^1 (F, \QQ) \otimes H^1 (A, \QQ) \big)^G .\end{equation}
This equality endows $H^2_{\tr}(X)$ with a $K$-vector space structure. To show the equality, notice first that $\big( H^{1,0}(F) \otimes H^{1,0}(A) \big)^G$ is $1$-dimensional (\Cref{lem:symplectic form generates G-invariant}), so the right hand side is a weight $2$ Hodge structure of K3 type. It cannot have any Hodge classes since
	\[ \Big( \big( H^1 (F, \QQ) \otimes H^1 (A, \QQ) \big)^G\Big)^{\MT} = \big( H^1 (F, \QQ) \otimes H^1 (A, \QQ) \big)^{G \times \MT} = 0 ,\]
where the last equality follows from the fact that $H^1 (F, \QQ)$ and $H^1 (A, \QQ)$ are simple $(G, \MT)$-bimodule (\Cref{prop:simple bimodule}) which are not dual. Indeed, they are of different dimensions since $\dim A > n$. Therefore, the right hand side of \eqref{eq:simpleH2} is a simple Hodge structure and the equality follows.
\end{proof}

\begin{proposition}
	Let $\mathcal X \xlongrightarrow{\pi} \mathcal B \longrightarrow M$ be a family of isotrivial Lagrangian fibration of hyper-K\"ahler $2n$-folds over a Zariski closed subset $M \subset \Def(\pi)$. Assume that $\pi$ has a rational section and that the elliptic curve factor $E$ of its general fiber does not have CM by $\QQ(\sqrt{-1})$ or $\QQ(\sqrt{-3})$. Then $M$ has dimension at most $2$.
\end{proposition}
\begin{proof}
	Under these hypotheses, the isotrivial fibration $\pi_t : X_t \longrightarrow \mathcal{B}_t$ over a general point $t \in M$ is of type A.1 by \Cref{main:genimpliestypeA}. Note that the hypothesis that the general singular fibers of $\pi_t$ are non-multiple is satisfied because $\pi_t$ has a rational section. It follows that $X_t$ is birational to the quotient of a $2n$-dimensional abelian variety $T$ by a finite group $G$. Moreover, $T$ is isogenous to $E^n \times (E')^n$ for elliptic curves $E$ and $E'$. The result follows from the fact that such a $T$ can only vary in a two-parameter family.
\end{proof}

\appendix

\section{Isotrivial K3 surfaces} \label{K3appendix}

In this appendix we classify isotrivial elliptic K3 surfaces with a section. We make no claim to originality -- in fact, much of what we have to say already appears in \cite[\S 1.4.2]{FM94} and \cite{moonen18}. However, since we could not find \Cref{thm:isotrivial K3} in the literature and our main goal is to generalize this result to hyper-K\"ahler manifolds, we feel this appendix is useful and necessary.\\

Let $\pi : S \longrightarrow B \cong \PP^1$ be an isotrivial elliptic fibration of a projective K3 surface, $B_0\subset B$ the complement of the discriminant locus, and $S_0:= S\times_{B} B_0$ the complement of the singular fibers of $\pi$. Write $\pi_0: S_0\longrightarrow B_0$ for the restriction of $\pi$ to $S_0$. The fibers of $\pi_0$ are isomorphic to a fixed curve $F$ of genus one. Let $G$ be the image of the monodromy representation $\pi_1(B_0) \longrightarrow \GL(H^1 (F, \ZZ))$ corresponding to the local system $R^1(\pi_0)_* \underline \ZZ$, and consider the minimal Galois cover $U \longrightarrow B_0$ trivializing this local system. We consider the base change diagram
\[\begin{tikzcd}[sep=normal]
	Z \arrow[r] \arrow[d,swap,"p"] & S_0 \arrow[d, "\pi"] \\
	U \arrow[r] & B_0 .
\end{tikzcd}\]
As in \Cref{prop:G-action on cohomology}, the monodromy representation $\pi_1(B_0) \longrightarrow \GL(H^1 (F, \ZZ))$ factors through the subgroup $\Aut_0 F$, so $G$ is forced to be isomorphic to one of four cyclic groups as follows:
\begin{equation} \label{Gcases}
	G=\begin{cases} \mu_2, \mu_3, \text{ or } \mu_6 &\text{if } j(F)=0,\\
	\mu_2\text{ or } \mu_4 &\text{if } j(F)= 1728,\\
	\mu_2 & \text{otherwise.}\end{cases}
\end{equation}

\begin{theorem}[= \Cref{main:K3}] \label{thm:isotrivial K3}
	Let $\pi : S \longrightarrow \PP^1$ be an isotrivial elliptic fibration of a projective K3 surface.
	\begin{enumerate}
		\item If $G = \mu_2$, there is an abelian surface $A$ and a fibration $p: A\longrightarrow E$ to an elliptic curve $E$ with fiber $F$, such that $\pi$ is birational to the quotient of $p$ by the involution $[-1]$. In particular, $S$ is a Kummer surface.
		\item  If $\pi$ admits a section, there is a curve $C$ with a $G$-action such that $\pi$ is birational to the quotient of the  projection $F \times C \longrightarrow C$ by a diagonal $G$-action. The possible curves $C$ for $G = \mu_3, \mu_4$, and $\mu_6$ are classified in tables \ref{table:classification for mu_3}, \ref{table:classification for mu_4}, and \ref{table:classification for mu_6} respectively.
	\end{enumerate}
\end{theorem}

The remainder of this appendix will be devoted to a few preliminaries and the proof of \Cref{thm:isotrivial K3}.\\

Note that the smooth open curve $U$ can be uniquely compactified to a smooth projective curve $C$, and the $G$-action on $U$ uniquely extends to $C$. The quotient $C/G$ is normal and birational to the base $B \cong \PP^1$ of the elliptic fibration, so it is isomorphic to $B$. Let $\tilde Z\supset Z$ be a projective compactification to which the $G$-action extends. Then $\pi : S \longrightarrow \PP^1$ is birational to the map
	\[ \tilde Z / G \longrightarrow C/G = \PP^1, \]
and can be obtained as its (unique) \emph{relative minimal model}.\\	

We will make repeated use of the following table which describes Euler characteristic contributions, monodromy, and local models for the singular fibers arising in the Kodaira classification. Let us explain the meaning of the last row. We can realize each Kodaira singular fiber as the central fiber of the relative minimal model of a family
\[ (F \times \tilde \Delta)/G \longrightarrow \Delta ,\]
where $\Delta$ is an open disc, $\tilde \Delta = \{ (z,t) \in \CC \times \Delta : z^m = t^d \}$ is a cyclic branched covering of $\Delta$ of degree $m$, and $G$ acts diagonally. When $d > 1$, $\tilde \Delta$ is singular and it may be replaced by its normalization.

\begin{table}[h]
	\begin{tabular}{|c||c|c|c|c|c||c|c|c|c|c|}
		\hline
		\rule[-1ex]{0pt}{2.5ex} Kodaira type & $\text{I}_0$ & $\text{I}_{n \ge 1}$ & $\text{II}$ & $\text{III}$ & $\text{IV}$ & $\text{I}_0^*$ & $\text{I}_{n \ge 1}^*$ & $\text{II}^*$ & $\text{III}^*$ & $\text{IV}^*$ \\\hline
		\rule[-1ex]{0pt}{2.5ex} Euler characteristic & $0$ & $n$ & $2$ & $3$ & $4$ & $6$ & $n+6$ & $10$ & $9$ & $8$ \\\hline
		\rule[-1ex]{0pt}{2.5ex} Monodromy order $m$ & $1$ & $\infty$ & $6$ & $4$ & $3$ & $2$ & $\infty$ & $6$ & $4$ & $3$ \\\hline
		\rule[-1ex]{0pt}{2.5ex} $d$ for $z^m = t^d$ & $1$ & - & $1$ & $1$ & $1$ & $1$ & - & $5$ & $3$ & $2$ \\\hline
		\rule[-1ex]{0pt}{2.5ex} Log canonical threshold $c$ & $1$ & $1$ & $\frac{5}{6}$ & $\frac{3}{4}$ & $\frac{2}{3}$ & $\frac{1}{2}$ & - & $\frac{1}{6}$ & $\frac{1}{4}$ & $\frac{1}{3}$ \\\hline
	\end{tabular}
	\caption{Euler characteristic contributions, monodromy, local models, and log canonical thresholds for singular fibers in the Kodaira classification.}
	\label{table:Kodaira}
\end{table}

We note that the quotient $(F \times \tilde \Delta)/G$ has canonical singularities in the cases $\text{I}_0^*$, $\text{II}^*$, $\text{III}^*$, and $\text{IV}^*$, and its crepant resolution is then the relative minimal model for $(F \times \tilde \Delta)/G\longrightarrow \tilde \Delta/G$. The quotient for the other cases is only klt, and its relative minimal model should be constructed by a sequence of blowups and blowdowns.
	
\begin{proof}[Proof of \Cref{thm:isotrivial K3} (1)]
	Consider an isotrivial elliptic fibration $\pi : S \longrightarrow B$ for which $G=\mu_2$. We will show that there is an elliptic curve $E$, an abelian surface $A$ with a fibration $f: A\longrightarrow E$ and fiber $F$, and a commutative diagram as follows, where the top horizontal arrow is a minimal resolution of singularities:
	\[
	\begin{tikzcd}[sep=normal]
	S\ar[d,"\pi"] \ar[r] & A/[-1]\ar[d,"f/{[-1]}"]\\
	\mathbb{P}^1\ar[r,"\cong"] & E/[-1]\;.
	\end{tikzcd}
	\]
	 In particular, $S$ is the Kummer surface associated to the abelian surface $A$. Note that $\pi$ admits a section if and only if $f$ does.\\

	\begin{lemma}
		$C$ is a genus $1$ curve $E$ and the quotient $E \longrightarrow \PP^1$ is the quotient by a hyperelliptic involution.
	\end{lemma}
	\begin{proof}
	Since the quotient of $C$ by $G = \mu_2$ is rational, $G$ acts by multiplication by $(-1)$ on $H^{1,0}(C)$. Similarly, since $q(S)=0$ the local system $R^1(\pi_0)_*\underline{\mathbb{Z}}$ does not contain a non-zero $G$-invariant local system. It follows that $G$ acts by multiplication by $(-1)$ on $H^{1,0}(F)$. Accordingly, one computes that 
		$$1=p_g(S)=h^{1,0}(C)\cdot h^{1,0}(E)=h^{1,0}(C)=g(C).$$
		Any involution on a genus $1$ curve is hyperelliptic.
	\end{proof}
	
	According to \Cref{table:Kodaira}, the Kodaira fiber $\text{I}_0^*$ is the only possible singular fiber with local monodromy of order $2$. This means that all the singular fibers of $\pi$ are of type $\text{I}_0^*$. Since each $\text{I}_0^*$ fiber contributes $6$ to the Euler characteristic of $S$, there must be precisely four singular fibers. (This method was already used in \cite{sawon14}.)\\
		
	The four copies of $\PP^1$ with multiplicity one in a fiber of type $\text{I}_0^*$ are $(-2)$-curves. Contracting these sixteen $(-2)$-curves gives a singular symplectic surface $\bar S$ and an isotrivial elliptic fibration $\bar \pi : \bar S \longrightarrow \PP^1$. Each singular fiber is $\PP^1$ with multiplicity $2$, and $\bar S$ is locally the quotient of the trivial family $F\times \Delta\longrightarrow \Delta$ by the involution $(x,t)\mapsto (\iota(x),-t)$, where $\iota$ is a hyperelliptic involution. We can thus locally and equivariantly compactify $Z\longrightarrow U$ into a trivial family. Gluing these partial compactifications gives a smooth surface $\tilde{Z}$ and an isotrivial elliptic fibration $p : \tilde{Z} \longrightarrow E$ . By the canonical bundle formula, $\kappa(\tilde{Z}) = 0$. Finally, $q(\tilde{Z})\neq 0$ implies that $\tilde{Z}$ is an abelian surface.
\end{proof}

In the rest of this appendix we assume that the elliptic fibration $\pi : S \longrightarrow B \cong \PP^1$ admits a section and describe the curve $C$ obtained as a smooth compactification of $U$ for all possibilities of the group $G$.\\

We will make repeated use of the fact that such a section gives a trivialization of the elliptic torsor $Z\longrightarrow U$ for which the $G$-action is diagonal (\Cref{lem:rational section}). Indeed, if $\pi: S\longrightarrow \mathbb{P}^1$ admits a section, then $p: Z\longrightarrow U$ is an isotrivial elliptic fibration with an equivariant section $\sigma$ and trivial monodromy. Since the moduli space of elliptic curves with an appropriate level structure is a fine moduli space, there is a commutative diagram
\[\begin{tikzcd}[row sep=normal]
Z \arrow[rr, "\cong"] \arrow[dr, "p"'] && F \times U \arrow[dl, "\pr_2"] \\
& U&,
\end{tikzcd}\]
such that the $G$-action transported to $F\times U$ is diagonal.\\

\noindent \ref*{K3appendix}.1. \textbf{Classification for $G = \mu_3$.} As shown above, $p : Z \longrightarrow U$ is isomorphic to $\pr_2 : E \times U \to U$ with a diagonal $G$-action on $E \times U$, and the $G$-action on $E$ fixes the origin of $E$. Recall that $C$ is the unique smooth projective compactification of $U$. The $G$-action uniquely extends to $C$ and hence to $F \times C$. As a result, $\pi : S \longrightarrow \PP^1$ is the relative minimal model of
\[ (F \times C)/G \longrightarrow C/G = \PP^1 .\]

It remains to classify all possible curves $C$ which can arise. Each singular fibers of $\pi$ has local monodromy of order $3$. By \Cref{table:Kodaira}, all the singular fibers are of type IV or $\text{IV}^*$. Type IV and $\text{IV}^*$ fibers contribute $4$ and $8$ respectively to the Euler characteristic of $S$. Thus if there are $f_{\text{IV}}$ type IV fibers and $f_{\text{IV}^*}$ type $\text{IV}^*$ fibers we must have
\[ 4 f_{\text{IV}} + 8 f_{\text{IV}^*} = 24. \]
Hence there are four possibilities for $(f_{\text{IV}}, f_{\text{IV}^*})$ as in \Cref{table:classification for mu_3}. We note that these examples already appeared in \cite{del-mos} and \cite{moonen18}. The last column refers to the index in Table~A1 of \cite{moonen18}.

\begin{table}[h]
	\begin{tabular}{|c||c|c||c||c|c|}
		\hline
		\rule[-1ex]{0pt}{2.5ex} No.	& $\text{IV}$ & $\text{IV}^*$ & $g(C)$ & Ramification degree & No. in {\cite{moonen18}} \\\hline
		\rule[-1ex]{0pt}{2.5ex} 1 & 6 & 0 & 4 & $1^6$ & 115 \\\hline
		\rule[-1ex]{0pt}{2.5ex} 2 & 4 & 1 & 3 & $1^4, 2$ & 84 \\\hline
		\rule[-1ex]{0pt}{2.5ex} 3 & 2 & 2 & 2 & $1^2, 2^2$ & 1 \\\hline
		\rule[-1ex]{0pt}{2.5ex} 4 & 0 & 3 & 1 & $2^3$ & - \\\hline
	\end{tabular}
	\caption{Classification of possible singular fibers configurations and the genus of $C$ when $G = \mu_3$.}
	\label{table:classification for mu_3}
\end{table}

Type IV and $\text{IV}^*$ fibers can be described as the central fibers of the relative minimal models of $(F \times \tilde \Delta)/\mu_3$, where $\tilde \Delta$ is a ramified degree $3$ covering $\tilde \Delta \longrightarrow \Delta$ given by $z^3 = t$ and $z^3 = t^2$ respectively. From this we obtain a global description of $C$. For example, let us describe the second item of \Cref{table:classification for mu_3}. The curve $C$ is the normalization of the degree $3$ cyclic covering of $\PP^1$ given by
\[ z^3 = (x-a_1) (x-a_2) (x-a_3) (x-a_4) (x-b_1)^2 \]
for five distinct points $a_1, \cdots, a_4, b_1 \in \PP^1$. For such a curve $C$, we construct $(F \times C)/\mu_3 \longrightarrow \PP^1$ as above. Taking the relative minimal model modifies the four singular fibers over $a_1, \cdots, a_4 \in \PP^1$ through a sequence of blowups and blowdowns, resulting in fibers of Kodaira type IV. The singular fiber over $b_1 \in \PP^1$ is different; we only need blowups to obtain a fiber of type $\text{IV}^*$. The resulting surface is the minimal elliptic surface over $\PP^1$ with four type IV singular fibers and one type $\text{IV}^*$ singular fiber, which is an elliptic K3 surface.\\

\noindent \ref*{K3appendix}.2. \textbf{Classification for $G = \mu_4$.} In the case $G=\mu_4$, the local monodromy around singular fibers is of order $2$ or $4$. By \Cref{table:Kodaira}, singular fibers with local monodromy of order $2$ are of type $\text{I}_0^*$ and singular fibers with local monodromy of order $4$ are of type III or $\text{III}^*$. If the number of singular fibers of each type is $f_{\text{III}}$, $f_{\text{I}_0^*}$, and $f_{\text{III}^*}$, then
\[ 3f_{\text{III}} + 6f_{\text{I}_0^*} + 9f_{\text{III}^*} = 24 .\]
There are $10$ possible triples $(f_{\text{III}}, f_{\text{I}_0^*}, f_{\text{III}^*})$ satisfying this condition, which are listed in \Cref{table:classification for mu_4}. Note that item 9 gives a curve $C$ which is a disjoint union of two elliptic curves obtained in the $G = \mu_2$ case, so it can safely be ignored.

\begin{table}[h]
	\begin{tabular}{|c||c|c|c||c||c|c|}
		\hline
		\rule[-1ex]{0pt}{2.5ex} No.	& $\text{III}$ & $\text{I}_0^*$ & $\text{III}^*$ & $g(C)$ & Ramification degree & No. in {\cite{moonen18}} \\\hline\hline
		\rule[-1ex]{0pt}{2.5ex} 1 & 8 & 0 & 0 & 9 & $1^8$ & 138 \\\hline
		\rule[-1ex]{0pt}{2.5ex} 2 & 6 & 1 & 0 & 7 & $1^6, 2$ & 130 \\\hline
		\rule[-1ex]{0pt}{2.5ex} 3 & 5 & 0 & 1 & 5 & $1^5, 3$ & 116 \\\hline
		\rule[-1ex]{0pt}{2.5ex} 4 & 4 & 2 & 0 & 5 & $1^4, 2^2$ & 117 \\\hline
		\rule[-1ex]{0pt}{2.5ex} 5 & 3 & 1 & 1 & 5 & $1^3, 2, 3$ & 85 \\\hline
		\rule[-1ex]{0pt}{2.5ex} 6 & 2 & 3 & 0 & 3 & $1^2, 2^3$ & 86 \\\hline
		\rule[-1ex]{0pt}{2.5ex} 7 & 2 & 0 & 2 & 3 & $1^2, 3^2$ & 2 \\\hline
		\rule[-1ex]{0pt}{2.5ex} 8 & 1 & 2 & 1 & 5 & $1, 2^2, 3$ & 3 \\\hline
		\rule[-1ex]{0pt}{2.5ex} 9 & 0 & 4 & 0 & 1 ($\times 2$) & $2^4$ & - \\\hline
		\rule[-1ex]{0pt}{2.5ex} 10& 0 & 1 & 2 & 1 & $2, 3^2$ & - \\\hline
	\end{tabular}
	
	\caption{Classification of possible singular fiber configurations and the genus of $C$ when $G = \mu_4$.}
	\label{table:classification for mu_4}
\end{table}

\noindent \ref*{K3appendix}.3. \textbf{Classification for $G = \mu_6$.}
The order of the local monodromy around singular fibers can be $2$, $3$, or $6$. By \Cref{table:Kodaira}, type $\text{I}_0^*$ fibers have local monodromy of order $2$, while type IV and $\text{IV}^*$ fibers have local monodromy of order $3$, and type II and $\text{II}^*$ fibers have local monodromy of order $6$. Again the Euler characteristic computation gives the following constraint on the numbers $f_{\text{II}}, f_{\text{IV}}, f_{\text{I}_0^*}, f_{\text{IV}^*}, f_{\text{II}^*}$ of fibers of each type:
\[ 2f_{\text{II}} + 4f_{\text{IV}} + 6f_{\text{I}_0^*} + 8f_{\text{IV}^*} + 10f_{\text{II}^*} = 24 .\]
The classification for all such tuples $(f_{\text{II}}, \cdots, f_{\text{II}^*})$ is provided in \Cref{table:classification for mu_6}. There are 47 cases, with the items 38, 39, 42, 45, and 47 corresponding to disjoint unions of examples which arose for the Galois groups $\mu_2$ and $\mu_3$. In this sense, item 43 is the only connected example with $g(C)>1$ that is not covered in Moonen's list \cite[Table A1]{moonen18}.

\begin{table}[h]
	\begin{tabular}{|c||c|c|c|c|c||c||c|c|}
		\hline
		\rule[-1ex]{0pt}{2.5ex} No. & $\text{II}$ & $\text{IV}$ & $\text{I}_0^*$ & $\text{IV}^*$ & $\text{II}^*$ & $g(C)$ & Ramification degree & No. in {\cite{moonen18}} \\\hline\hline
		\rule[-1ex]{0pt}{2.5ex} 1 & 12& 0 & 0 & 0 & 0 & 25& $1^{12}$ & 150 \\\hline
		\rule[-1ex]{0pt}{2.5ex} 2 & 10& 1 & 0 & 0 & 0 & 22& $1^{10}, 2$ & 149 \\\hline
		\rule[-1ex]{0pt}{2.5ex} 3 & 9 & 0 & 1 & 0 & 0 & 19& $1^9, 3$ & 147 \\\hline
		\rule[-1ex]{0pt}{2.5ex} 4 & 8 & 2 & 0 & 0 & 0 & 19& $1^8, 2^2$ & 148 \\\hline
		\rule[-1ex]{0pt}{2.5ex} 5 & 8 & 0 & 0 & 1 & 0 & 17& $1^8, 4$ & 144 \\\hline
		\rule[-1ex]{0pt}{2.5ex} 6 & 7 & 1 & 1 & 0 & 0 & 16& $1^7, 2, 3$ & 145 \\\hline
		\rule[-1ex]{0pt}{2.5ex} 7 & 7 & 0 & 0 & 0 & 1 & 15& $1^7, 5$ & 139 \\\hline
		\rule[-1ex]{0pt}{2.5ex} 8 & 6 & 3 & 0 & 0 & 0 & 16& $1^6, 2^3$ & 146 \\\hline
		\rule[-1ex]{0pt}{2.5ex} 9 & 6 & 1 & 0 & 1 & 0 & 14& $1^6, 2, 4$ & 140 \\\hline
		\rule[-1ex]{0pt}{2.5ex} 10& 6 & 0 & 2 & 0 & 0 & 13& $1^6, 3^2$ & 142 \\\hline
		\rule[-1ex]{0pt}{2.5ex} 11& 5 & 2 & 1 & 0 & 0 & 13& $1^5, 2^2, 3$ & 141 \\\hline
		\rule[-1ex]{0pt}{2.5ex} 12& 5 & 1 & 0 & 0 & 1 & 12& $1^5, 2, 5$ & 131 \\\hline
		\rule[-1ex]{0pt}{2.5ex} 13& 5 & 0 & 1 & 1 & 0 & 11& $1^5, 3, 4$ & 132 \\\hline
		\rule[-1ex]{0pt}{2.5ex} 14& 4 & 4 & 0 & 0 & 0 & 13& $1^4, 2^4$ & 143 \\\hline
		\rule[-1ex]{0pt}{2.5ex} 15& 4 & 2 & 0 & 1 & 0 & 11& $1^4, 2^2, 4$ & 133 \\\hline
		\rule[-1ex]{0pt}{2.5ex} 16& 4 & 1 & 2 & 0 & 0 & 10& $1^4, 2, 3^2$ & 134 \\\hline
		\rule[-1ex]{0pt}{2.5ex} 17& 4 & 0 & 1 & 0 & 1 & 9 & $1^4, 3, 5$ & 118 \\\hline
		\rule[-1ex]{0pt}{2.5ex} 18& 4 & 0 & 0 & 2 & 0 & 9 & $1^4, 4^2$ & 119 \\\hline
		\rule[-1ex]{0pt}{2.5ex} 19& 3 & 3 & 1 & 0 & 0 & 10& $1^3, 2^3, 3$ & 135 \\\hline
		\rule[-1ex]{0pt}{2.5ex} 20& 3 & 2 & 0 & 0 & 1 & 9 & $1^3, 2^2, 5$ & 120 \\\hline
		\rule[-1ex]{0pt}{2.5ex} 21& 3 & 1 & 1 & 1 & 0 & 8 & $1^3, 2, 3, 4$ & 121 \\\hline
		\rule[-1ex]{0pt}{2.5ex} 22& 3 & 0 & 3 & 0 & 0 & 7 & $1^3, 3^3$ & 123 \\\hline
		\rule[-1ex]{0pt}{2.5ex} 23& 3 & 0 & 0 & 1 & 1 & 7 & $1^3, 4, 5$ & 88 \\\hline
		\rule[-1ex]{0pt}{2.5ex} 24& 2 & 5 & 0 & 0 & 0 & 10& $1^2, 2^5$ & 136 \\\hline
		\rule[-1ex]{0pt}{2.5ex} 25& 2 & 3 & 0 & 1 & 0 & 8 & $1^2, 2^3, 4$ & 124 \\\hline
		\rule[-1ex]{0pt}{2.5ex} 26& 2 & 2 & 2 & 0 & 0 & 7 & $1^2, 2^2, 3^2$ & 125 \\\hline
		\rule[-1ex]{0pt}{2.5ex} 27& 2 & 1 & 1 & 0 & 1 & 6 & $1^2, 2, 3, 5$ & 89 \\\hline
		\rule[-1ex]{0pt}{2.5ex} 28& 2 & 1 & 0 & 2 & 0 & 6 & $1^2, 2, 4^2$ & 90 \\\hline
		\rule[-1ex]{0pt}{2.5ex} 29& 2 & 0 & 2 & 1 & 0 & 5 & $1^2, 3^2, 4$ & 92 \\\hline
		\rule[-1ex]{0pt}{2.5ex} 30& 2 & 0 & 0 & 0 & 2 & 5 & $1^2, 5^2$ & 6 \\\hline
		\rule[-1ex]{0pt}{2.5ex} 31& 1 & 4 & 1 & 0 & 0 & 7 & $1, 2^4, 3$ & 126 \\\hline
		\rule[-1ex]{0pt}{2.5ex} 32& 1 & 3 & 0 & 0 & 1 & 6 & $1, 2^3, 5$ & 93 \\\hline
		\rule[-1ex]{0pt}{2.5ex} 33& 1 & 2 & 1 & 1 & 0 & 5 & $1, 2^2, 3, 4$ & 94 \\\hline
		\rule[-1ex]{0pt}{2.5ex} 34& 1 & 1 & 3 & 0 & 0 & 4 & $1, 2, 3^3$ & 96 \\\hline
		\rule[-1ex]{0pt}{2.5ex} 35& 1 & 1 & 0 & 1 & 1 & 4 & $1, 2, 4, 5$ & 7 \\\hline
		\rule[-1ex]{0pt}{2.5ex} 36& 1 & 0 & 1 & 2 & 0 & 3 & $1, 3, 4^2$ & 11 \\\hline
		\rule[-1ex]{0pt}{2.5ex} 37& 1 & 0 & 2 & 0 & 1 & 3 & $1, 3^2, 5$ & 9 \\\hline
		\rule[-1ex]{0pt}{2.5ex} 38& 0 & 6 & 0 & 0 & 0 & 4 ($\times 2$) & $2^6$ & 115 ($\times 2$) \\\hline
		\rule[-1ex]{0pt}{2.5ex} 39& 0 & 4 & 0 & 1 & 0 & 3 ($\times 2$) & $2^4, 4$ & 84 ($\times 2$) \\\hline
		\rule[-1ex]{0pt}{2.5ex} 40& 0 & 3 & 2 & 0 & 0 & 4 & $2^3, 3^2$ & 97 \\\hline
		\rule[-1ex]{0pt}{2.5ex} 41& 0 & 2 & 1 & 0 & 1 & 3 & $2^2, 3, 5$ & 12 \\\hline
		\rule[-1ex]{0pt}{2.5ex} 42& 0 & 2 & 0 & 2 & 0 & 2 ($\times 2$) & $2^2, 4^2$ & 1 ($\times 2$) \\\hline
		\rule[-1ex]{0pt}{2.5ex} 43& 0 & 1 & 0 & 0 & 2 & 2 & $2, 5^2$ & - \\\hline
		\rule[-1ex]{0pt}{2.5ex} 44& 0 & 1 & 2 & 1 & 0 & 2 & $2, 3^2, 4$ & 14 \\\hline
		\rule[-1ex]{0pt}{2.5ex} 45& 0 & 0 & 4 & 0 & 0 & 1 ($\times 3$) & $3^4$ & - \\\hline
		\rule[-1ex]{0pt}{2.5ex} 46& 0 & 0 & 1 & 1 & 1 & 1 & $3, 4, 5$ & - \\\hline
		\rule[-1ex]{0pt}{2.5ex} 47& 0 & 0 & 0 & 3 & 0 & 1 ($\times 2$) & $4^3$ & - \\\hline
	\end{tabular}
	
	\caption{Classification of the possible singular fibers configurations and the genus of $C$ when $G = \mu_6$.}
	\label{table:classification for mu_6}
\end{table}

\section{Background on the Albanese morphism and fibrations} \label{sec:background}
\noindent \ref*{sec:background}.1. \textbf{Albanese torsors and morphisms.}
The aim of this section is to clarify the notions of Albanese varieties and morphisms used in this article. 
\begin{theorem-definition}
	We can associate to a smooth quasi-projective variety $Y$ a triple $(\Alb_Y, \Alb_Y^{\tor}, \textup{alb}_Y)$ consisting of:
	\begin{enumerate}
		\item an abelian variety $\Alb_Y$ called the Albanese variety of $Y$,
		\item an $\Alb_Y$-torsor $\Alb_Y^{\tor}$ called the Albanese torsor of $Y$, and
		\item a morphism $\textup{alb}_Y: Y \longrightarrow \Alb_Y^{\tor}$ called the Albanese morphism of $Y$.
	\end{enumerate}
	These objects satisfy the following universal property: for any abelian variety $A$, an $A$-torsor $A^{\tor}$, and a morphism $Y \longrightarrow A^{\tor}$, there exists a unique homomorphism of abelian varieties $\Alb_Y \longrightarrow A$ and equivariant morphism $\Alb_Y^{\tor} \longrightarrow A^{\tor}$ giving a commutative diagram
	\[\begin{tikzcd} [row sep=tiny]
		& \Alb_Y^{\tor} \arrow[dd] & \Alb_Y \arrow[dd] \\
		Y \arrow[ru] \arrow[rd] \\
		& A^{\tor} & A .
	\end{tikzcd}\]
\end{theorem-definition}

\begin{remark}
	Typically, in the literature, the Albanese variety of a smooth quasi-projective variety $Y$ is a semi-abelian variety (e.g., \cite{wit08}). Strictly speaking, we work with the Albanese variety of a smooth compactification of $Y$, which is a birational invariant of smooth quasi-projective varieties.
\end{remark}

Given a smooth quasi-projective variety $Y$, there is a mixed Hodge structure on $H^1(Y, \QQ)$ with weights $\geq 1$. We can obtain a free $\mathbb{Z}$-module by letting
\[ W_1 H^1 (Y, \ZZ) := H^1 (Y, \ZZ) \cap W_1 H^1 (Y, \QQ) .\]
For any smooth projective compactification $j : Y \hookrightarrow \tilde Y$, the restriction map $j^* : H^1 (\tilde Y, \ZZ) \longrightarrow H^1 (Y, \ZZ)$ is injective with image $W_1 H^1 (Y, \ZZ)$.

\begin{proposition}
	Let $Y$ be a smooth quasi-projective variety. Then there exist isomorphisms of Hodge structures (induced by the Albanese morphism)
	\[ H^1 (\Alb_Y, \ZZ) \cong H^1 (\Alb_Y^{\tor}, \ZZ) \cong W_1 H^1 (Y, \ZZ) .\]
	Moreover, if $Y$ and $Y'$ are birational then their Albanese varieties, torsors, and morphisms are isomorphic. \qed
\end{proposition}

\begin{example}
	Let $C$ be a smooth projective curve. Then its Albanese variety is $\Alb_C = \Pic^0_C$, its Albanese torsor is $\Alb_C^{\tor} = \Pic^1_C$, and its Albanese morphism is
	\begin{align*} C& \longrightarrow \;\; \Pic^1_C\\
	 x &\longmapsto \big[ \mathcal O_C(x) \big] .\end{align*}
\end{example}

\begin{proposition} \label{prop:structure theorem for abelian morphism}
	Let $X$ and $Y$ be smooth quasi-projective varieties, $A$ an abelian variety and $f : X \times Y \longrightarrow A$ a morphism. Then there are morphisms $g : X \longrightarrow A$ and $h : Y \longrightarrow A$ such that $f(x,y) = g(x) + h(y)$.
\end{proposition}
\begin{proof}
	The map $f$ factorizes as $X \times Y \longrightarrow \Alb_X^{\tor} \times \Alb_Y^{\tor} \longrightarrow A$, where the latter map is a morphism between abelian torsors, and therefore is the sum of the two factors $\Alb_X^{\tor} \longrightarrow A$ and $\Alb_Y^{\tor} \longrightarrow A$.
\end{proof}

\begin{proposition}
	Let $Y$ be a smooth quasi-projective variety with a $G$-action. Then
	\begin{enumerate}
		\item There are $G$-actions on both $\Alb_Y$ and $\Alb_Y^{\tor}$.
		\item The Albanese morphism $Y \to \Alb_Y^{\tor}$ is $G$-equivariant.
	\end{enumerate}
\end{proposition}
\begin{proof}
	The existence of the $G$-action on both $\Alb_Y$ and $\Alb_Y^{\tor}$ comes from the universal property of the Albanese morphism. For any $g \in G$, there exists a unique commutative diagram
	\[\begin{tikzcd}
		Y \arrow[d, "g"] \arrow[r] & \Alb_Y^{\tor} \arrow[d, "g"] & \Alb_Y \arrow[d, "g"] \\
		Y \arrow[r] & \Alb_Y^{\tor} & \Alb_Y .
	\end{tikzcd}\]
	The $G$-actions on $A^{\tor} = \Alb_Y^{\tor}$ and $A = \Alb_Y$ are related by the following morphism
	\[ G \longrightarrow \Aut(A^{\tor}) \cong A \rtimes \Aut_0(A) \twoheadrightarrow \Aut_0(A) .\qedhere\]
\end{proof}

\begin{remark}
	Though $\Alb_Y$ and $\Alb_Y^{\tor}$ are isomorphic as varieties, they may not be isomorphic at $G$-varieties  because some elements of $G$ may act on $\Alb_Y^{\tor}$ by translation. For example, this is the case if $Y$ is an abelian variety and $G$ acts on $Y$ by translations. The Albanese morphism is then the ($G$-equivariant) identity map. However, the $G$-action on the Albanese variety $\Alb_Y$ is trivial.
\end{remark}

\noindent \ref*{sec:background}.2. \textbf{Rational fibrations.}
A proper morphism $f : Y \longrightarrow Z$ between irreducible normal varieties is called a fibration if it is surjective and has connected fibers. This is equivalent to the equality of coherent sheaves $\mathcal O_Z = f_* \mathcal O_Y$. One can generalize this notion to rational maps as follows.
	
\begin{definition}[{\cite[Ex 2.1.12]{laz:pos1}}]
	A rational map $f : Y \dashrightarrow Z$ between irreducible normal varieties is called a \emph{fibration} (or \emph{rational fibration} to emphasize the fact that $f$ is a rational map) if it is dominant and the subfield $f^*K(Z)\subset K(Y)$ is algebraically closed.
\end{definition}

\begin{example}
	If $f : Y \longrightarrow Z$ is a proper morphism, then this definition coincides with the usual one by \cite[Ex 2.1.12]{laz:pos1}.
\end{example}

\begin{example}
	If a fibration $f : Y \longrightarrow Z$ is not proper, then not all fibers of $f$ need to be connected. For example, given a proper fibration $\bar f : \bar Y \longrightarrow Z$ with a reducible fiber $F = F_1 \cup F_2$ over $z \in Z$, consider the composition of the inclusion $Y = \bar Y \setminus (F_1 \cap F_2)\subset \bar Y$ with $\bar f$.\end{example}

\section{Equivariant versions of theorems of Ueno and Kawamata} \label{sec:equivariant Ueno-Kawamata}
This section is devoted to generalizing results of Ueno \cite[Thm 10.9]{ueno} \cite[Thm 3.7]{mori87} and Kawamata \cite[Thm 13]{kaw81} to the equivariant setting.

\begin{theorem} [Equivariant version of Ueno's theorem] \label{thm:equivariant Ueno}
	Let $G$ be a finite group and $X \subset A^{\tor}$ a closed irreducible $G$-subvariety of a $G$-abelian torsor $A^{\tor}$. Then there exists a $G$-abelian subvariety $B \subset A$, a closed $G$-subvariety $Y \subset A^{\tor}/B$, and a $G$-equivariant commutative diagram
	\[\begin{tikzcd}
		X \arrow[r, hook] \arrow[d] & A^{\tor} \arrow[d] \\
		Y \arrow[r, hook] & A^{\tor}/B
	\end{tikzcd}\]
	with the following properties:
	\begin{enumerate}
		\item The diagram is cartesian, i.e., $X \longrightarrow Y$ is a $B$-torsor.
		\item $X \longrightarrow Y$ is the Iitaka fibration of (the normalization of) $X$.
	\end{enumerate}
\end{theorem}
\begin{proof}
	Ueno's theorem is a characterization of the Iitaka fibration of subvarieties of abelian varieties. We show that the proof works $G$-equivariantly. Ueno defines the abelian subvariety $B \subset A$ as the identity component of 
	\[ B' = \{ a \in A : a + X \subset X \} ,\]
	and sets $Y := X / B$. It thus suffices to show that $B$ is a $G$-abelian subvariety of $A$. This follows from the fact that $B'\subset A$ is an algebraic subgroup closed under the $G$-action, and that the $G$-action fixes $0\in A$, thus taking the identity component of $B'$ to itself.
\end{proof}
	
\begin{theorem} [Equivariant version of Kawamata's theorem] \label{thm:equivariant Kawamata}
Consider a finite group $G$ and a finite $G$-equivariant morphism $f : X \longrightarrow A^{\tor}$  from a normal projective irreducible $G$-variety $X$ to a $G$-abelian torsor $A^{\tor}$. Then there exists a $G$-abelian subvariety $B \subset A$, a finite $G$-equivariant morphism $h : Y \longrightarrow A^{\tor}/B$ from a normal $G$-variety $Y$, and a $\,G$-equivariant commutative diagram
	\[\begin{tikzcd}
		X \arrow[r, "f"] \arrow[d] & A^{\tor} \arrow[d] \\
		Y \arrow[r, "h"] & A^{\tor}/B
	\end{tikzcd}\]
	with the following properties.
	\begin{enumerate}
		\item $X \longrightarrow Y$ is generically $\tilde B$-isotrivial for an abelian variety $\tilde B$ isogenous to $B$.
		\item $X \longrightarrow Y$ is the Iitaka fibration of $X$.
	\end{enumerate}
\end{theorem}
\begin{proof}
	The proof is essentially the same as Kawamata's original proof from \cite[Thm 13]{kaw81}. We provide details for the sake of completeness. Start with an Iitaka fibration $\phi : X \dashrightarrow Y'$ of $X$. Since the Iitaka fibration is defined up to birational equivalence, we may assume that $Y'$ is smooth projective. By realizing the Iitaka fibration as a stable image of the pluricanonical maps, $Y'$ admits a rational $G$-action and $\phi$ is $G$-equivariant.\\
	
	Let $F_y = \overline {\phi^{-1}(y)}$ be the Zariski closure of a very general fiber of $\phi$. It is an irreducible subvariery of $X$ with $\kappa(F_y) = 0$. The image $f(F_y) \subset A^{\tor}$ has Kodaira dimension $0$ since it is a subvariety of $A^{\tor}$ and is dominated by $F_y$. By Ueno's theorem, it is isomorphic to an abelian torsor $B_y^{\tor} \subset A^{\tor}$ for some abelian subvariety $B_y \subset A$. Since the set of abelian subvarieties of $A$ is discrete, $B:=B_y$ does not depend on the choice of a very general point $y\in Y'$.\\
	
	We claim that the abelian subvariety $B \subset A$ is a $G$-abelian subvariety. For $y\in Y'$ very general, the automorphism $g : X \longrightarrow X$ induces an isomorphism $g : B_y^{\tor} \to B_{g\cdot y}^{\tor}$. Since $B_y^{\tor}, B_{g\cdot y}^{\tor} \subset A^{\tor}$ are translations of one another, they are canonically isomorphic up to translation. Ignoring the translation part, we obtain an isomorphism $g : B \longrightarrow B$. This realizes $B$ as a $G$-abelian subvariety of $A$.\\
	
	Consider the quotient $p : A^{\tor} \longrightarrow A^{\tor}/B$. A very general fiber of $\phi$ is sent to a point via
	$$p \circ f : X \longrightarrow A^{\tor} \longrightarrow A^{\tor}/B.$$ By the rigidity lemma \cite[Lem 14]{kaw81}, this induces an equivariant rational map $h' : Y' \dashrightarrow A^{\tor}/B$ making the following diagram commutative:
	\[\begin{tikzcd}
		X \arrow[r, "f"] \arrow[d, "\phi", dashed] & A^{\tor} \arrow[d] \\
		Y' \arrow[r,dashed, "h'"] & A^{\tor}/B.
	\end{tikzcd}\]
	
	As $Y'$ is smooth and $A^{\tor}/B$ is an abelian torsor, the rational map $h'$ is defined everywhere. Denote by $\bar X$ and $\bar Y$ the image of $f$ and $h'$:
	\[\begin{tikzcd}
		X \arrow[r, "f"] \arrow[d, dashed, "\phi"] & \bar X \arrow[r, hook] \arrow[d] & A^{\tor} \arrow[d, "p"] \\
		Y' \arrow[r, "h'"] & \bar Y \arrow[r, hook] & A^{\tor}/B .
	\end{tikzcd}\]
	By construction $\bar X \longrightarrow \bar Y$ is an equivariant and surjective morphism. We claim that it is in fact a $B$-torsor, i.e., that $\bar X = p^{-1}(\bar Y)$. Since both $\bar X$ and $p^{-1}(\bar Y)$ are projective and irreducible, the claim is equivalent to the dimension count $\dim \bar X = \dim \bar Y + \dim B$. This follows from a sequence of inequalities
	\[ \dim \bar Y + \dim B \le \dim Y' + \dim B = \dim X = \dim \bar X \le \dim \bar Y + \dim B .\]
	This shows that $\bar X \to \bar Y$ is a $B$-torsor and that $\dim \bar Y = \dim Y' = \kappa(X)$.\\
	
	Finally, consider the Stein factorization $X\longrightarrow Y$ of the composition $X \longrightarrow A^{\tor} \longrightarrow A^{\tor}/B$:
	\[\begin{tikzcd}
		X \arrow[r, "f"] \arrow[d] & \bar X \arrow[r, hook] \arrow[d] & A^{\tor} \arrow[d, "p"] \\
		Y \arrow[r, "h"] & \bar Y \arrow[r, hook] & A^{\tor}/B .
	\end{tikzcd}\]
	Since $\bar X \longrightarrow \bar Y$ is a $B$-torsor, the Stein factorization $X \longrightarrow Y$ is generically isotrivial with general fiber $\tilde B$, a finite \'etale cover of $B$, hence an abelian variety. It follows that a very general fiber of $X \longrightarrow Y$ has Kodaira dimension $0$ and $\dim Y = \kappa(X)$. These two properties characterize the Iitaka fibration (see \cite[Def-Thm 1.11]{mori87}), so $X \longrightarrow Y$ is again an Iitaka fibration of $X$. The diagram is $G$-equivariant.
\end{proof}

	\bibliographystyle{amsalpha}
	\bibliography{Bibliography.bib}
\end{document}